\documentclass[11 pt]{amsart}
\usepackage{amssymb}
\usepackage{amsmath}
\usepackage{amsthm}
\usepackage{amsfonts}
\usepackage{amssymb,amssymb}
\usepackage{graphicx}
\usepackage{color}
\usepackage{epsfig}

\DeclareGraphicsExtensions{.pdf,.jpg,.png,.ps,.eps,.gif}

\theoremstyle{plain}
\newtheorem{theorem}{Theorem}\numberwithin{theorem}{section}
{}
\newtheorem{main}{Main~Theorem}{}
\newtheorem{lemma}{Lemma}\numberwithin{lemma}{section}
\newtheorem{proposition}{Proposition}\numberwithin{proposition}{section}
\newtheorem{corollary}{Corollary}\numberwithin{corollary}{section}
\theoremstyle{definition}
\newtheorem{definition}{Definition}\numberwithin{definition}{section}
\theoremstyle{remark} \theoremstyle{exam} \theoremstyle{ob}
\newtheorem{remark}{Remark}\numberwithin{remark}{section}
\newtheorem{question}{Question}\numberwithin{question}{section}
\newtheorem{exam}{Example}\numberwithin{exam}{section}
\newtheorem{ob}{Observation}\numberwithin{ob}{section}
\numberwithin{nt}{section}
\numberwithin{equation}{section}

\begin{document}
\title[Sublevel Set Estimates over Global Domains]
{Sublevel Set Estimates over Global Domains}

\author
{Joonil Kim}

\address{Department of Mathematics \\
       Yonsei University \\
       Seoul 120-729, Korea}
\email{jikim7030@yonsei.ac.kr}


 
 \begin{abstract} 
Since Varchenko's seminal paper,    the asymptotics of oscillatory integrals and related problems have been elucidated through the Newton polyhedra associated with the phase $P$.   The supports of those integrals  are concentrated on   sufficiently small neighborhoods. 
 The aim of this paper is to investigate the estimates of  sub-level-sets and oscillatory integrals whose supports are     global domains $D$.  A basic model of $D$ is $ \mathbb{R}^d$.   For this purpose, we define the Newton polyhedra associated with  $(P,D)$ and establish    analogues of Varchenko's theorem in   global domains $D$, under   non-degeneracy  conditions  of        $P$.   
 \end{abstract}
 \maketitle

\setcounter{tocdepth}{1}

\tableofcontents

\section{Introduction}
Asymptotic estimates of  the sub-level-sets   $\{x\in D:|\lambda P(x)|\le 1\} $ and the oscillatory integrals  $\int e^{i\lambda P(x)}\psi (x)dx$ with $\psi\in C^\infty(D)$  arise in many areas of mathematics.   
  For   systematic estimates for  global   $D$,  we consider the regions $D$ defined as:
\begin{align*} 
D_B:=\{x\in\mathbb{R}^d:   |x^{\mathfrak{b}}|\le 1\ \text{for all $\mathfrak{b}\in B$}\}, \ \text{given   finite subsets $B\subset \mathbb{Q}^d$}.
\end{align*}
 Here   $x^{\mathfrak{b}}:=x^{b_1}_1\cdots x^{b_d}_d$   for $\mathfrak{b}=(b_1,\cdots,b_d)$ and $x_\nu^0:=1$ for  every real $x_\nu $.  For instance,  let $\{{\bf e}_\nu\}_{\nu=1}^d$ be the set of the standard unit vectors in $\mathbb{R}^d$, then  $$D_{\{{\bf 0}\}}=\mathbb{R}^d,\ D_{\{{\bf e}_1,\cdots,{\bf e}_d\}}=[-1,1]^d\ \text{and}\ D_{\{-{\bf e}_1,\cdots,-{\bf e}_d\}}=\bigcap_{\nu=1}^d\{x\in \mathbb{R}^d:|x_\nu|\ge 1\}.$$  As a phase function, we shall take a    real valued polynomial     $P(x)$ in $\mathbb{R}^d$:
  \begin{align*} 
  P(x)=\sum_{\mathfrak{m}\in\Lambda(P)}c_{\mathfrak{m}}x^{\mathfrak{m}}\ \text{where    $\Lambda(P)=\{\mathfrak{m}\in\mathbb{Z}^d_+:c_{\mathfrak{m}}\ne 0\}$. }
  \end{align*}
\subsection{Main Questions}
Choose a  model  polynomial $P(x)=x_1x_2$ on the two regions  $D_{\{{\bf e}_1,{\bf e}_2\}}=[-1,1]^2$ and $D_{\{ {\bf 0}\}}=\mathbb{R}^2$. Then one can compute
 $|\{x\in D_{\{{\bf e}_1,{\bf e}_2\}} :|\lambda x_1x_2|\le1 \}|=4 \lambda^{-1} (1+|\log \lambda|)$ 
and $|\{x\in D_{\{ {\bf 0}\}}:|\lambda  x_1x_2|\le1 \}|=\infty$. This simple computation leads  us to study   the following  questions regarding the sublevel-set estimate in this paper:
\begin{question}\label{q11}
Find a condition of $(P,D_B)$  that determines whether  a sub-level-set measure $  \left|\{x\in D_B :|\lambda P(x)|\le 1\} \right|$ converges or diverges.
\end{question}
\begin{question}\label{q1}
When it converges, under a minimal non-degeneracy type condition of  $(P,D_B)$,  determine the   indices  $ \rho$ and   $a$ satisfying  
the sublevel set estimate $ \left|\{x\in D_B :|\lambda P(x)|\le 1\} \right|\approx\lambda^{-\rho} 
 (|\log  \lambda |+1)^a, $ according as $\lambda\in [1,\infty)$ or $\lambda\in (0,1)$.
\end{question}
For the corresponding oscillatory integrals, we are not asking their asymptotics with a fixed individual amplitude function $\psi_{D_B}$, but we are concerned with   convergences and  upper bounds, universal to  all suitable cutoff functions  $\Psi\in C^\infty(D_B)$.
\begin{question}\label{q3}
Under a minimal non-degeneracy type condition of $(P,D_B)$,
\begin{itemize}
\item  check if $
 \int e^{i\lambda P(x)} \Psi(x)dx  
$ converge for all appropriate    $\Psi\in C^\infty(D_B)$, or not.
\item find the best indices $\rho$ and $a$:  $
 \left|\int e^{i\lambda P(x)} \Psi(x)dx \right|\lesssim \lambda^{-\rho} (|\log  \lambda|+1)^a 
$  for all appropriate    $\Psi\in C^\infty(D_B)$, according as $\lambda\in [1,\infty)$ or $\lambda\in (0,1)$.
\end{itemize}
The constants involved in $\approx$ and $\lesssim$,    depend on $(P,D_B)$, but are independent of $\lambda$.  
\end{question}

\subsection{Local Estimates}

We first go over  these questions in a sufficiently small neighborhood $D$ of the origin. Write the usual Newton polyhedron as  ${\bf N}(P)$, which is defined by  ${\bf conv}(\Lambda(P)+\mathbb{R}_+^d)$  where ${\bf conv}(E)$ denotes the convex hull of a set $E$. Call  the non-negative number $ d({\bf N}(P) )$  displayed below,   the Newton distance:  $$  {\bf N}(P) \cap \rm{cone}({\bf 1})=\left[d({\bf N}(P) ),\infty\right){\bf 1}\ \text{ with ${\bf 1}=(1,\cdots,1)\in \mathbb{R}^d$}$$
where   $\rm{cone}({\bf 1}):=\{t {\bf 1} :t\ge 0\}$ is   the diagonal ray.
 In 1976, Varchenko  \cite{V}  studied the asymptotics of the oscillatory integrals  in Question \ref{q3} with $\Psi$ supported near the origin,   associated with real analytic phase functions $P$  with   $ \nabla P({\bf 0})={\bf 0}$.  He assumed that $d({\bf N}(P) )>1$ and  imposed the  face-nondegenerate-hypothesis:
\begin{align}\label{vc0}
\text{$\nabla P_{\mathbb{F}}|_{(\mathbb{R}\setminus\{0\})^d} $ are non-vanishing for all compact faces of $\mathbb{F}$ of ${\bf N}(P)$} \end{align}
where  $P_{\mathbb{F}}(x):=\sum_{\mathfrak{m}\in \Lambda(P)\cap\mathbb{F}}c_{\mathfrak{m}}x^{\mathfrak{m}}$. Then he calculated the     oscillation  index $\rho$  to be    $1/ d({\bf N}(P) )$ and
  the multiplicity $a$  to be   $ d-1-k$ for $k=\text{dim}(\mathbb{F}^{\rm{main}})$. Here $\mathbb{F}^{\rm{main}}$ is  the lowest dimensional face of ${\bf N}(P)$ containing   $d({\bf N}(P) )\bf{1}$.  In 1977,   Vassiliev \cite{V1} proved that   the sublevel-set-growth-index of    Question \ref{q1} for the local domain,  $\rho$  is $1/d({\bf N}(P))$  and the multiplicity $a$   is $d-1 -k  $, under the  normal-crossing   assumption:
\begin{align}\label{vc1}
\text{there is  $c>0$ so that $|P(x)|\ge c\sum_{\mathfrak{m}\in \Lambda(P)} |x^{\mathfrak{m}}|$ $\forall\ x\in D\cap (\mathbb{R}\setminus\{0\})^d$}.
\end{align}

\subsection{Model Result}
 To obtain the corresponding indices of Varchenko and Vassilev \cite{V,V1} in the whole domain $D_B=\mathbb{R}^d$,   define the  
the Newton  polyhedron ${\bf N}(P,\mathbb{R}^d)$   as the convex hull of $\Lambda(P)$.   Then $\rm{cone}(\Lambda(P)\cap\{ -{\bf 1}\})$ determines the convergence in Questions 1.1, and  the line segment
 $[\delta_{\rm{for}},\delta_{\rm{bac}}] {\bf 1}={\bf N}(P,\mathbb{R}^d)\cap\,\rm{cone}({\bf 1}) $ determines the growth rates  in Questions 1.2,  as well as those of Question 1.3:
\begin{itemize}
\item[{\bf A 1.1}.]  If  $\rm{cone}(\Lambda(P)\cap\{ -{\bf 1}\})\ne\mathbb{R}^d$, then  $|\{x\in \mathbb{R}^d:|\lambda P(x)|\le 1\}|$ diverges. 
\item[{\bf  A 1.2}.]  If  $\rm{cone}(\Lambda(P)\cap\{ -{\bf 1}\})=\mathbb{R}^d$, then the   growth rate $\rho$ and its multiplicity    $a$ in  the sub-level set estimate in Question \ref{q1}  under the condition  (\ref{vc1}) are 
$$(\rho,a)=\begin{cases} (1/\delta_{\rm{for}},d-1-\text{dim}(\mathbb{F}_{\rm{for}}^{\rm{main}})) \ \text{ if $\lambda\ge 1$}\\
 (1/\delta_{\rm{bac}},d-1-\text{dim}(\mathbb{F}_{\rm{bac}}^{\rm{main}})) \ \text{ if $0<\lambda<1$}
 \end{cases} $$
with $\mathbb{F}_{\rm{for}}^{\rm{main}}$ and $\mathbb{F}_{\rm{bac}}^{\rm{main}}$ smallest faces of ${\bf N}(P,\mathbb{R}^d)$ containing $\delta_{\rm{for}}{\bf 1}$ and $\delta_{\rm{bac}}{\bf 1}$.

\item[{\bf A 1.3}]  If  $\rm{cone}(\Lambda(P)\cap\{ -{\bf 1}\})\ne\mathbb{R}^d$, then there is an appropriate $\Psi$ supported in $D_B$ such that $ |\int e^{i\lambda P(x)} \Psi(x)dx  |=\infty$. If $\rm{cone}(\Lambda(P)\cap\{ -{\bf 1}\})=\mathbb{R}^d$, then  $  \left|\int e^{i\lambda P(x)} \Psi(x)dx \right|\lesssim \lambda^{-\rho} (|\log  \lambda|+1)^a $ with the same indices $\rho$ and $a$ above.
\end{itemize}

  Not like a local domain, in the global domain,  the indices $\rho$ and $a$  turn out to be different according  as  $0<\lambda\ll 1\ \text{or}\ \lambda\gg 1$. 
 \begin{remark}
From the computations of the model phase   $P(x_1,x_2)=x_1x_2$:
\begin{align*}
& \int_{D_{\{ {\bf 0}\}}} e^{2\pi i \lambda x_1 x_2}dx_1dx_2:=\lim_{R\rightarrow \infty}\int_{|x_1|<R}\int_{|x_2|\le R} e^{2\pi i \lambda x_1 x_2}dx_1dx_2 = \frac{2}{\lambda}\\
& \int  e^{2\pi i \lambda x_1 x_2}\psi_{D_{\{ {\bf 0}\}}} (x)dx_1dx_2:=\lim_{R\rightarrow \infty}\int e^{2\pi i \lambda x_1 x_2} \psi(\frac{x_1}{R})\psi(\frac{x_2}{R}) dx_1dx_2=O(\frac{1}{\lambda})
\end{align*}
we can observe that the criterion of the convergence of the oscillatory integral for a fixed amplitude $\psi_{D_B}$ would be different from that of the sublevel-set estimates above. However, we do not deal with these main issues in this paper.
\end{remark}

\subsection{Resolution of Singularities in the local region}
 When (\ref{vc1}) breaks down, one needs appropriate resolutions of the singularities.    For the study of the classical    resolution of singularities  of analytic functions, we refer  selectively \cite{Ab,H1,Jung} with its evolution  \cite{BM,EV,Ha,Su,W}     in the context of  algebraic geometry.
When the non-degeneracy hypothesis (\ref{vc0}) (or (\ref{vc1})) fails,   Varchenko \cite{V}  established, via toric geometry,   the resolution of singularity algorithm in $\mathbb{R}^2$ for finding an adaptable local coordinate system $\Phi$ satisfying $$d({\bf N}(P\circ\Phi))=\sup_{\phi\ \text{local coordinates}}d({\bf N}(P\circ\phi)) .$$
    Later, 
 Phong, Stein and Sturm \cite{PSS} utilized the Weierstrass preparation theorem and the   Piuseux series expansions of the roots $ r_i$ of $P(x_1,r_i(x_1))=0$ for constructing pullbacks   $\phi_i(x)=(x_1,x_2+r_i(x_1))$ of the horns $D_i$    for $D\subset\bigcup_{i=1}^M D_i$  making $P\circ \phi_i$ satisfying (\ref{vc0}) and (\ref{vc1}) on $\phi_i^{-1}(D_i)$.  Moreover,  Ikromov and Muller  \cite{IM} accomplished the Varchenko's algorithm for the adaptable local   coordinate systems  of the form $\phi(x_1,x_2)= (x_1,x_2+r(x_1))\ \text{or}\  (x_1+r(x_2),x_2)$  with $r$ analytic and $r({\bf 0})=0$  in $\mathbb{R}^2$ by using the   elementary analysis of the Newton polyhedron. They can  handle  a class of smooth functions. Moreover,  Greenblatt \cite{Gr} computed the leading terms of asymptotics of related integrals for the smooth phase after constructing Varchenko's adaptable local  coordinate system in $\mathbb{R}^2$ by performing only an elementary analysis such as an implicit function theorem. 
   With only analysis tools,  Greenblatt \cite{G2}   utilized the induction argument, as in a spirit of Hironaka's,  to establish   an elementary local resolution of singularities in $\mathbb{R}^d$ for all $d\ge 1$.  More recently, Collins, Greenleaf and Pramanik \cite{CGP}  developed  the classical  resolution of singularities   to obtain a higher dimensional resolution of singularity algorithm  applicable to  the above  integrals   in  a local domian of $\mathbb{R}^d$.   In  a   small neighborhood $D$ of the origin,    the    aforementioned oscillatory integral estimates yield the  oscillation index of $P$ at the origin  given by the infimum  $\rho_{0} $ below over all   $\psi$ of the asymptotics 
 \begin{align*} 
 \int e^{i\lambda P(x)} \psi(x)dx\sim\sum_{i=0}^{\infty}\sum_{n=0}^{d-1} c_{i,n}(\psi) \lambda^{-\rho_i}(\log\lambda)^{n} 
 \end{align*}
where $\rho_i<\rho_{i+1}$ and $ \sum_{n=0}^{d-1}|c_{0,n}(\psi)|\ne 0$ as $\lambda\rightarrow \infty$ (its existence  follows from Hironaka's resolution of singularities in \cite{H1}). In this paper, we do not establish the global resolution of singularity nor exact initial coefficient $c_{0,n(0)}$.  But,  we   fix  a coordinate system  under a  non-degeneracy hypothesis  as a global variant  of  (\ref{vc0}) or (\ref{vc1}), and  focus only on finding the   leading indices $\rho$ and $a$  of  Varchenko or Vassiliev in  Main Theorems \ref{mainth13} and \ref{mainth15}.    Next,  we partition  the domain   $D\subset\bigcup_{i=0}^M D_i$  so that    $P\circ \phi_i$ are normal-crossing on $\phi_i^{-1}(D_i)$ for all $i=0,\cdots,  M$,  in Main Theorem 3. 
\\
\\
   \noindent
 {\bf Notation}.
Denote the set of non-negative real numbers (integers, rationals)  by $\mathbb{R}_+$ ($\mathbb{Z}_+$, $\mathbb{Q}_+$).   For $j=(j_1,\cdots,j_d)\in \mathbb{R}^d$, we write ${\bf 2}^{-j}=(2^{-j_1},\cdots, 2^{-j_d})$. Moreover,  by $\pm {\bf 2}^{-j}$, we denote  the $2^d$ number of  all possible vectors of the forms $(\pm 2^{-j_1},\cdots,\pm 2^{-j_d})$. Set those vectors with their exponents in $K\subset \mathbb{R}^d$ as
\begin{align}\label{jmq4}
 {\bf 2}^{-K}:=\{ \pm {\bf 2}^{-j}:j\in  K \}. 
  \end{align}
Let $[d]:=\{1,\cdots,d\}$.    Given $j= (j_1,\cdots,j_d)\in \mathbb{Z}^d$ and  $x=(x_1,\cdots,x_d)\in\mathbb{R}^d$ , we write the dilation
${\bf 2}^{-j}x=(2^{-j_1}x_1,\cdots,2^{-j_{d}}x_{d})$ and denote $ x\sim {\bf 2}^{-j}$ if \begin{align}\label{jmq3}
 2^{-j_\nu-1}\le |x_\nu|\le 2^{-j_\nu+1}\ \text{for all $\nu\in [d]$}.
 \end{align}
Thus $x\sim {\bf 2}^{-{\bf 0}}$ iff $1/2\le |x_\nu|\le 2$ for $\nu\in [d]$. Sometimes, we shall use the notation $x\sim_h {\bf 2}^{{\bf 0}}$ representing 
  $1/h\le |x_\nu|\le  h$ for a fixed  number $1\le h<\infty$.  
 For $K\subset \mathbb{Z}^d$ and $B\subset \mathbb{Z}^d\cap [-r,r]^d$, we denote the set $K+ B$    by  $K+O(r)$. 
  We employ the following smooth non-negative cutoff functions
 \begin{itemize}
 \item[(1)] $\psi$ supported in $\{u\in\mathbb{R}^d:|u|\le 1\}$ for $\psi(u)\equiv 1$ in $|u|<\frac{1}{2}$ and $\psi^c=1-\psi$,
 \item[(2)]  $\chi$  supported in $\{ u\in \mathbb{R}:1/2\le |u|\le 2\}  \ \text{or}\  \bigcap_{\nu=1}^d\{ u\in\mathbb{R}^d:\frac{1}{2}\le |u_\nu |\le 2\}$, 
  \end{itemize}
  allowing slight line-by-line modifications of $\chi$ and $\psi$.  In this paper, we let $D$ be a  Borel set  and let $\psi_{D}$   indicate 
 a   cutoff function supported in   $D$.  Given two scalars $a,b$,  write $a\lesssim b$  if $a\le Cb$ for some  $C>0$ depending only on  $(P,D_B)$ in (I.1)-(I.2). The notation $a\approx b$ means that $a\lesssim b$ and $b\lesssim a$.  Notice the bounds involved in $\approx,\lesssim$ of   (I.1) and (I.2) are independent of $\lambda$ and $x$.  In additions,
  denote $0\le a\ll b$ if $a/b$ is a sufficiently small number compared with 1. Note that   our positive constants $\epsilon$ ($\epsilon\ll 1$) and    $c,C$ may be   different   line by line.  Finally,    $\text{rank}(A)$ is the number of linearly independent vectors in $A$.  \\
 \\
   {\bf Organization}.
   In Section \ref{Sec2}, we define the Newton polyhedron ${\bf N}(P,D_B)$ associated with a general domain $D_B$ and its balancing condition that determines a divergence of  our integral. In Section \ref{Sec4}, we  state the two main theorems regarding Questions 1.1 and 1.2 under a    normal-crossing  hypothesis derived from    (\ref{vc1})  on $(P,D_B)$.
In Section \ref{Sec3}, we   introduce   basic properties of  polyhedra and  its supporting planes.     In Sections \ref{Sec5}, we prove  some combinatorial lemmas regarding the distances and orientations of  ${\bf N}(P,D_B)$.  In Section \ref{Sec6}, we decompose our integrals according to the oriented and simplicial   dual faces of ${\bf N}(P,D_B)$. In Sections \ref{Sec7}-\ref{Sec9}, we give a bulk of proofs for the two main theorems stated in Section \ref{Sec4}.  In Section 10, by partitioning  domains $D$  into finite  pieces,  we restate the main results under a type of face-nondegeneracy (\ref{vc0}). 
    In the last two sections, we prove the dual face decomposition (Theorem 6.1), and the equivalence of   two  nondegeneracy conditions (Theorem 10.1).

 \section{Global Integrals and Newton Polyhedra}\label{Sec2}

  \subsection{Two Aspects of Global Integrals}\label{ej1}
The following  examples    illustrate  the two main features of  global integral estimates.
 \begin{itemize}
\item[(F1)]  Divergence due to unbalanced  ${\bf Ch}(\Lambda(P))$.   Compare:  
\begin{itemize}  
\item[(i)] $  \lim_{R\rightarrow \infty} \int \psi(\lambda (x_1^2+x_2^2))\psi\left(\frac{x_1}{R}\right)\psi\left(\frac{x_2}{R}\right)dx_1dx_2\approx   \lambda^{-1}$,
 \item[(ii)] $ 
    \lim_{R\rightarrow \infty} \int \psi(\lambda x_1^2)\psi\left(\frac{x_1}{R}\right)\psi\left(\frac{x_2}{R}\right)dx_1dx_2 \approx  \lim_{R\rightarrow \infty}\lambda^{ -1/2}  R
=\infty$.
\end{itemize}
Observe that   the divergence  of (ii) 
is owing to the deviation ${\bf Ch}(\Lambda(x_1^2))=\{(2,0)\} $   from  the   $\rm{cone}({\bf 1})$  causing  
 the   biased integration   $dx_2$.   
\item[(F2)] Different   decays  according to   $\lambda\gg 1$ or $\lambda\ll 1$.  Consider the 1-D estimates  
 \begin{align*}
  \int_{\mathbb{R}}\psi(\lambda(t^4+t^6))dt&\approx \begin{cases}
 \int_{|t|<1}\psi(\lambda t^4)dt\approx \lambda^{-1/4}\ \text{if $\lambda\in [1,\infty)$}\\
 \int_{|t|>1}\psi(\lambda t^6)dt\approx\lambda^{-1/6}\ \text{if $\lambda\in (0,1)$}.
\end{cases}
\end{align*} 
This $1D$ estimate  yields the  distinct decay rates of the  $2D$ sublevel-set measure for $P(x_1,x_2)=x_1^4
 +x_2^4+x_1^6+x_2^6$
 according as $\lambda\gg 1$  or $\lambda\ll 1$:
\begin{align*}
\qquad\ \ \int_{\mathbb{R}^2} \psi(\lambda  P(x)) dx_1dx_2 & \approx   \prod_{i=1}^2 \int_{\mathbb{R}} \psi(\lambda  (x_i^4+x_i^6)) dx_i \approx \begin{cases}
\lambda^{-1/2}\ \text{if  $\lambda\in [1,\infty)$}\\
\lambda^{-1/3}\ \text{if $\lambda\in (0,1)$.}
\end{cases}
\end{align*} 
Observe   the    exponents  $1/2$  and  $1/3$   are   in $[2,3]{\bf 1}= {\bf Ch}(\Lambda(P))\cap \rm{cone}({\bf 1})$.  
\end{itemize} 
The   geometric intuition   in (F1) suggests us to define the  balancing condition of Newton polyhedron along $\rm{cone}({\bf 1})$ in Section \ref{Sec2.4}.  The regions $|t|<1$ and $|t|\ge 1$ in (F2) suggests us to define forward and backward polyhedra   in Section \ref{sec6.2}.

    \subsection{Domain $D_B$ and its dual cone representation}
  \begin{definition}\label{deir1} 
Given a  finite $B\subset\mathbb{Q}^{d}$, we set     $ \rm{cone}(B)$  and  its dual   $ \rm{cone}^{\vee}(B) $: \begin{align*} 
\rm{cone}(B)&:= \left\{\sum_{\mathfrak{b}\in B}\alpha_{\mathfrak{b}}\mathfrak{b}: \alpha_{\mathfrak{b}}\ge 0 \right\} \ \text{and}\  
\rm{cone}^{\vee}(B) := \bigcap_{\mathfrak{b}\in\rm{cone}(B)} \{\mathfrak{q}:\langle\mathfrak{b},\mathfrak{q}\rangle\ge 0\}.\end{align*} 
\end{definition} 
\begin{exam}\label{exx21}
If   $B=\{{\bf e}_\nu\}_{\nu=1}^d$,  $\rm{cone}(B)=\rm{cone}^{\vee}(B)=\mathbb{R}_+^d$, and $\rm{cone}^{\vee}(\{{\bf 0}\})=\mathbb{R}^d.$ 
\end{exam}
\begin{lemma}
For $K\subset \mathbb{R}^d$, recall ${\bf 2}^{-K}=\{\pm{\bf 2}^{-j}:j\in K\}$ defined in (\ref{jmq4}).  Then,
\begin{align}
D_B= D_{\rm{cone}(B)}   =  {\bf 2}^{-\rm{cone}^{\vee}(B)}\ \text{except set of measure $0$}. \label{2323}
\end{align}
Hence, for evaluating  integrals  $\int_{D_B}$,  we can regard $D_{B}={\bf 2}^{-\rm{cone}^{\vee}(B)}$  from (\ref{2323}). \end{lemma}
\begin{proof}
To show $D_B= D_{\rm{cone}(B)}$, it suffices to claim $D_B\subset D_{\rm{cone}(B)} $ for $B=\{\mathfrak{b}_1,\mathfrak{b}_2\}$. If  $x\in D_B$, then  $|x^{\mathfrak{b}_1}|, |x^{\mathfrak{b}_2}|\le 1$. Thus $|x^{\alpha_1\mathfrak{b}_1 +\alpha_2\mathfrak{b}_2}| =|x^{\mathfrak{b}_1}|^{\alpha_1}|x^{\mathfrak{b}_2}|^{\alpha_2}\le 1$   for all $\alpha_1,\alpha_2\ge 0$. So $x\in D_{\rm{cone}(B)}$. Next, 
set $ D'_{\rm{cone}(B)}:=  \{x\in   (\mathbb{R}\setminus\{0\})^d:|x^{\mathfrak{b}}|\le 1\ \text{for all}\ \mathfrak{b}\in \rm{cone}\left(B\right) \} = \{\pm {\bf 2}^{-\mathfrak{q}}:2^{-\langle  \mathfrak{b},\mathfrak{q}\rangle  }\le 1\ \text{for all}\ \mathfrak{b}\in \rm{cone}\left(B\right)\}  $    which is
 $ 
   \{\pm {\bf 2}^{-\mathfrak{q}}:\langle \mathfrak{b},\mathfrak{q}\rangle\ge 0\ \text{for all}\ \mathfrak{b}\in \rm{cone}\left(B\right)\} = \{\pm {\bf 2}^{-\mathfrak{q}}:\mathfrak{q}\in \rm{cone}\left(B\right)^{\vee}\} =  {\bf 2}^{-\rm{cone}^{\vee}(B)}$.  Hence, the second part of (\ref{2323}) holds as  $ D_{\rm{cone}(B)} \setminus D_{\rm{cone}(B)}'\subset \bigcup_{\nu}\{x: x_\nu=0\}$.
\end{proof}

\subsection{Cutoff functions on $D_B$ and Strongly Convexity of $\rm{cone}(B)$}
 \begin{definition}[Amplitude  on $D_B$] \label{topdef}
   Let $\psi\in C^\infty([-1,1])$  such that $\psi \equiv 1$ on $[-1/2,1/2]$ and let $ D_{B,R}:=D_{B}\cap [-R,R]^d$ for  $R>0$. 
   Then  put
\begin{align*} 
\psi_{D_B}(x):= \prod_{\mathfrak{b}\in B}\psi\left(x^{\mathfrak{b}}\right)\ \  \text{and}\ \  \psi_{D_{B,R}}(x):= \psi_{D_B}(x) \psi\left( \frac{x_1}{R}  \right) \cdots \psi\left( \frac{x_d}{R}  \right) 
\end{align*} 
satisfying $\text{supp}(\psi_{D_B})\subset D_B$ and $\text{supp}(\psi_{D_{B,R}})\subset D_{B,R}$. More generally,  
denote by $\mathcal{A}(D_B)$    the set of   smooth cutoff functions  $\Psi$  supported  in $ D_B $, 
 satisfying  a   zero  symbol condition:
 \begin{eqnarray}\label{29a}
 \sup_{x\in   D_B}\left|x^{\alpha}\partial^{\alpha}_{x} \Psi(x) \right|\,\le C_\alpha\ \ \text{for}\  \alpha\in\mathbb{Z}_+^d.
  \end{eqnarray}
For example, $ \psi_{D_B},\psi_{D_{B,R}}$ defined above belong  to the class $ \mathcal{A}(D_B)$. 
  \end{definition}
By using $\psi_{D_B},\psi_{D_{B,R}}$,  we  express  the sub-level set measure $ |\{x\in D_{B}:|\lambda P(x)|\le 1\}|$ and  the oscillatory integral $\int e^{i\lambda  P(x)}\psi_{D_{B}}(x)dx$     as  the  limits:
\begin{align}\label{jmq31}
  \lim_{R\rightarrow\infty} \int \psi(\lambda P(x))\psi_{D_{B,R}}(x)dx, \ \text{and}\ 
  \lim_{R\rightarrow\infty}\int e^{i\lambda  P(x)}\psi_{D_{B,R}}(x)dx 
   \end{align}
respectively. 
   Into these integrals,  insert  the dyadic decomposition   $\sum_{j\in \mathbb{Z}^d} \chi\left(\frac{x}{{\bf 2}^{-j}}\right)\equiv 1$ with $\chi\left(\frac{x}{{\bf 2}^{-j}}\right): =\prod_{\nu=1}^d \chi \left(\frac{x_\nu}{2^{-j_\nu}}\right) $.  Then  the convergence  of the limit  in $R$   follows from the boundedness for the   sum over $j$ of  the absolute values below:
\begin{align*}  
 \sum_{j\in \mathbb{Z}^d}  \int \psi(\lambda P( x)) \psi_{D_{B}}( x)\chi\left(\frac{x}{{\bf 2}^{-j}}\right) dx\ \text{and}\ \sum_{j\in \mathbb{Z}^d}   \int e^{i \lambda P(x)}\psi_{D_B}(x)\chi\left(\frac{x}{{\bf 2}^{-j}}\right) dx.
\end{align*}
Write the two summands as $\mathcal{I}^{\rm{sub}}_{j}(\lambda)$ and $\mathcal{I}^{\rm{osc}}_{j}(\lambda)$ respectively, and observe    $ {\bf 2}^{-j}\sim x$  for  $x\in D_B={\bf 2}^{-\rm{cone}^{\vee}(B)}$   in  (\ref{2323}). Henceforth,  we  shall rewrite    (\ref{jmq31})  as
 \begin{align} \label{25}
 \mathcal{I}^{\rm{sub}}(P,D_B,\lambda) = \sum_{j\in    \rm{cone}^{\vee}(B) \cap \mathbb{Z}^d }  \mathcal{I}^{\rm{sub}}_{j}(\lambda),\\
\mathcal{I}^{\rm{osc}}(P,D_B,\lambda) = \sum_{j\in    \rm{cone}^{\vee}(B) \cap \mathbb{Z}^d }  \mathcal{I}^{\rm{osc}}_{j}(\lambda).\nonumber
\end{align}

 \begin{definition}\label{strong1}
Note  that $\rm{cone}(B)$ is said to be {\bf strongly convex} if   $\rm{cone}(B)\cap (-\rm{cone}(B))=\{{\bf 0}\}$.   
If $\rm{cone}(B)$ is not strongly convex, then there exist nonzero $\mathfrak{r}$ and $-\mathfrak{r}$ contained in $\rm{cone}(B)$. This implies that$$D_B =\{x:   |x^{\mathfrak{b}}|\le 1\ \text{for all $\mathfrak{b}\in B$}\} \subset \{x:  |x^{\mathfrak{r}}|\le 1\ \text{and}\  |x^{-\mathfrak{r}}|\le 1\}= \{x:  |x^{\mathfrak{r}}|=1\},$$ whose  measure is zero,  so that (\ref{jmq31}) and (\ref{25}) vanish.  Hence, we shall state our main main theorems in Section \ref{Sec42},  assuming   that   $\rm{cone}(B)$ is strongly convex. But, we shall define a   class of domains generalizing $D_B$ and  remove the strong convexity of $\rm{cone}(B)$ in  Section \ref{Sec43}, which enables us to treat a class of Laurent polynomials.
 \end{definition}

 \subsection{Newton Polyhedra} \label{Sec2.4}   Viewing $\mathbb{R}_+^2$ as  $ \rm{cone}({\bf e}_1,{\bf e}_2) $ in the original definition  ${\bf N}(P):={\bf conv}(\Lambda(P)+\mathbb{R}_+^2)$ in  the local region $ D_{\{{\bf e}_1,{\bf e}_2\}}$, we   extend the notion of the Newton polyhedron   to all pairs   of  polynomial $P$ and domain $D_B$ (with $B\subset \mathbb{Q}^d$).
\begin{definition} [Newton Polyhedron and Balancing Condition]\label{ded28}
Recall $\Lambda(P)$ the exponent set of $P(x)$ with $x\in\mathbb{R}^d$.
We define the Newton polyhedron  for $(P,D_B)$:
 \begin{align} \label{nph}
  {\bf N}(P,D_B)&: ={\bf conv}\left(   \Lambda(P) + \rm{cone}(B)  \right).
\end{align}
We say that ${\bf N}(P,D_B)\subset \mathbb{R}^d$ is {\bf balanced} if $\rm{cone}(  B\cup \Lambda(P)\cup\{-{\bf 1}\})= \mathbb{R}^d$ and  unbalanced  if $ \rm{cone}(  B\cup \Lambda(P)\cup\{-{\bf 1}\})\ne \mathbb{R}^d$.   See Figure \ref{osc10} for $B=\{{\bf 0}\}$.   
  \end{definition} 
 \begin{exam}\label{ded281}
Given a polynomial $P$, the regions $ \mathbb{R}^3=D_{\{{\bf 0}\}}$,   $[-1,1]^3=D_{\{{\bf e}_1,{\bf e}_2,{\bf e}_3\}}$ and $ (\mathbb{R}\setminus(-1,1))^3 =D_{\{-{\bf e}_1,-{\bf e}_2,-{\bf e}_3\}}$  in (\ref{ded281}), have the following Newton polyhedra:
\begin{itemize} 
\item[(1)]  ${\bf N}(P,D_{\{{\bf 0}\}}
   ) ={\bf conv}[\Lambda(P)]$: Convex hull of  $\Lambda(P)$,
 \item[(2)]
${\bf N}(P,D_{\{{\bf e}_1,{\bf e}_2,{\bf e}_3\}} ) ={\bf conv}\left[\Lambda(P)+
\rm{cone}({\bf e}_1,{\bf e}_2,{\bf e}_3)
\right]$:   Originally defined ${\bf N}(P)$.
\item[(3)]
${\bf N}(P,D_{\{-{\bf e}_1,-{\bf e}_2,-{\bf e}_3\}} ) ={\bf conv}\left[\Lambda(P)+
\rm{cone}(-{\bf e}_1,-{\bf e}_2,-{\bf e}_3)
\right] $.
  \end{itemize}  
  
\end{exam} 
 \begin{exam}\label{exx1}
Observe that    ${\bf N}(x_1^2+x_2^2,D_{\{{\bf 0}\}})= \overline{(2,0),(0,2)}$ are balanced since $\rm{cone}(B\cup \Lambda(P)\cup\{-{\bf 1}\})=\rm{cone}(\{(2,0),(0,2),-{\bf 1}\})=\mathbb{R}^2$. But   ${\bf N}(x_2^2,D_{\{{\bf 0}\}}) =\{(0,2)\}$ are unblanced as $\rm{cone}(  \{(0,0),(0,2),-{\bf 1}\})\ne \mathbb{R}^2$, and ${\bf N}(x_1^2x_2^2+x_1^4x_2^4,D_{\{{\bf 0}\}})=\overline{(2,2),(4,4)}$  are unbalanced since     $\rm{cone}(\{  (0,0),(2,2),(4,4), -{\bf 1}\}) \ne\mathbb{R}^2$.  Note that ${\bf N}(P,D_{\{{\bf e}_1,{\bf e}_2\}})\subset\mathbb{R}^2$ is balanced for any polynomial $P$ since $\rm{cone}({\bf e}_1,{\bf e}_2,-{\bf 1})=\mathbb{R}^2$.
  \end{exam}
 
\section{Statements of Main Theorems}\label{Sec4}
  \subsection{Normal Crossing Condition of $(P,D)$} 
    \begin{definition}\label{de41}
Given a polynomial $P(x)=\sum_{\mathfrak{m}\in \Lambda(P)}c_{\mathfrak{m}}x^{\mathfrak{m}}$  and a Borel set   $D\subset\mathbb{R}^d $,  call $(P,D) $   {\bf  normal-crossing of type}  $[\sigma,\tau]$  if $\tau\in \mathbb{Z}_+$ is the minimal number: 
\begin{align}\label{sv3}   
\sum_{\sigma\le |\alpha|\le \tau} |x^{\alpha}\partial_{x}^{\alpha}P(x)|\ge c\sum_{\mathfrak{m}\in \Lambda(P)}| x^{\mathfrak{m}}|\ \text{for all}\ x\in D\cap (\mathbb{R}\setminus\{0\})^d
\end{align} 
where $c>0$,  independent of $x$, can  depend  on  $(P,D)$.  Given $\tau\ge 1$,  type $[1,\tau]$  implies type $[0,\tau]$.
Denote the number $\tau$ above by $\tau_{\sigma}(P,D)$ or $\tau(P,D)$ for simplicity.
\end{definition}
  
\begin{exam}\label{ex41}
Let $P (x)=\sum_{\mu=1}^dc_\mu x^{\mathfrak{m}_\mu} $ with all $c_\mu\ne 0$. From $x^{\mathfrak{m}_\mu}=\prod_{\nu=1}^dx_\nu^{m_{\mu,\nu}}$, it follows that $x_\nu\partial_{x_\nu}P(x)=\sum_{\mu=1}^dc_\mu m_{\mu,\nu}x^{\mathfrak{m}_\mu}$.
Regarding $ c_{\mu}\mathfrak{m}_\mu$ as column vectors,   
  $$\big[(x_\nu\partial_{x_\nu}P(x))_{\nu=1}^d\big]^T=\sum_{\mu=1}^d  c_\mu  x^{\mathfrak{m}_\mu}\mathfrak{m}_\mu   =\big(c_1\mathfrak{m}_1,c_2\mathfrak{m}_2,\cdots,c_d\mathfrak{m}_d\big)(x^{\mathfrak{m}_1},\cdots,x^{\mathfrak{m}_d})^T$$ where $\big(c_1\mathfrak{m}_1,c_2\mathfrak{m}_2,\cdots,c_d\mathfrak{m}_d\big)$ is the $d\times d$  matrix. Hence $\tau_1(P,\mathbb{R}^d)=1$   if and only if   $\text{rank}\,(\mathfrak{m}_1,\cdots,\mathfrak{m}_d)=d$.   So $\tau_1(x_1^3-x_1x_2^2,\mathbb{R}^2)=\tau_1(x_1x_2+x_2x_3+x_3x_1,\mathbb{R}^3) =1$.  
   \end{exam}

\subsection{Main Results}  \label{Sec42}
 \begin{definition}\label{defi90}
 Let $\mathbb{P}\cap \rm{cone}({\bf 1})\ne \emptyset$ for  $\mathbb{P}={\bf N}(P,D_B)$, then there are   $\delta_{\rm{for}},\delta_{\rm{bac}}\ge 0$:
\begin{align}\label{jq514}
    \mathbb{P}\cap \rm{cone}({\bf 1})=[\delta_{\rm{for}},\delta_{\rm{bac}}] {\bf 1}.
\end{align}
Call   the face   $\mathbb{F} $ of $\mathbb{P}$,   of the minimal dimension, containing $\delta_{\rm{for}}{\bf 1}$ ($\delta_{\rm{bac}}{\bf 1}$),
{\bf  the main forward  (backward)  face}. Denote the face   by  $ \mathbb{F}^{\rm{main}}_{\rm{for}}$ ($\mathbb{F}^{\rm{main}}_{\rm{bac}}$), and  its dimension 
  by $ k_{\rm{for}}=  \text{dim}(\mathbb{F}_{\rm{for}}^{\rm{main}})\  (k_{\rm{bac}}=\text{dim}(\mathbb{F}_{\rm{bac}}^{\rm{main}}))$ respectively.  
  See Figures \ref{osc10} and \ref{forba89}. 
  \end{definition}
 
\begin{main}[Sublevel-Set]\label{mainth13}   \noindent
Let $P(x) $ be a   polynomial in $\mathbb{R}^d$. Suppose that $\rm{cone}(B)$  is  strongly convex  in Definition \ref{strong1} and  $0\in P(D_B)=\{P(x):x\in D_B\}$.
\begin{itemize}
\item[(A)] Suppose that ${\bf N}(P,D_B)$ is balanced.   If
  $  (P,D_B)$  is  normal-crossing of  type  $[0,\tau]$ for $ \tau<\delta_{\rm{for}}$, then it holds that 
\begin{align*} 
 |\{x\in D_{B}:|\lambda P(x)|\le 1\}|\approx \begin{cases} \lambda^{-1/\delta_{\rm{for}}  }  (|\log \lambda|+1)^{d-1-k_{\rm{for}}} \ \text{if}\  \lambda \in [1,\infty),\\
 \lambda^{-1/\delta_{\rm{bac}}  } (|\log\lambda|+1)^{d-1-k_{\rm{bac}}} \ \text{if}\   \lambda\in (0,1)  \end{cases}
\end{align*}
where the constants involving $\approx $ depend on $(P,D_B)$, independent of $\lambda$.
 \item[(B)] Suppose that ${\bf N}(P,D_B)$ is unbalanced. Then there exists $c>0$ such that
  $$ |\{x\in D_{B}:|\lambda P(x)|\le 1\}| =\infty\ \text{ for  all $ \lambda\in (0,c)$ }$$ 
where  the   range $(0,c)$ is $(0,\infty)$ if $D_B $ contains a neighborhood of the origin.  \end{itemize}
  \end{main} 
  
 \begin{remark}\label{rm42}
  We assume   $0\in P(D_B)$ for generalizing the condition $P({\bf 0})=0$.
\end{remark}

 \begin{corollary}[Powers,  Integrability]\label{mainth14}
\noindent  Let $(P,D_B) $ be   in Main Theorem \ref{mainth13}.
  \begin{itemize}
\item[(A)] Under the hypothesis of (A) of Main Theorem \ref{mainth13},
   $$  \left\{\rho\in (0,\infty):\int_{D_{B}} |P(x)|^{-\rho} dx<\infty\right\} =(1/\delta_{\rm{bac}},1/\delta_{\rm{for}}).$$ 
  \item[(B)]  
Under the hypothesis of (B) of Main Theorem \ref{mainth13},
$$\left\{\rho\in (0,\infty):\int_{D_{B}} |P(x)|^{-\rho} dx<\infty\right\} =\emptyset.$$
 \end{itemize}
  \end{corollary}

  \begin{main}[Oscillatory Integral]\label{mainth15}
 \noindent  Let $P(x) $ be a   polynomial in $\mathbb{R}^d$ and  let    $\rm{cone}(B)$ be   strongly convex  with    $0\in P(D_B)$.
 \begin{itemize}
 \item[(A)] Suppose that  ${\bf N}(P,D_B)$ is  balanced.  If $(P,D_B)$ is normal-crossing of  type  $[1,\tau]$ for $ \tau<\delta_{\rm{for}} $, then it holds that
\begin{align*}  
&\qquad\left|\int e^{i\lambda   P(x)}\psi_{D_{B}}(x)dx\right| 
   \le C\begin{cases}  \lambda^{-1/\delta_{\rm{for}}   } (|\log\lambda|+1)^{d-1-k_{\rm{for}}}\ \text{if}\  \lambda \in [1,\infty),\   \\
 \lambda^{-1/\delta_{\rm{bac}} }  (|\log\lambda|+1)^{d-1-k_{\rm{bac}}} \ \text{if}\  \lambda \in (0,1),
\end{cases}
\end{align*}  
and   there is $c>0$  with $p=d -k_{\rm{for}} $ and $q=d -k_{\rm{bac}}$ such that
\begin{align*} 
\ \ \  \  \limsup_{|\lambda|\rightarrow\infty} \left|\frac{  \int e^{i\lambda   P(x)}\psi_{D_{B}}(x)dx  }{\lambda^{-1/\delta_{\rm{for}}   }(|\log\lambda|+1)^{p-1}}\right|\ge c   \ \text{and}\ 
\limsup_{|\lambda|\rightarrow 0} \left|\frac{ \int e^{i\lambda   P(x)}\psi_{D_{B}}(x)dx }{\lambda^{-1/\delta_{\rm{bac}}   } (|\log\lambda|+1)^{q-1}}\right| \ge c. 
\end{align*} 
 \item[(B)] Suppose   ${\bf N}(P,D_B)$ is unbalanced. Then there is $\Psi_{D_B} \in \mathcal{A}(D_B)$  in (\ref{29a}):
  $$\left|\int_{\mathbb{R}^d} e^{i\lambda  P(x)}\Psi_{D_{B}}(x) dx\right|=\infty\ \text{ for almost every  $\lambda>0 $.} $$
 \end{itemize}
Once $\tau=1$ in the estimate $\lesssim$ of (A),  the restriction $ \tau<\delta_{\rm{for}}$ is not needed.
\end{main} 

 \begin{remark}\label{k66}
For the estimate of $\lesssim$ in (A) of Main Theorem  \ref{mainth15}, one can replace the amplitudes $\psi_{D_{B}}$  with all functions in  $ \mathcal{A}(D_B)$.
 \end{remark}
 \begin{remark}\label{ex43e}
The  unbalanced condition of ${\bf N}(P,D_B)$  in (B) of Main Theorem \ref{mainth15},   does not always imply a   divergence of    $\int e^{i\lambda   P(x)}\psi_{D_{B}}(x)dx$, but guarantees  an existence of an amplitude $\Psi_{D_B}\in \mathcal{A}(D_B)$  such that $|\int e^{i\lambda   P(x)}\Psi_{D_B}(x)dx|=\infty$. For instance,  
despite  the unblanced condition  of   ${\bf N}(x_1x_2,\mathbb{R}^2)=\{(1,1)\}$,    
\begin{align*}
   \lim_{R\rightarrow\infty} \int e^{i\lambda x_1x_2}  \psi(x_1/R)\psi(x_2/R)dx& = \lim_{R\rightarrow\infty} \int \widehat{\psi}(R\lambda x_2) R\psi(x_2/R)dx_2= O(\lambda^{-1})
\end{align*}
while   $ |\{(x_1,x_2)\in\mathbb{R}^2:|x_1x_2\lambda|<1\}| =\infty$  straightforwardly from Main Theorem \ref{mainth13}.
  \end{remark}  
 \begin{exam}
 Let $D=\mathbb{R}^2$ and $P(x_1,x_2)= x_1^4x_2^4(1-x_1^2-x_2^2)^2 $ whose nontrivial zeros are in  the unit circle. Then $\tau_1(P,D) =2<4=\delta_{\rm{for}}$ and $\delta_{\rm{bac}}=6$ with  $k_{\rm{for}}=0$ and $k_{\rm{bac}}=1$.  By applying (A) of Main Theorems \ref{mainth13} and \ref{mainth15},
 \begin{align*} 
  \int_{\mathbb{R}^2} \psi(\lambda P(x)) dx &\approx  \begin{cases} \lambda^{-1/4}  (|\log\lambda|+1)   \ \text{if}\  \lambda \in [1,\infty)\\
  \lambda^{-1/6}    \ \text{if}\  \lambda \in(0,1),
 \end{cases}\\
  \left|\int_{\mathbb{R}^2}e^{i\lambda P(x)} dx \right|&\lesssim  \begin{cases} \lambda^{-1/4}  (|\log\lambda|+1)   \ \text{if}\  \lambda \in [1,\infty)\\
  \lambda^{-1/6}    \ \text{if}\  \lambda \in(0,1).
 \end{cases}
\end{align*}
By Corollary \ref{mainth14}, we have  $ \int_{\mathbb{R}^d} |P(x)|^{-\rho} dx<\infty$ if and only if $\rho\in (1/6,1/4)$. 
 \end{exam}

 \begin{exam}
 Let  $D=\mathbb{R}^2$ and $P(x_1,x_2)= x_1^4x_2^4(x_2-x_1^2)^2 $ whose nontrivial zeros are in a parabola. Then $\tau_1(P,D) =2<16/3=\delta_{\rm{for}}=\delta_{\rm{bac}}$ and  $k_{\rm{for}}=k_{\rm{bac}}=1$. Apply (A) of Main Theorems \ref{mainth13} and \ref{mainth15} together with Corollary \ref{mainth14}  to have  \begin{align*} 
& \int_{\mathbb{R}^2} \psi(\lambda P(x)) dx  \approx  \lambda^{-3/16} \ \text{and}\    |\int_{\mathbb{R}^2} e^{i\lambda P(x)}dx| \lesssim \lambda^{-3/16}  \ \text{for all}\  \lambda \in(0,\infty)
\end{align*}
whereas $ \int_{\mathbb{R}^d} |P(x)|^{-\rho} dx=\infty$ for all $\rho\ge 0$.
 \end{exam} 
If the order of zeros of $P(x)$  
is less than $\delta_{\rm{for}}$, that is, $\tau_0(P,D)<\delta_{\rm{for}}$, then one can apply Main Theorems \ref{mainth13}    to obtain the exact growth indices of the sub-level set measure as well as its convergence.

\subsection{The Case  $ \tau <\delta_{\rm{for}}$  breaks}\label{rk43}
Suppose that $\delta_{\rm{for}}\le  \tau=\tau (P,D_B)<\delta_{\rm{bac}}$.  Then, once  $0<\lambda<1$, the  estimates of   Main Theorems  \ref{mainth13} and \ref{mainth15} still hold. If  $\lambda\ge 1$,  without resolution of singularity,  we   shall obtain at least a non-sharp estimate:
\begin{align} \label{es0}
|\mathcal{I}(P,D_B,\lambda)|\lesssim
 \frac{(|\log \lambda|+1)^{p } }{\lambda^{1/ \tau  }  }  \ \text{if}\  p  = \begin{cases}0 \ \ \ \ \ \ \   \ \text{if  $ \tau >\delta_{\rm{for}}$}\\
 d-k_{\rm{for}}  \ \text{if $ \tau =\delta_{\rm{for}}$.}\end{cases}
 \end{align}
where   $\mathcal{I}(P,D_B,\lambda) $  stands for both $\mathcal{I}^{\rm{sub}}(P,D_B,\lambda) $ and $\mathcal{I}^{\rm{osc}}(P,D_B,\lambda) $ in  (\ref{25}).  
This with (A)   of Main Theorems \ref{mainth13} and  \ref{mainth15},
implies that 
 \begin{align} \label{hd1}
\delta\in (\delta_{\rm{for}},\delta_{\rm{bac}})\cap ( \tau,\delta_{\rm{bac}})   \Rightarrow  |\mathcal{I}(P,D_B,\lambda)|   \le C_{\delta} \lambda^{-\frac{1}{\delta}  }\ \text{for all $\lambda\in (0,\infty)$}
\end{align}
since $  |\lambda|^{-1/\delta_{\rm{for}} }\le  |\lambda|^{-1/\delta_{\rm{bac}} }$ if $\lambda\ge 1$ and  $  |\lambda|^{-1/\delta_{\rm{bac}} }\le  |\lambda|^{-1/\delta_{\rm{for}} }$ if $\lambda\le 1$.   
Next, we consider the worse case $\tau\ge \delta_{\rm{bac}}$.  Then not only, it can break 
 (\ref{es0}), but also diverge. For example,  $|\{x\in \mathbb{R}^2:\lambda |(x_2 -x_1)^2|\le 1\}|=\infty$ for $\delta_{\rm{bac}}=1<2= \tau$ and $|\{x\in \mathbb{R}^2:\lambda |x_2^2-x_1^2|\le 1\}|=\infty$ for $ \delta_{\rm{bac}}=1=\tau$, because  ${\bf N}(x_2^2,D_{(-1,1)}), {\bf N}(x_2(x_2+2x_1),D_{(-1,1)})$ after coordinate changes, are unbalanced. 
 In Section 10, we   split $D_B=\bigcup D_{B_i}$ so as to treat the cases $ \tau (P,D_B)\ge \delta_{\rm{for}}$.

 \section{Polyhedra and  Balancing Conditions}\label{Sec3}  
\subsection{Two Representations of Polyhedra} \label{ex31}
   \begin{definition}[Polyhedron]\label{d12}
Let $V$ be  an inner product space of dimension $d$.  For $\mathfrak{q}\in  V\setminus\{{\bf 0}\}$ and $r\in\mathbb{R}$,  set  a hyperplane  and its upper half-space, 
\begin{align*} 
 \pi_{\mathfrak{q},r}=
\{ {\bf y}\in V:
 \langle\mathfrak{q}, {\bf y} \rangle = r\}\ \text{with}\     
\pi^+_{\mathfrak{q},r} = \{{\bf y}\in V:
  \langle\mathfrak{q}, {\bf y} \rangle \ge r \}.  
  \end{align*}
  Denote its interior  
 $ \{{\bf y}\in  V:
 \langle\mathfrak{q}, {\bf y} \rangle > r \}$ by  $(\pi^+_{\mathfrak{q},r} )^{\circ}. $  
Given    a  finite set  $\Pi(\mathbb{P})=\{\pi_{\mathfrak{q}_i,r_i}\}_{i=1}^M  $ of   hyperplanes,   define   a  {\bf convex polyhedron} $\mathbb{P}$  (convex polytope) as the intersection of the upper half-spaces of the elements in $\Pi(\mathbb{P})$:
\begin{align}\label{jq820}
\mathbb{P}=\bigcap_{\pi_{\mathfrak{q},r}\in\Pi(\mathbb{P})}  \pi_{\mathfrak{q},r}^+. 
\end{align}
  If all $r=0$,  then  $\mathbb{P}$ is called a {\bf convex polyhedral cone}.   If   $\mathbb{P}\cap (-\mathbb{P})=\{{\bf 0}\}$,  then $\mathbb{P}$ is said to be {\bf strongly convex}.    
\end{definition}
\begin{definition}[Supporting Plane]
Let $\mathbb{P}$ be a polyhedron in $V$. We say that a hyperplane $\pi_{\mathfrak{q},r}$ (which needs not belong to $\Pi(\mathbb{P})$)  is a {\bf supporting  plane of $\mathbb{P}$} if 
\begin{align*} 
\text{$\pi_{\mathfrak{q},r}\cap \mathbb{P}\ne \emptyset $ and  $\pi_{\mathfrak{q},r}^{+}\supset\mathbb{P} $.}
\end{align*}
Call   $\pi_{\mathfrak{q},r}^+$ a {\bf supporting-upper-half-space} of $\mathbb{P}$. Let  $\overline{\Pi}(\mathbb{P})$ stand  for the set of all  supporting  planes $\pi_{\mathfrak{q},r}$ of $\mathbb{P}$. Then   the inner normal vectors $\mathfrak{q}$ of the all elements in $\overline{\Pi}(\mathbb{P})$ form a convex polyhedral cone:
\begin{align}\label{jmq27}
\mathbb{P}^{\vee}:=\{\mathfrak{q}\in V:\pi_{\mathfrak{q},r}\in \overline{\Pi}(\mathbb{P})\}. 
\end{align}  
We call $\mathbb{P}^{\vee}$ the {\bf dual cone} of $\mathbb{P}$. In view of (\ref{jq820}), we can observe that   
\begin{align*} 
  \mathbb{P}^{\vee}=\left( \bigcap_{i=1} ^M \pi_{\mathfrak{q}_i,0}^+ \right)^{\vee}=\rm{cone}(\mathfrak{q}_1,\cdots,\mathfrak{q}_M),
  \end{align*}
where the second $\vee$  indicates a set of inner normal vectors of the supporting planes.   
 \end{definition}
We shall show that ${\bf N}(P,D_B)$ can be expressed as $\mathbb{P}$ of (\ref{jq820}) with $\mathbb{P}^{\vee}=\rm{cone}^{\vee}(B)$. The restriction of the bilinear form  $\langle,\rangle:\mathbb{P}\times \mathbb{P}^{\vee}\rightarrow\mathbb{R}$ to  $\Lambda(P)\times \rm{cone}^{\vee}(B)$  enables us to control all  $| x^{\mathfrak{m}}|\sim 2^{-\langle \mathfrak{m},\mathfrak{q}\rangle}$  for $x\sim {\bf 2}^{-\mathfrak{q}}\in {\bf 2}^{-\rm{cone}^{\vee}(B)}=D_B$ wiht $\mathfrak{q}=j$ in (\ref{25}).

 \begin{lemma}  $\mathbb{P}=\bigcap_{\pi_{\mathfrak{q},r}\in\Pi(\mathbb{P})}  \pi_{\mathfrak{q},r}^+$ in  (\ref{jq820}) if and only if
$\mathbb{P}={\bf N}(P,D_B)$ in (\ref{nph}) for some $P,B$. Here   $\mathfrak{q},\Lambda(P),B\in \mathbb{Q}^d$ and $r\in Q$.
\end{lemma}
\begin{proof} In   Definition \ref{deir1},   $\mathfrak{b}\in \rm{cone}(B)$  if and only if $\langle \mathfrak{b},  \mathfrak{q}\rangle \ge 0$    for all $\mathfrak{q}\in \rm{cone}^{\vee}(B)$:  $$\rm{cone}(B)=\bigcap_{\mathfrak{q}\in \rm{cone}^{\vee}(B)}\pi^+_{\mathfrak{q},0}=\bigcap_{\mathfrak{q}\in U}  \pi_{\mathfrak{q},0}^+  \ \text{as in  (\ref{jq820}) for $U=\{\mathfrak{q}_i\}_{i=1}^N\subset\rm{cone}\left(B\right)^{\vee}$.} $$
Thus  $
 \mathfrak{m} + \rm{cone}\left( B\right) =\bigcap_{\mathfrak{q}\in U}\pi_{\mathfrak{q},\langle \mathfrak{m},\mathfrak{q}\rangle}^+$ 
with $ \langle\mathfrak{n}(\mathfrak{q}),\mathfrak{q}\rangle:=\min\{ \langle\mathfrak{m},\mathfrak{q}\rangle :\mathfrak{m}\in \Lambda(P)\}$ implies 
\begin{align*} 
  {\bf N}(P,D_B)& ={\bf conv}\left(   \Lambda(P) + \rm{cone}(B)  \right)={\bf conv}\left\{  \bigcap_{\mathfrak{q}\in U}  \pi_{\mathfrak{q},\mathfrak{m}\cdot\mathfrak{q}}^+  : \mathfrak{m}\in \Lambda(P)  \right\}  \nonumber\\
  &= \bigcup_{\mathfrak{m}\in {\bf conv}(\Lambda(P))}  \bigcap_{\mathfrak{q}\in U}  \pi_{\mathfrak{q},\langle\mathfrak{m},\mathfrak{q}\rangle }^+  =  \bigcap_{\mathfrak{q}\in U} \bigcup_{\mathfrak{m}\in \Lambda(P)}   \pi_{\mathfrak{q}, \langle\mathfrak{m},\mathfrak{q}\rangle}^+ =  \bigcap_{\mathfrak{q}\in U} \pi_{\mathfrak{q},\langle \mathfrak{n}(\mathfrak{q}),\mathfrak{q}\rangle}^+
 \end{align*}
showing $\Leftarrow$.
Next $\Rightarrow$ follows from  
$\bigcap_{\pi_{\mathfrak{q},r}\in\Pi(\mathbb{P})}  \pi_{\mathfrak{q},r}^+= {\bf N}(P,D_B)$ where $\Lambda(P)=\{ r\mathfrak{q}/|\mathfrak{q}|:\pi_{\mathfrak{q},r}\in\Pi(\mathbb{P})\}$ and $\rm{cone}(B)^{\vee}=\{\mathfrak{q}:\pi_{\mathfrak{q},r}\in\Pi(\mathbb{P})\}$. 
 \end{proof}


 \subsection{Balancing Condition  of  Supporting Planes}\label{seca4}
  \begin{figure}
 \centerline{\includegraphics[width=13cm,height=9cm]{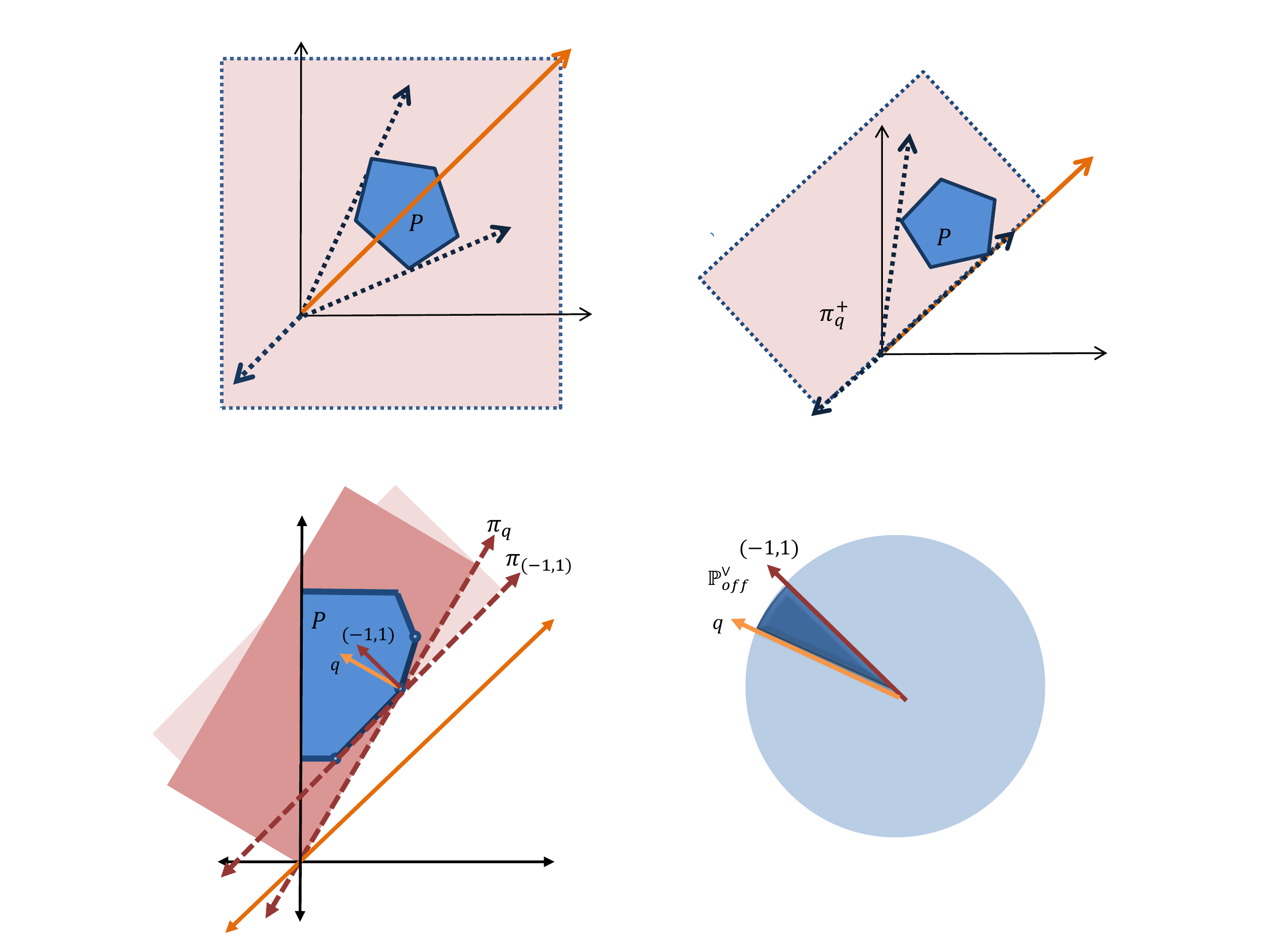}}
 \caption{ The first polyhedron $\mathbb{P}$ is balanced. The second $\mathbb{P}$ is unbalanced where  its supporting  upper space $\pi^+_{\mathfrak{q}}$  is   off-diagonal.  In the  lower  part,   $\mathbb{P}$ is off-diagonal. Set $\mathbb{P}^{\vee}_{\rm{off}}:=\rm{cone}^{\vee}(B\cup \Lambda(P)\cup \{-{\bf 1}\})$  as the   normal vectors of the off-diagonal supporting upper spaces.} \label{osc10}
 \end{figure}    
We state a balancing condition of a polyhedron ${\bf N}(P,D_B)=\bigcap \pi^+_{\mathfrak{q},r}$ in terms of $\pi^+_{\mathfrak{q},r}$. See    Figure \ref{osc10}. 
 \begin{definition}
An upper half space $\pi_{\mathfrak{q},r}^+$ is  {\bf across-diagonal} if   $(\pi_{\mathfrak{q},r}^+)^{\circ}\cap \rm{cone}({\bf 1})\ne\emptyset$, or
  {\bf off-diagonal} if    $(\pi_{\mathfrak{q},r}^+)^{\circ}\cap \rm{cone}({\bf 1})=\emptyset $  as in Figure \ref{osc10}. This implies
\begin{align}\label{5.6}
\text{ $\pi_{\mathfrak{q},r}^+ $ is off-diagonal\  if and only if    $  \langle\mathfrak{q},t{\bf 1}\rangle\le r$  $\forall t\ge 0$. }
\end{align}
 \end{definition}

\begin{lemma}\label{lems01}
From RHS of (\ref{5.6}),   it follows that
\begin{itemize}
\item[(1)] $
\text{  $\pi_{\mathfrak{q},r}^+ $ is off-diagonal if and only if  
  $  \rm{(i)}\  r\ge 0 \ \text{and}   \ \rm{(ii)}\   
  \langle\mathfrak{q},{\bf 1}\rangle \le 0$,  }
$
\item[(2)] $\pi_{\mathfrak{q},r}^+ $ is across-diagonal if and only if  
  $  \rm{(i)}\  r> 0 \ \text{or}   \ \rm{(ii)}\   
  \langle\mathfrak{q},{\bf 1}\rangle > 0$.
\item[(3)]
$\pi^+_{\mathfrak{q},r}$ is  across-diagonal if and only if  $\pi^+_{\mathfrak{q},r}$ as a Newton polyhedron  is balanced. 
\end{itemize}
 \end{lemma}
\begin{proof} 
Put $t=0$  and $t\gg 1$ in    (\ref{5.6})  to get (1), that gives (2).  To show (3),  observe 
  $\pi^+_{\mathfrak{q},r}={\bf N}(P, B)$ for
$\Lambda(P)=\{r\mathfrak{q}/|\mathfrak{q}|\}$ and $B=\rm{cone}^{\vee}(\mathfrak{q})$. Then  $\pi^+_{\mathfrak{q},r}$   balanced if and only if
$\rm{cone}(r\mathfrak{q}/|\mathfrak{q}|\cup \rm{cone}^{\vee}(\mathfrak{q})\cup\{-\bf 1\})=\mathbb{R}^d$ if and only if 
$r<0$ or $\langle\mathfrak{q},{\bf 1}\rangle>0$.
 \end{proof}
Recall  the set  $ \overline{\Pi}(\mathbb{P})$ of all supporting planes of $\mathbb{P}$.
\begin{proposition}\label{lem53}
Let $\mathbb{P}={\bf N}(P,D_B)$. Then
all $\pi^+$ of  $\pi\in\overline{\Pi}(\mathbb{P})$   are  across-diagonal (balanced)
 if and only if  $\mathbb{P}$  is  balanced. 
\end{proposition}
\begin{proof}  We claim its contraposition.  We can see that    there is $\pi_{\mathfrak{q},r}\in \overline{\Pi}({\bf N}(P,D_B)) $ such that  $\pi_{\mathfrak{q},r}$ is off-diagonal, i.e., $  \rm{(i)}\  r\ge 0 \ \text{and}   \ \rm{(ii)}\   
  \langle\mathfrak{q},{\bf 1}\rangle \le 0$ if and only if  there  is a nonzero $\mathfrak{q}\in \rm{cone}^{\vee}(B)$ such that $\langle q,\mathfrak{m}\rangle \ge 0$ for all $\mathfrak{m}\in\Lambda(P)$ and  $\langle\mathfrak{q},-{\bf 1}\rangle \ge 0$ if and only if there is a nonzero $\mathfrak{q}\in \rm{cone}^{\vee}(B\cup \Lambda(P)\cup\{-{\bf 1}\})$ if and only if $\rm{cone}(  B\cup \Lambda(P)\cup\{-{\bf 1}\})\ne \mathbb{R}^d$. 
\end{proof} 
\begin{remark}
 Geometrically,  $\mathbb{P}$ is balanced if and only if  $\rm{cone}(\mathbb{P})$ contains a conical neighborhood $\rm{cone}(\{{\bf 1}+\epsilon {\bf e}_\nu\}_{\nu\in [d]})$ of $\rm{cone}({\bf 1})$.
\end{remark}
 


    \section{Combinatorial Lemmas}\label{Sec5}
  
\subsection{Forward and Backward Orientation}\label{sec6.2} 
Let $\mathbb{P}={\bf N}(P,D_B)$ and consider  the integral    
 $\sum_{\mathfrak{q}\in \rm{cone}^{\vee}(B)\cap \mathbb{Z}^d}\int \psi(\lambda P(x)) \chi\left(\frac{x}{{\bf 2}^{-\mathfrak{q}}}\right)\psi_{D_B}(x)dx $  with $\mathfrak{q}=j$ in (\ref{25}).
    As $|\mathfrak{q}|\rightarrow \infty$,
the volume $|\{x:x\sim {\bf 2}^{-\mathfrak{q}}\} |\approx 2^{-  \langle\mathfrak{q},{\bf 1}\rangle}$ of its support  is to vanish or to blow up    as $ \langle\mathfrak{q},{\bf 1}\rangle>0$ or $ \langle\mathfrak{q},{\bf 1}\rangle<0$.  This observation leads us to  bisect  $  \rm{cone}^{\vee}(B)=\mathbb{P}^{\vee} $  according to  the  signs of  $   \langle\mathfrak{q},{\bf 1}\rangle  $ and split the domain $D_B=2^{-\rm{cone}^{\vee}(B)} $ of  the integral.

\begin{figure}
 \centerline{\includegraphics[width=13cm,height=9cm]{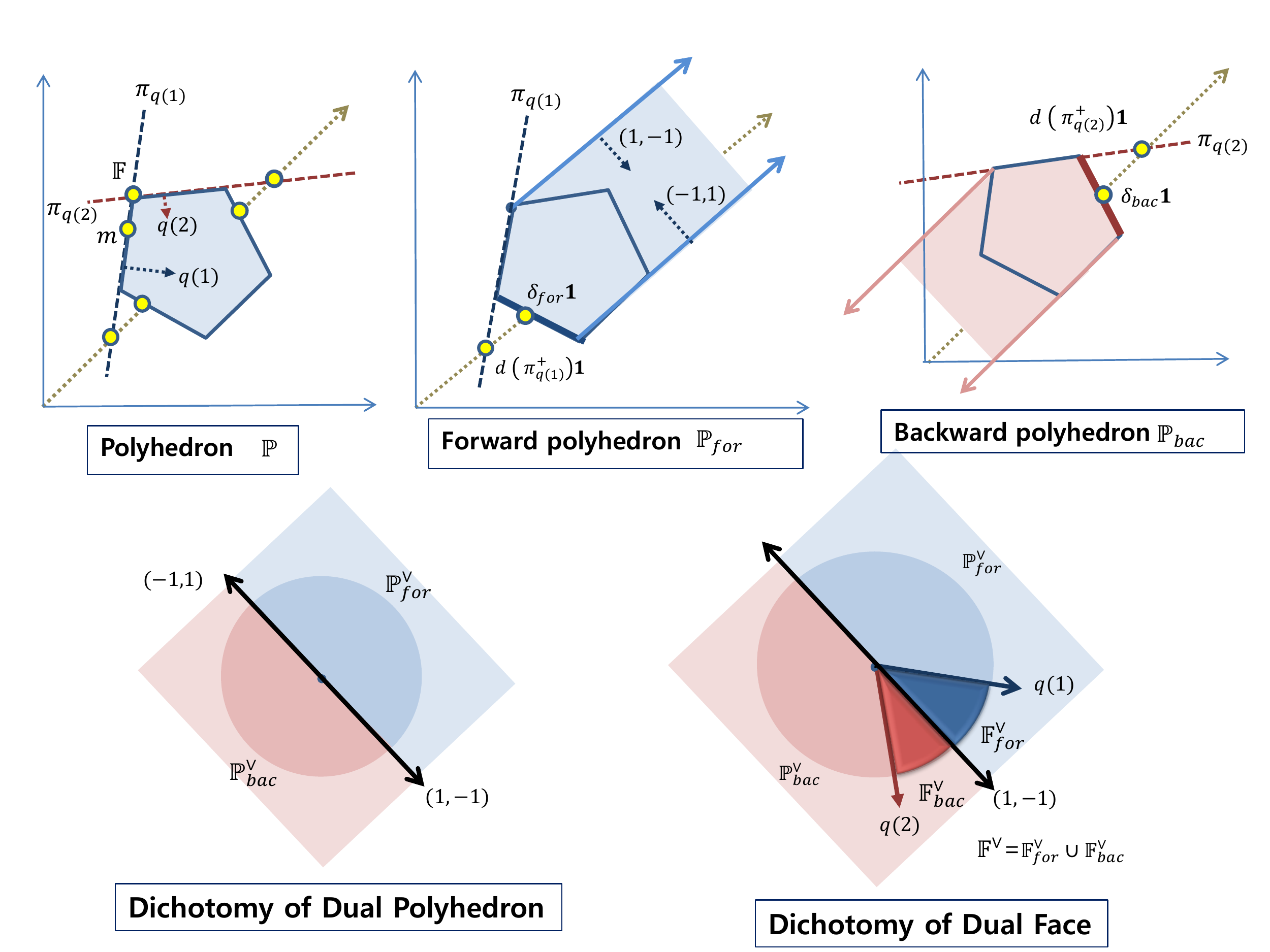}}
\caption{The pentagon $ \mathbb{P} $ has its forward and backward polyhedra   satisfying  $ \mathbb{P} =\mathbb{P}_{\rm{for}}\cap \mathbb{P}_{\rm{bac}}$  and  $\mathbb{P}^{\vee}_{\rm{for}}\cup \mathbb{P}^{\vee}_{\rm{back}}=\mathbb{P}^{\vee}=\mathbb{R}^2$. The vertex $\mathbb{F}$   has the dual face
 $\mathbb{F}^{\vee}$ splitting $ \mathbb{F}_{\rm{for}}^{\vee}\cup \mathbb{F}_{\rm{bac}}^{\vee}$ with $\mathbb{F}_{\rm{for}}^{\vee}=\rm{cone}(\mathfrak{q}(1),((1,-1))$ and $\mathbb{F}_{\rm{bac}}^{\vee}=\rm{cone}(\mathfrak{q}(2),((1,-1))$ in the last part.  } 
 \label{forba89}
 \end{figure}

\begin{definition} [Forward and Backward Orientation of Polyhedron]\label{ob52}
Let $\pi_{\mathfrak{q},r}\in \overline{\Pi}(\mathbb{P})$.  We call $\mathfrak{q},\pi_{\mathfrak{q},r}$ and $\pi_{\mathfrak{q},r}^+$  {\bf forward}
   if $ \langle\mathfrak{q},{\bf 1}\rangle\ge 0$, and {\bf backward}  if $ \langle\mathfrak{q},{\bf 1}\rangle\le 0$.
 We   split the dual cone $\mathbb{P}^{\vee}$ of $\mathbb{P}$ in (\ref{jmq27}) into the two sectors as $\mathbb{P}^{\vee}=\mathbb{P}^{\vee} _{\rm{for}} \cup \mathbb{P}^{\vee} _{\rm{bac}}  $ where
\begin{align*} 
\mathbb{P}^{\vee} _{\rm{for}} &:=  \mathbb{P}^{\vee} \cap \rm{cone}^{\vee}({\bf 1})\ \text{and}\ 
\mathbb{P}^{\vee} _{\rm{bac}}  := \mathbb{P}^{\vee} \cap \rm{cone}^{\vee}(-{\bf 1}).
\end{align*}
 Given $\mathbb{P}=\bigcap_{\mathfrak{q}\in \mathbb{P}^{\vee}} \pi_{\mathfrak{q},r}^+ $,  we can define a forward and a backward polyhedron of $\mathbb{P}$ as
\begin{align}\label{jmq6}
\mathbb{P}_{\rm{for}}:=\bigcap_{ \mathfrak{q}\in \mathbb{P}^{\vee}_{\rm{for}} } \pi_{\mathfrak{q},r}^+ \ \text{and}\ \mathbb{P}_{\rm{bac}}:= \bigcap_{ \mathfrak{q}\in \mathbb{P}^{\vee}_{\rm{bac}} } \pi_{\mathfrak{q},r}^+. 
\end{align}
 Theses are illustrated in   Figure \ref{forba89}.  When computing  $\mathcal{I}(P,D_B,\lambda)$ with $\mathbb{P}={\bf N}(P,D_B)$, it turns out that  $\mathbb{P}_{\rm{for}} $ determines the decay rate of  $\lambda\ge 1$, whereas
 $ \mathbb{P}_{\rm{bac}}$  determines not only the decay rate of  $\lambda< 1$,  but also the convergence of the integral.
  \end{definition}

 \subsection{Oriented Distances of Upper Half Spaces} 
\begin{definition}\label{df5.2}[Distance]
Let $\pi_{\mathfrak{q},r}\in \overline{\Pi}(\mathbb{P})$. To each across-diagonal upper half space $\pi_{\mathfrak{q},r}^+$,  from $(\pi_{\mathfrak{q},r}^+)^{\circ}\cap \rm{cone}({\bf 1})\ne \emptyset$,   one can assign the distance $ d (\pi_{\mathfrak{q},r}^+)$    satisfying
\begin{align}\label{jq143}
   (\pi_{\mathfrak{q},r}^+)^{\circ}\cap\mathbb{R} {\bf 1} =\begin{cases}
   (d(\pi_{\mathfrak{q},r}^+),\infty){\bf 1}\ \text{if $\mathfrak{q}$ is forward,}\\
   (-\infty,d(\pi_{\mathfrak{q},r}^+) ){\bf 1} \ \text{if $\mathfrak{q}$ is backward.}     \end{cases}
      \end{align}
In case of confusion, we  write  $d(\pi_{\mathfrak{q},r}^+)$ as $ d_{\rm{for}} (\pi_{\mathfrak{q},r}^+)$ or $  d_{\rm{bac}} (\pi_{\mathfrak{q},r}^+)$       according as  $\mathfrak{q}\cdot {\bf 1}\ge 0$ or $\mathfrak{q}\cdot {\bf 1}\le 0$ respectively.   See the second and third pictures of Figure \ref{forba89}. Once $\pi_{\mathfrak{q},r}^+$  is across-diagonal,  the  distance in (\ref{jq143}) exists from the observation: \begin{itemize}
 \item[(1)]
If $\langle\mathfrak{q},{\bf 1}\rangle\ne 0 $,  $ -\infty<d(\pi_{\mathfrak{q},r}^+)<\infty  $,  since $\pi_{\mathfrak{q},r} \cap\mathbb{R}{\bf 1}=  d(\pi_{\mathfrak{q},r}^+){\bf 1} $ is a singleton. \item[(2)] If $  \langle\mathfrak{q},{\bf 1}\rangle= 0 $,  then $d_{\rm{for}} (\pi^+_{\mathfrak{q},r})=-\infty$ and  $d_{\rm{bac}} (\pi^+_{\mathfrak{q},r})= \infty$ due to $(\pi^+_{\mathfrak{q},r})^{\circ}\supset\mathbb{R}{\bf 1}$.\end{itemize}
If $\mathbb{P}$  is balanced, we  define the distances $d(\mathbb{P}_{\rm{for}})$  and $d(\mathbb{P}_{\rm{bac}})$ by the numbers \footnote{If  $\mathbb{P}_{\rm{for}} \supset \mathbb{R}{\bf 1}$ or $ \mathbb{P}_{\rm{bac}} \supset  \mathbb{R}{\bf 1}$, then  $d(\mathbb{P}_{\rm{for}}):=-\infty$ or $d(\mathbb{P}_{\rm{bac}}):=\infty$ respectively.} \begin{align}\label{jmq9}
 [d(\mathbb{P}_{\rm{for}}),\infty){\bf 1} =\mathbb{P}_{\rm{for}} \cap \mathbb{R}{\bf 1}\ \text{and}\ (-\infty, d(\mathbb{P}_{\rm{bac}})]{\bf 1} =\mathbb{P}_{\rm{bac}} \cap \mathbb{R}{\bf 1}. 
\end{align}
 \end{definition}
 
\begin{lemma}\label{lem10}
Let $\mathbb{P}$  be balanced.  
Then   $  [d(\mathbb{P}_{\rm{for}}),  d(\mathbb{P}_{\rm{bac}})]{\bf 1}=\mathbb{P} \cap \mathbb{R}{\bf 1}\ne \emptyset$ with $ d(\mathbb{P}_{\rm{bac}})>0$.   
If $\pi_{\mathfrak{q}_1,r_1},\pi_{\mathfrak{q}_2,r_2}\in \overline{\Pi}(\mathbb{P})$ with $\mathfrak{q}_1\in \mathbb{P}^{\vee}_{\rm{for}}$ and $\mathfrak{q}_2\in \mathbb{P}^{\vee}_{\rm{bac}}$,  then 
\begin{align}\label{m18}
 d(\pi^+_{\mathfrak{q}_1,r_1})\le   d(\mathbb{P}_{\rm{for}})\le  \delta_{\rm{for}}=\max\{0,d(\mathbb{P}_{\rm{for}})\}\le  \delta_{\rm{bac}}=d(\mathbb{P}_{\rm{bac}}) \le d(\pi^+_{\mathfrak{q}_2,r_2}). 
 \end{align}
Recall that if $\mathbb{P}={\bf N}(P,D_B)$ for a polynomial $P$, then $  \mathbb{P}\, \cap \rm{cone}({\bf 1}) =[\delta_{\rm{for}},\delta_{\rm{bac}}]{\bf 1} $ as in (\ref{jq514}).   See the first three pictures in Figure  \ref{forba89}. 
  \end{lemma}

\begin{proof}[Proof of  (\ref{m18})]
Observe  that
 $\mathbb{P}\cap \{t{\bf 1}:t>0\}\ne \emptyset$ as $\mathbb{P}$  is  balanced.  Take $t^*{\bf 1}\in \mathbb{P}\cap \{t{\bf 1}:t>0\}\subset \mathbb{P}\cap \mathbb{R}{\bf 1}$.    
This  with (\ref{jmq9}) and $\mathbb{P}=\mathbb{P}_{\rm{for}}\cap \mathbb{P}_{\rm{bac}}$ yield  $[d(\mathbb{P}_{\rm{for}}),  d(\mathbb{P}_{\rm{bac}})]{\bf 1}=\mathbb{P} \cap \mathbb{R}{\bf 1} \ne\emptyset$ and $ d(\mathbb{P}_{\rm{bac}})>0$.  By this, we have  the middle part $ d(\mathbb{P}_{\rm{for}})\le  \delta_{\rm{for}} \le  \delta_{\rm{bac}}=d(\mathbb{P}_{\rm{bac}})$ in (\ref{m18}).
From (\ref{jmq6}),(\ref{jq143}) and (\ref{jmq9}), it follows that  
\begin{align*}
\begin{cases}
[d(\mathbb{P}_{\rm{for}}),\infty){\bf 1} =\bigcap_{\mathfrak{q}\in \mathbb{P}^{\vee}_{\rm{for}}} \pi_{\mathfrak{q},r}^+\cap \mathbb{R}{\bf 1}=\bigcap_{\mathfrak{q}\in \mathbb{P}^{\vee}_{\rm{for}}} [d(\pi_{\mathfrak{q},r}^+),\infty){\bf 1}, 
\\
(-\infty, d(\mathbb{P}_{\rm{bac}})]{\bf 1} =\bigcap_{\mathfrak{q}\in \mathbb{P}^{\vee}_{\rm{bac}}} \pi_{\mathfrak{q},r}^+\cap \mathbb{R}{\bf 1}=\bigcap_{\mathfrak{q}\in \mathbb{P}^{\vee}_{\rm{bac}}} (-\infty,d(\pi_{\mathfrak{q},r}^+)]{\bf 1}
\end{cases} 
 \end{align*}
which  give the first and last inequalities of (\ref{m18}).   
  \end{proof}

 \begin{lemma}\label{lem34joo}
Let $\pi_{\mathfrak{q},r}$ be an hyperplane containing $\mathfrak{m}\in    \pi_{\mathfrak{q},r}$. If $\langle \mathfrak{q}, {\bf 1}\rangle\ne 0$,  then
 \begin{align}\label{pm12}
\langle \mathfrak{m}, \mathfrak{q}\rangle=  \langle d (\pi^+_{\mathfrak{q},r}){\bf 1}, \mathfrak{q}\rangle.
\end{align}
 \end{lemma}
 \begin{proof}
  As $ d(\pi^+_{\mathfrak{q},r}){\bf 1},\mathfrak{m} \in \pi_{\mathfrak{q},r}$, it holds $(d(\pi^+_{\mathfrak{q},r}){\bf 1}-\mathfrak{m} )\perp \mathfrak{q}$ in the first of Figure \ref{forba89}.   
\end{proof} 

\subsection{Model Estimates}  Take a supporting plane  $\pi_{\mathfrak{q},r}\in \overline{\Pi}(\mathbb{P})$ such that $\langle \mathfrak{q},{\bf 1}\rangle\ne 0$. We apply Lemmas \ref{lems01} and \ref{lem34joo}  for the model sum in    (\ref{25}), given by $$\mathcal{I}^{\rm{sub}}(\lambda, \mathfrak{q} ):=\sum_{\alpha\in \mathbb{Z}_+} \int_{ x \sim  {\bf 2}^{-\alpha\mathfrak{q}}}  \psi(\lambda P(x))dx\ \text{with}\ d(\pi_{\mathfrak{q},r}^+)=\delta.$$
\begin{itemize}
\item[(A)]  Let $\pi_{\mathfrak{q},r}^+$  be  across-diagonal and
  $\tau=\tau_0(P,D_B)$. If  $\tau<  \delta$,  we  show that
  \begin{itemize}
  \item[(A-1)]    $  \int_{ x \sim 1}  \psi(\lambda P(2^{-\alpha \mathfrak{q}}x))dx=O(  |\lambda 2^{-\alpha \mathfrak{q}\cdot \mathfrak{m}}|^{-1/\tau}  )$  for $\mathfrak{m}\in\pi_{\mathfrak{q},r}$,  
  \item[(A-2)]     $ \mathfrak{q}\cdot \mathfrak{m} =  \delta \langle\mathfrak{q},{\bf 1}\rangle $ as in  Lemma \ref{lem34joo}, to obtain that  
  \end{itemize}
\begin{align*}
&\ \ \ \quad  \mathcal{I}^{\rm{sub}}(\lambda, \mathfrak{q} ) \approx \sum_{\alpha\in \mathbb{Z}_+}  \frac{2^{- \alpha\langle\mathfrak{q},{\bf 1}\rangle} }{(1+|\lambda 2^{-\alpha \mathfrak{q}\cdot \mathfrak{m}}|)^{1/\tau}}  =    \sum_{\alpha\in \mathbb{Z}_+} \frac{2^{- \alpha\langle\mathfrak{q},{\bf 1}\rangle} }{(1+|\lambda 2^{-\alpha \delta \langle\mathfrak{q},{\bf 1}\rangle }|)^{1/\tau}}  \lesssim \lambda^{-1/\delta }.
  \end{align*}
  \item[(B)] Let  $\pi_{\mathfrak{q},r}^+$ be off-diagonal for $\mathfrak{q}$ in  Lemma \ref{lems01}. Then  $ \int_{ x \sim  {\bf 2}^{-\alpha\mathfrak{q}}}  \psi(\lambda P(x))dx$ has
\begin{itemize}
 \item[(B-1)]  a  big support   $|\{x:x \sim  {\bf 2}^{-\alpha\mathfrak{q}}\}|\approx 2^{- \alpha\langle\mathfrak{q},{\bf 1}\rangle}\ge 1$    and
  \item[(B-2)]   a small phase   $| P(x)|\lesssim |x^{\mathfrak{m}}|\sim 2^{-\alpha\langle \mathfrak{m},\mathfrak{q}\rangle}= 2^{-\alpha \delta \langle\mathfrak{q},{\bf 1}\rangle }\ll 1 $ for  $\mathfrak{m}\in\pi_{\mathfrak{q},r} $,
  \end{itemize}
\begin{align*}
\text{showing  that}\ \mathcal{I}^{\rm{sub}}(\lambda, \mathfrak{q} )\approx \sum_{j=\alpha\mathfrak{q}} 2^{- \alpha\langle\mathfrak{q},{\bf 1}\rangle}  
  \approx    \infty\ \text{  for all real $\lambda$. }
  \end{align*}
\end{itemize} 

  \section{Oriented-Simplicial-Cone Decomposition}\label{Sec6}  
 \subsection{Basic Dual Face Decomposition}
   \begin{definition}\label{d412}
   [Face and Dual Faces]  
  Let $\mathbb{P}$ be a polyhedron in  $V$.  A subset $\mathbb{F}\subset \mathbb{P}$ is called a {\bf face} of $\mathbb{P}$
(denoted by $\mathbb{F}\preceq \mathbb{P}$) if there is  a supporting plane $ \pi_{ \mathfrak{q},r }\in \overline{\Pi}(\mathbb{P}) $ such that
$
\mathbb{F}= \pi_{ \mathfrak{q},r } \cap \mathbb{P}.  
$     Denote the set of $k$-dimensional faces of $\mathbb{P}$ by $\mathcal{F}^k(\mathbb{P})$ and the set of all faces by  $\mathcal{F}(\mathbb{P})$. 
Define the {\bf dual face} $\mathbb{F}^{\vee}$ of $\mathbb{F}\in \mathcal{F}(\mathbb{P})$ as   the set of all normal vectors $\mathfrak{q}$ of supporting planes $\pi_{\mathfrak{q},r} $ containing $\mathbb{F}$:\begin{align*} 
  \mathbb{F}^{\vee}=\left\{\mathfrak{q}\in\mathbb{P}^{\vee}:   \pi_{\mathfrak{q},r}\cap \mathbb{P}\supset \mathbb{F} \right\}\ \text{with}\ (\mathbb{F}^{\vee})^{\circ}=\left\{\mathfrak{q}\in\mathbb{P}^{\vee}:   \pi_{\mathfrak{q},r}\cap \mathbb{P}= \mathbb{F} \right\}.
  \end{align*}  
   \end{definition}
See $\mathbb{F}$ and $\mathbb{F}^{\vee}$ in the first and last  of Figure   \ref{forba89}. Finally,  call $\mathbb{F},\mathbb{F}^{\vee}$ {\bf oriented}, that is, forward or backward  if $\mathbb{F}^{\vee}\subset \rm{cone}^{\vee}({\bf 1})$ or $\mathbb{F}^{\vee}\subset \rm{cone}^{\vee}(-{\bf 1})$ respectively.

 \begin{proposition}[Dual Face decomposition]
If $\text{rank}(\mathbb{P}^{\vee})=d-k_0$, then,   
   \begin{align}\label{07hh}
 \mathbb{P}^{\vee}=  \bigcup_{\mathbb{F}\in \mathcal{F}(\mathbb{P})}\mathbb{F}^{\vee}= \bigcup_{  \mathbb{F}\in\mathcal{F}^{k_0}(\mathbb{P})}\mathbb{F}^{\vee} \end{align} 
 where $k_0$ is the minimal possible dimension\footnote{Every minimal (under $\subset$)   face  of  $\mathbb{P}$ has dimension $k_0=d-\text{rank}(\mathbb{P}^{\vee})$. }
of  faces in $\mathbb{P}$.  The second part  of (\ref{07hh}) follows from the fact 
 $\mathbb{G}\preceq\mathbb{F} \Rightarrow\mathbb{F}^{\vee}\preceq \mathbb{G}^{\vee}$. See \cite{Ful,O} 
\end{proposition}

\subsection{Oriented Simplicial Cone Decomposition of Integral}
By  $\mathbb{P}={\bf N}(P,D_B)$ with $\mathbb{P}^{\vee}=\rm{cone}^{\vee}(B)$ in the second part of (\ref{07hh}), we  can write
      (\ref{25})   as
\begin{align}\label{255}
\mathcal{I}^{\rm{sub}}(P,D_B,\lambda)= \sum_{\mathbb{F}\in\mathcal{F}^{k_0}({\bf N}(P,D_B))} \sum_{j \in\mathbb{F}^{\vee}\cap \mathbb{Z}^d} \mathcal{I}^{\rm{sub}}_j(\lambda).
\end{align} 
\begin{remark}\label{rmk61}
Reset $\mathbb{F}^{\vee}_n$ as $\mathbb{F}^{\vee}_n\setminus(\bigcup_{i=1}^{n-1}\mathbb{F}_i^{\vee})$ in $\mathcal{F}^{k_0}=\{\mathbb{F}_n\}_{n=1}^M$  to make all $\mathbb{F}^{\vee}$ in (\ref{255}) mutually disjoint, though they  originally  may overlap on their boundaries.
\end{remark}
 \begin{definition}[Essential Disjointness]\label{de3434}
  We say polyhedra $\mathbb{G}_1,\cdots,\mathbb{G}_m$ in $V$ are essentially disjoint (simply ess-disjoint) if $  \mathbb{G}_i^{\circ}\cap \mathbb{G}_j^{\circ}= \emptyset$ for all pairs with $\mathbb{G}^{\circ}=\mathbb{G}\setminus \partial \mathbb{G}$. 
   \end{definition}
\begin{definition}
Let $\mathbb{K}$ be a polyhedral cone of dimension $n$ in $ V $. We say that $\mathbb{K}$  is simplicial    if     $\mathbb{K}=\rm{cone}(\{\mathfrak{q}_i\}_{i=1}^n)$ for some linearly independent vectors $\mathfrak{q}_i$'s in $V$.   
\end{definition}
For a convenient computation,   we  shall make a dual face $\mathbb{F}^{\vee}$ in (\ref{255}):
\begin{itemize}
\item[(i)]  simplicial   $\rm{cone}(\mathfrak{q}_1,\cdots,\mathfrak{q}_{d_0})$  with $d_0=d-k_0$,
\item[(ii)] contained in  an oriented dual cone $ \mathbb{P}^{\vee}_{\rm{for}}\ \text{or}\ \mathbb{P}^{\vee}_{\rm{bac}}$.
\end{itemize}
\begin{theorem}\label{lemir1}[Oriented-Simplicial-Cone Decomposition]  Let $\mathbb{P}={\bf N}(P,D_B)$ and let
     $d_0=\text{dim}(\rm{cone}^{\vee}(B))= d-k_0$. Then,  we can make   the  integral  of (\ref{255}) as
 \begin{align} 
& \mathcal{I}^{\rm{sub}}(P,D_B,\lambda)  = \sum_{\mathbb{F}\in  \mathcal{F}^{k_0}(\mathbb{P}_{\rm{for}})   \cup  \mathcal{F}^{k_0}(\mathbb{P}_{\rm{bac}}) }  \sum_{j\in \mathbb{F}^{\vee}\cap\mathbb{Z}^d} \mathcal{I}^{\rm{sub}}_j(\lambda)\label{s40} 
    \end{align} 
 where  $ \mathbb{F}^{\vee}=\rm{cone}(\mathfrak{q}_1,\cdots,\mathfrak{q}_{d_0})$   are oriented and  disjoint simplicial  cones.  Moreover, $ \mathbb{F}^{\vee}\cap \mathbb{Z}^d$ is equipped with the rational coordinates of the basis 
  $\{\mathfrak{q}_i\}_{i=1}^{d_0}\subset \mathbb{Q}^d $:
\begin{align}\label{asan1}
&\big\{\sum_{i=1}^{d_0}\alpha_i \mathfrak{q}_i: (\alpha_i)\in (M_0 \mathbb{Z}_+)^{d_0}\big\} \subset    \mathbb{F}^{\vee}\cap \mathbb{Z}^d \subset\big\{ \sum_{i=1}^{d_0}\alpha_i \mathfrak{q}_i:(\alpha_i)\in  \big(\frac{\mathbb{Z}_+}{M_1}\big)^{d_0}\big\} 
\end{align}
where $M_0, M_1\in \mathbb{N} $ and $1\le|\mathfrak{q}_i|\le 2$. One can replace $\mathcal{F}^{k_0}$ with $\bigcup_{k=k_0}^{d-1}\mathcal{F}^{k}$ in (\ref{s40}).
\end{theorem}
We use (\ref{255}) with some geometric argument to prove Theorem \ref{lemir1} in Section \ref{Sec13}. 
\begin{remark}
To show (\ref{s40}),  we take
    $\mathcal{F}^{k_0} $ as a newly formed set (denoted by  $\mathcal{F}^{k_0}_{\rm{os}}$) by modifying its element $\mathbb{F} $  to have an oriented  and  simplicial  dual face.  In (\ref{s40}), we need to correct $\mathbb{F}^{\vee}$ to be disjoint 
as in Remark \ref{rmk61}.  
\end{remark}
\begin{remark}
There is a  similar   decomposition of $ \mathcal{I}^{\rm{osc}}(P,D_B,\lambda)$ as in (\ref{s40}).
\end{remark}
\begin{remark}\label{rmkk61}  
It suffices to take  
$\alpha_i\in \mathbb{Z}_+$   rather than $ \alpha_i\in (1/M_1)\mathbb{Z}_+ $ in (\ref{asan1}). 
\end{remark}

  \section{Estimate of One Dyadic Piece}\label{Sec7}

Under the normal-crossing condition    (\ref{sv3}), we  estimate  $\mathcal{I}^{\rm{sub}}_{j}(\lambda)$  and  $ \mathcal{I}^{\rm{osc}}_{j}(\lambda)$ of Theorem \ref{lemir1}. 
The estimate is based on the  control of  derivatives of $P$ 
in the following lemma.
  \begin{lemma}[Monomialization]\label{l910}
Suppose that  $\mathbb{F}\in\mathcal{F}(\mathbb{P})$ in Theorem \ref{lemir1}, $\mathfrak{m} \in\mathbb{F}\cap \Lambda(P)$  and  $j\in  \mathbb{F}^{\vee}\cap \mathbb{Z}^d $.  Then for $x\sim {\bf 2}^{-{\bf 0}}$, it holds that     
 \begin{align}
  &  |P({\bf 2}^{-j}x)|  \lesssim\sum_{\mathfrak{n}\in \Lambda(P)} 2^{-j\cdot\mathfrak{n}} \approx 2^{-j\cdot \mathfrak{m}}\ \text{with}\  j\cdot \mathfrak{m}
  = j\cdot \tilde{\mathfrak{m}}\ \ \forall\  \tilde{\mathfrak{m}} \in\mathbb{F}\cap \Lambda(P),\label{jq501}\\
& 2^{c|j|}  \sum_{\mathfrak{n}\in \Lambda(P)\setminus \pi_{\mathfrak{q},r} } 2^{-j\cdot \mathfrak{n}}   \lesssim  2^{-j\cdot \mathfrak{m}} \ \text{for some $c>0$ if  $\mathbb{F}=\pi_{\mathfrak{q},r}\cap \mathbb{P}$.}\label{jq524}
\end{align}
As a consequence,  if $(P,D_B)$ is  normal-crossing of type $[\sigma,\tau]$, then
\begin{align}
 &\sum_{\sigma\le |\alpha|\le \tau}|\partial^{\alpha} (P ({\bf 2}^{-j}x))|  \approx 2^{-j\cdot \mathfrak{m}} \label{aptment}
  \end{align}
  with constants   in $\lesssim$,   depend on the coefficients of $P$, but independent of  $j$ and $x$.  
One can perturbate $x\sim  {\bf 2}^{-{\bf 0}}$   as $x\sim_h {\bf 2}^{-{\bf 0}}$ i.e., $1/h\le |x_\nu|\le  h$ for a fixed $ h\ge 1$. 
 \end{lemma}
   \begin{proof}[Proof of (\ref{jq501})]
  By the triangle inequality, we have 
  \begin{align}\label{9256}
  |P({\bf 2}^{-j}x)|\le  \sum_{\mathfrak{n}\in \Lambda(P)} |c_{\mathfrak{n}}|2^{-j\cdot\mathfrak{n}}|x^{\mathfrak{n}}|\le  \max\{|c_{\mathfrak{n}}|\}  \sum_{\mathfrak{n}\in \Lambda(P)}  2^{-j\cdot\mathfrak{n}}.
  \end{align}
Let $\mathbb{F}\in \mathcal{F}^k(\mathbb{P})$. Then,   there are $\ell $  supporting planes  $\pi_{\mathfrak{q}_\nu,r_\nu}$ such that
   $$\mathbb{F}=\bigcap_{\nu=1}^{\ell} \pi_{\mathfrak{q}_\nu,r_\nu}\cap\mathbb{P}\ \text{and}\ \mathbb{F}^{\vee}=\rm{cone}(\mathfrak{q}_1,\cdots,\mathfrak{q}_{\ell}).$$
Let $\mathfrak{m}\in \mathbb{F}\cap \Lambda(P)$ and $\mathfrak{n}\in \Lambda(P)\subset \mathbb{P}$. Then
 $\mathfrak{m}\in  \pi_{\mathfrak{q}_\nu,r_\nu}$ and $\mathfrak{n}\in \pi_{\mathfrak{q}_\nu,r_\nu}^+$, that is,
$
 \mathfrak{q}_\nu\cdot (\mathfrak{n}-\mathfrak{m})\ge 0\ \text{for all $\nu\in [\ell]$}.$  By this with    $j=\alpha_1\mathfrak{q}_1+\cdots+\alpha_{\ell}\mathfrak{q}_{\ell} \in\mathbb{F}^{\vee}\cap \mathbb{Z}^d$   for  $\alpha_1,\cdots,\alpha_{\ell}\ge0$, 
  one obtain  $ j\cdot (\mathfrak{n}-\mathfrak{m})\ge 0$. Thus $ 2^{- j\cdot \mathfrak{m} }\ge 2^{-  j\cdot \mathfrak{n}}$ in (\ref{9256}), which  implies  $\approx$ of (\ref{jq501}) 
   because   $\Lambda(P)$ is finite.   Since $\mathfrak{m},\tilde{\mathfrak{m}}\in \mathbb{F}\subset \pi_{\mathfrak{q}_\nu,r_\nu}$ and $\mathfrak{q}_\nu\cdot (\mathfrak{m}-\tilde{\mathfrak{m}})=0$ for all $\nu=1,\cdots,{\ell}$. So, $j\cdot (\mathfrak{m}-\tilde{\mathfrak{m}})=0$ in (\ref{9256}). This completes the proof of  (\ref{jq501}).
  \end{proof}

  \begin{proof}[Proof of (\ref{jq524})]
As $j\in \mathbb{F}^{\vee}=\rm{cone}(\mathfrak{q})$,    $j=\alpha\mathfrak{q}$ for $ \alpha>0$. If  $\mathfrak{m}\in  \pi_{\mathfrak{q},r} $ and $\mathfrak{n}\in   (\pi_{\mathfrak{q},r}^+)^{\circ}$,  there is $c(\mathfrak{m},\mathfrak{n})>0 $ such that
 $  \frac{ \mathfrak{q}}{|\mathfrak{q}|}\cdot (\mathfrak{n}-\mathfrak{m})>c(\mathfrak{m},\mathfrak{n})$.  Take $c$ as the minimal $c(\mathfrak{m},\mathfrak{n})$ over $\mathfrak{n}\in\Lambda(P)\setminus \pi_{\mathfrak{q},r}$. Then $\frac{ j}{|j|}\cdot (\mathfrak{n}-\mathfrak{m})>c(\mathfrak{m},\mathfrak{n})\ge c $, that is, $2^{c|j|}   2^{-j\cdot \mathfrak{n}} \le 2^{-j\cdot \mathfrak{m}}$.
  \end{proof}

   \begin{proof}[Proof of (\ref{aptment})]
By the chain rule in differentiation and $x\sim {\bf 1}$,  
 $$\text{LHS  of (\ref{aptment})}\approx  \sum_{\sigma\le |\alpha|\le \tau} |({\bf 2}^{-j}x)^{\alpha}(\partial_{x}^{\alpha}P) ({\bf 2}^{-j}x) |\approx \sum_{\mathfrak{n}\in \Lambda(P)}  2^{-j\cdot \mathfrak{n}} \approx 2^{-j\cdot \mathfrak{m}}$$ where the second $\approx$ is owing to    (\ref{sv3}) and  the last $\approx$ is   due to 
  (\ref{jq501}).  \end{proof}

   \begin{lemma}[Decay Estimates]\label{lembd}
Let $\mathbb{P}={\bf N}(P,D_B)$ and $\mathbb{F}\in \mathcal{F}(\mathbb{P}) $.  Take
  $\mathfrak{m}\in \mathbb{F}\cap \Lambda(P)$ and $j\in \mathbb{Z}^d\cap\mathbb{F}^{\vee}$ in  Theorem \ref{lemir1}. Then
 \begin{align}
   |  \mathcal{I}^{\rm{sub}}_j(\lambda)|&\lesssim  2^{-j\cdot {\bf 1}}\min\left\{ 1, \frac{1}{|\lambda 2^{-j\cdot \mathfrak{m} }|^{ 1/\tau}}\right\}\ \text{if  $\tau_0(P,D_B)=\tau$ as in (\ref{sv3}), } \label{ss66}\\
      |  \mathcal{I}^{\rm{osc}}_j(\lambda)|&\lesssim  2^{-j\cdot {\bf 1}}\min\left\{ 1, \frac{1}{|\lambda 2^{-j\cdot \mathfrak{m} }|^{ 1/\tau}}\right\}\ \text{if  $\tau_1(P,D_B)=\tau$.} \label{sv77}
       \end{align} 
  If $\tau_0(P,D_B) =0$ in (\ref{ss66}), or  $\tau_1(P,D_B) =1$ in (\ref{sv77}),   one can take      $0<\tau\ll 1$. 
  \end{lemma}
 
       \begin{proof}[Proof of (\ref{ss66}) and  (\ref{sv77})]
 By the change of variable, 
 \begin{align*}  
 \mathcal{I}^{\rm{sub}}_{j}(\lambda)&=2^{-j\cdot {\bf 1}}    
  \int \psi(\lambda P({\bf 2}^{-j}x)) \psi_{D_{B,R}}({\bf 2}^{-j}x)\chi(x) dx,\\
\mathcal{I}^{\rm{osc}}_{j}(\lambda)&= 2^{-j\cdot {\bf 1}}    
  \int e^{i \lambda P({\bf 2}^{-j}x)}\Psi_{D_{B,R}}({\bf 2}^{-j}x)\chi(x) dx. 
\end{align*}  
 From    (\ref{aptment}) for  $\mathfrak{m}\in \mathbb{F}\cap \Lambda(P )$ and $j\in \mathbb{Z}^d\cap \mathbb{F}^{\vee}$ with the normal-crossing assumption,
 \begin{align*} 
 \sum_{\sigma \le |\alpha|\le \tau}|\partial^{\alpha} (P ({\bf 2}^{-j}x))| \approx   2^{-j\cdot \mathfrak{m}} 
\ \text{  with  $x\sim {\bf 1}$. }
\end{align*}
This with the rapid decreasing property of $\psi$ yields the desired bound of $\mathcal{I}^{\rm{sub}}_{j}(\lambda)$.
From the hypothesis $\Psi_{D_{B,R}}\in \mathcal{A}(D_B)$ in (\ref{29a}), it follows that  
\begin{align*} 
|\partial_{x}^{\alpha}(\Psi_{D_{B,R}}({\bf 2}^{-j}x)\chi(x))|\lesssim 1 \ \text{in the support of the integral $\mathcal{I}^{\rm{osc}}_{j}(\lambda)$.}
\end{align*}
With this, we apply  the  van der Corput lemma to obtain the   bounds of $\mathcal{I}^{\rm{osc}}_{j}(\lambda)$.  \end{proof}

\section{Summation  Over a Dual Face}\label{Sec8}
We utilize Lemma \ref{lembd} for computing $  \mathcal{I}_{j}(\lambda) $  for $j\in \mathbb{F}^{\vee}$ in (\ref{s40}).
Next, we need to sum $  \mathcal{I}_{j}(\lambda) $ over   $\alpha_1,\cdots,\alpha_{d_0}\in \mathbb{Z}_+$  where $j=\alpha_1\mathfrak{q}_1+\cdots+\alpha_{d_0}\mathfrak{q}_{d_0}$ in (\ref{asan1}).
  \subsection{Summation Formula}    
We  use the following lemma for summing over $\alpha$.
\begin{lemma}\label{L76} 
 Let  $0<\tau<\delta$.  Then for all $\lambda>0$,  we have
\begin{align}
 &\sum_{(\alpha_1,\cdots,\alpha_p)\in \mathbb{Z}_+^p} 2^{-(\alpha_1 +\cdots+\alpha_p)} \min\left\{1, \frac{1}{(\lambda 2^{-(\alpha_1+\cdots+\alpha_p)\delta  })^{1/\tau}  }\right\}\lesssim\frac{ (|\log \lambda   |+1)^{p-1}}{  \lambda^{1/\delta} }, \label{wiki}
\\
 &\sum_{(\alpha_1,\cdots,\alpha_p)\in \mathbb{Z}_+^p} 2^{(\alpha_1 +\cdots+\alpha_p)} \min\left\{1,\frac{1}{(\lambda 2^{(\alpha_1+\cdots+\alpha_p)\delta  })^{1/\tau}  }\right\}\lesssim\frac{ (|\log \lambda  |+1)^{p-1}}{  \lambda^{1/\delta} }.  \label{wiki2}
\end{align}
Let  $ \tau\ge \delta$.  Then it holds   
\begin{align}\label{wik}
& \text{LHS of (\ref{wiki})}\lesssim \frac{ (|\log \lambda   |+1)^{p(\tau)}}{  \lambda^{1/\tau} }\ \text{for}\ p(\tau)=\begin{cases} p\  \text{if $\tau=\delta$}\\ 0\ \text{if $\tau>\delta$.}  
 \end{cases}\\
&\text{$\text{LHS of (\ref{wiki2})}=\infty$.\nonumber}
\end{align}
 
  \end{lemma}
\begin{proof}[Proof of (\ref{wiki})]
If $p=1$, one can obtain (\ref{wiki})  as
\begin{align}
 \sum_{\alpha\in \mathbb{Z}_+} 2^{-\alpha} \min\left\{1, \frac{1}{(\lambda 2^{-\alpha \delta })^{1/\tau}  }\right\}&\approx\sum_{\lambda 2^{-\alpha\delta}\le 1}2^{-\alpha}+ \sum_{\lambda 2^{-\alpha\delta}\ge 1}\frac{2^{\alpha(\frac{\delta}{\tau}-1)} }{\lambda^{1/\tau} } \approx  \lambda^{-1/\delta}.\label{siii}
\end{align}
Next, to show the case $p\ge 2$ of (\ref{wiki}),   split the  sum in (\ref{wiki})   into the two parts
\begin{align} 
\begin{cases}
& \sum_{2^{\alpha_1+\cdots+\alpha_p}\ge (\lambda+\lambda^{-1})^{2/\delta}}2^{-(\alpha_1+\cdots+\alpha_p)}\\
 &   \sum_{2^{\alpha_1+\cdots+\alpha_p}\le  (\lambda+\lambda^{-1})^{2/\delta}}2^{-(\alpha_1+\cdots+\alpha_p)}\min\left\{1, \frac{1}{(\lambda 2^{-(\alpha_1+\cdots+\alpha_p) \delta})^{1/\tau}  }\right\}
 \end{cases}\label{844g}
\end{align}
over  the  indices $\alpha_i\in\mathbb{Z}_+$ for $i\in [p]$. The first sum in (\ref{844g}) is bounded by 
\begin{align*}
\sum_{\alpha_i\in \mathbb{Z}_+\ \text{for}\ i=1,\cdots,p} 2^{-\frac{1}{2}(\alpha_1+\cdots+\alpha_p)}  (\lambda+\lambda^{-1})^{- 1/\delta} 
\lesssim
 \lambda^{- 1/\delta}\le \frac{ (|\log \lambda  |+1)^{p-1}}{  \lambda^{1/\delta} }.
 \end{align*}
The second sum  (\ref{844g}), after coordinate change $\alpha_1+\cdots+\alpha_p=\alpha$,  becomes
\begin{align*}
 &\sum_{  0\le \alpha_1+\cdots+\alpha_{p-1}\le \frac{2}{\delta}\log_2 (\lambda+\lambda^{-1}) } \left(\sum_{\alpha\in\mathbb{Z}_+
}2^{-\alpha}\min\left\{1, \frac{1}{(\lambda 2^{-\alpha \delta})^{1/\tau}  }\right\}\right)\approx \frac{(|\log   \lambda|+1 )^{p-1}}{\lambda^{1/\delta}} \end{align*}
because of (\ref{siii}) and $\log(\lambda +\lambda^{-1})\approx |\log \lambda|+1$.   This with (\ref{siii})   yields   (\ref{wiki}).   
\end{proof}
 \begin{proof}[Proof of (\ref{wiki2})]
 Let $\alpha=\alpha_1+\cdots+\alpha_p$ and  $A:=2^{\frac{\delta}{\tau}-1}>1$ and write the sum 
\begin{align*}
\lambda^{-1/\tau}\min\left[2^{\alpha}\lambda^{1/\tau},2^{-\alpha \left(  \frac{\delta}{\tau}-1\right)}  \right]&=
\lambda^{-1/\tau} A^{-\alpha}\min\left\{ \left(\lambda^{-1}\left(A^{-1} A^{- \frac{1}{(\frac{\delta}{\tau}-1)}}\right)^{\tau\alpha}\right)^{-1/\tau},1 \right\}\\
&=\lambda^{-1/\tau} A^{-\alpha}\min\left\{ \left(\lambda^{-1} A^{-\alpha \tilde{\delta}} \right)^{-1/\tau},1 \right\} 
\end{align*}
with $\tilde{\delta}= \tau\left(1+\frac{1}{\frac{\delta}{\tau}-1}\right)  $. Replace $2$  in (\ref{wiki})  by $A$ and sum the last line over  $\alpha_i$  in $\mathbb{Z}_+$ to have the bound 
$\lambda^{-1/\tau}(\lambda^{-1})^{-1/\tilde{\delta}}(|\log\lambda^{-1}|+1)^{p-1}=\lambda^{-1/\delta}(|\log\lambda|+1)^{p-1}$. 
 \end{proof} 
\begin{proof}[Proof of  (\ref{wik})]
The  case $\tau> \delta$  follows from the geometric sum, and  the  case $\tau=\delta$  from $\sum_{2^{\alpha_1+\cdots+\alpha_p}<\lambda} 1\approx ( \log \lambda)^p$. Next,  use $\sum_{\alpha\in \mathbb{Z}_+} 2^{\alpha}\min\{1, (\lambda 2^{\alpha})^{-1}\}=\infty$.  \end{proof}

 \subsection{Forward Face Sum} 
We estimate $\sum_{j\in \mathbb{F}^{\vee}\cap\mathbb{Z}^d}   \mathcal{I}^{\rm{sub}}_{j}(\lambda)$ in (\ref{s40}) for $\mathbb{F}\in \mathcal{F}(\mathbb{P}_{\rm{for}})$.
      \begin{proposition}[Forward Face Estimate]\label{propfor}
Let $\mathbb{F}=\bigcap_{i=1}^{d_0} \pi_{\mathfrak{q}_i,r_i} \in\mathcal{F}^{k_0}(\mathbb{P}_{\rm{for}})$
 in  Theorem \ref{lemir1}.
Suppose that    all $\pi_{ \mathfrak{q}_i,r_i}^+$  are across-diagonal and satisfying
 $$
  \sharp\{\pi_{\mathfrak{q}_i,r_i}  :
  d(\pi^+_{\mathfrak{q}_i,r_i})=d(\mathbb{P}_{\rm{for}})\}_{i=1}^{{d_0}}=p \ge 0 .$$ 
  \begin{itemize}
   \item[(1)]
   Let $\lambda>1$.  If $\delta_{\rm{for}} >0$ and $P$ is of type $[0,\tau]$ with $\tau<\delta_{\rm{for}} $ in  (\ref{ss66}), then,
    \begin{align*}
 \qquad\sum_{j\in \mathbb{F}^{\vee}\cap\mathbb{Z}^d}   \mathcal{I}^{\rm{sub}}_{j}(\lambda)   & \le C  
    \lambda^{-1/\delta_{\rm{for}} }(|\log \lambda|+1)^{p-1}\ \text{for $C$ independent of $\lambda$.}
\end{align*}
One can show it has  the lower bound   
\begin{align*}
 \sum_{j\in \mathbb{F}^{\vee}\cap\mathbb{Z}^d}  \mathcal{I}^{\rm{sub}}_{j}(\lambda) &\ge C^{-1}
 \lambda^{-1/\delta_{\rm{for}}}(|\log \lambda|+1)^{p-1}\ \text{if $p\ge 1$}. 
 \end{align*}
     \item[(2)]
Let $0< \lambda\le  1$.  Then   there is $C_\epsilon>0$,
 \begin{align*} 
 \sum_{j\in \mathbb{F}^{\vee}\cap\mathbb{Z}^d}  \mathcal{I}^{\rm{sub}}_{j}(\lambda)   \le C_\epsilon   \lambda^{-\epsilon}\ \text{for an arbitrary small $\epsilon>0$.}
 \end{align*} 
 \end{itemize} 
\end{proposition}

   \begin{proof}[Proof of Proposition \ref{propfor}]
Since $\mathbb{F}^{\vee}=\rm{cone}(\{\mathfrak{q}_i\}_{i=1}^{d_0})$ is simplicial  with
$\text{rank}(\{\mathfrak{q}_{i=1}^{d_0})\}=d_0$,  one can express $j\in\mathbb{F}^{\vee}\cap\mathbb{Z}^d$ in  $ \sum_{ j\in\mathbb{F}^{\vee}\cap\mathbb{Z}^d}\mathcal{I}^{\rm{sub}}_j(\lambda)$ in    Theorem \ref{lemir1}  as  $$j=\alpha_1 \mathfrak{q}_1+\cdots+
\alpha_{d_0}\mathfrak{q}_{d_0}\ \text{  with $\alpha_i\in  \mathbb{Z}_+ $.}$$ 
If $P$ is a polynomial, then   $-\infty\le d(\pi_{\mathfrak{q}_i,r_i}^+)\le d(\mathbb{P}_{\rm{for}}) =\delta_{\rm{for}} $
   in (\ref{jmq9}) and (\ref{m18}).  
  With     $0\le p\le d_0$,  we rearrange $\{\pi_{\mathfrak{q}_i,r_i}\}_{i=1}^{d_0}$ of the forward     supporting planes:
\begin{align}\label{jq106}
& \mathfrak{q}_i\cdot {\bf 1}>0\ \text{and}\ d(\pi^+_{\mathfrak{q}_i,r_i})=\delta_{\rm{for}} \ \text{for}\ i=1,\cdots,p,\nonumber\\
&\mathfrak{q}_i\cdot {\bf 1}>0\ \text{and}\  d(\pi^+_{\mathfrak{q}_i,r_i})<\delta_{\rm{for}}\ \text{for}\ i=p+1,\cdots,n,\\
&\mathfrak{q}_i\cdot {\bf 1}=0\ \text{for}\ i=n+1,\cdots,{d_0}\ \text{where}\ 0\le p\le n\le{d_0}.\nonumber
    \end{align}
Use  (\ref{ss66}) to have the upper bound of    $ \sum_{j\in  \mathbb{F}^{\vee}\cap\mathbb{Z}^d}  \mathcal{I}^{\rm{sub}}_{j}(\lambda)$   in terms of the above $\alpha$:\begin{align}\label{ssv77}
 \sum_{(\alpha_1,\cdots,\alpha_{d_0})\in  \mathbb{Z}_+^{d_0} }  2^{-\left(\alpha_1\mathfrak{q}_1+\cdots+\alpha_{n} \mathfrak{q}_{n}\right)\cdot {\bf 1}  }\min\left\{ 1, \frac{1}{\lambda^{\frac{1}{\tau}}2^{-(\alpha_1 \mathfrak{q}_1+\cdots+\alpha_{d_0}\mathfrak{q}_{d_0})\cdot \frac{\mathfrak{m} }{\tau}  }}\right\}
    \end{align}
where $\mathfrak{m}\in \mathbb{F}\cap \Lambda(P)$, that is,
 $\mathfrak{m}\in   \bigcap_{i=1}^{d_0}\pi_{\mathfrak{q}_i,r_i}$ in  (\ref{jq106}). Then it follows that  
 \begin{align}
  \mathfrak{q}_i\cdot \mathfrak{m}&=\begin{cases} d(\pi^+_{\mathfrak{q}_i,r_i}) \mathfrak{q}_i\cdot {\bf 1} \  \text{from  (\ref{pm12})  if}\     i=1,\cdots,n\ \text{and},   \label{sv101}\\
 r_i<0    \ \text{from  (2) of Lemma \ref{lems01}  if}\  i=n+1,\cdots,{d_0}. 
\end{cases}
   \end{align}
By using the second case $\mathfrak{q}_i\cdot \mathfrak{m}=r_i<0$ for $i=n+1,\cdots,d_0$ in (\ref{sv101}),   we put
\begin{align}
  \lambda(\alpha)&:=\lambda 2^{-\sum_{i=n+1}^{d_0} \mathfrak{q}_i\cdot \mathfrak{m} \alpha_i  }=\lambda 2^{-\sum_{i=n+1}^{d_0}  r_i\alpha_i  }=\lambda 2^{ \sum_{i=n+1}^{d_0} | r_i|\alpha_i  }\  \text{for $|r_i|>0$. } \label{w99}
  \end{align}
 Then by inserting (\ref{sv101}),(\ref{w99}) into   (\ref{ssv77}), one can rewrite it as 
   \begin{align}\label{s88}
& \sum_{\alpha_1,\cdots,\alpha_{d_0}}  2^{-\left(\alpha_1\mathfrak{q}_1+\cdots+\alpha_n \mathfrak{q}_n\right)\cdot {\bf 1}  }\min\left\{ 1, \frac{2^{\sum_{i=p+1}^n \alpha_i \mathfrak{q}_i\cdot {\bf 1}\frac{d(\pi_{\mathfrak{q}_i,r_i}^+)  }{\tau}  }     }{\left[  \lambda(\alpha)2^{ -\sum_{i=1}^p \alpha_i \mathfrak{q}_i\cdot {\bf 1}  \delta_{\rm{for}} }\right]^{\frac{1}{\tau}}  }\right\}.   \end{align}
For   $i=p+1,\cdots,n$, put  $\delta_i=\max\{0,d(\pi_{\mathfrak{q}_i,r_i}^+)\}$. By (\ref{jq106}) and  $\tau<\delta_{\rm{for}}$,  assume that   
\begin{align}\label{jmq17}
  \max\{\delta_i \}_{i=p+1}^n  <   \tau 
\end{align}  keeping $\tau<\delta_{\rm{for}}$  in (\ref{ss66}). Then one can see that (\ref{s88})  is bounded by
\begin{align*} 
\sum_{\alpha_1,\cdots,\alpha_{p} } 2^{- (\alpha_1 \mathfrak{q}_1\cdot {\bf 1}+\cdots+\alpha_p\mathfrak{q}_p\cdot {\bf 1}) } &  \sum_{\alpha_{p+1},\cdots,\alpha_{n} } 2^{- \left(\alpha_{p+1}\mathfrak{q}_{p+1}\cdot {\textbf 1}\big(1-\frac{\delta_{p+1}  }{\tau} \big)+\cdots+\alpha_n \mathfrak{q}_n\cdot {\textbf 1} \big(1-\frac{\delta_n }{\tau}  \big)\right)}  
 \\
&\times
 \sum_{\alpha_{n+1},\cdots,\alpha_{d_0}}  \min\left\{ 1, \frac{1 }{\left[  \lambda(\alpha) 2^{ -\sum_{i=1}^p \alpha_i \mathfrak{q}_i\cdot {\textbf 1}  \delta_{\rm{for}}    }\right]^{\frac{1}{\tau}}  }\right\}. 
 \end{align*}
Summing $\sum_{\alpha_{p+1},\cdots,\alpha_{n} } $ due to $\frac{\delta_i}{\tau}<1 $ in (\ref{jmq17}), majorize the above  by  $S(\lambda)$ where
\begin{align} \label{953}
 S(\lambda) 
 := \sum_{\alpha_1,\cdots,\alpha_{p},\alpha_{n+1},\cdots,\alpha_{d_0} } 2^{-\sum_{i=1}^p \alpha_i \mathfrak{q}_i\cdot {\bf 1} } \min\left[ 1, \frac{1}{\left( \lambda(\alpha) 2^{- \sum_{i=1}^p \alpha_i \mathfrak{q}_i\cdot {\bf 1} \delta_{\rm{for}} }\right)^{ \frac{1}{\tau}}  }\right].
    \end{align}   
 \begin{proof}[Proof    of (1)  for the case $\lambda\ge 1$]     Regard  $\alpha_i\mathfrak{q}_i\cdot {\bf 1}$   and   $  \lambda(\alpha)$  in $(\ref{953})$ as $\alpha_i$ and $\lambda$ in  (\ref{wiki}).
 Then,  apply  (\ref{wiki}) for the sum $  \sum_{\alpha_1,\cdots,\alpha_{p} } $ of (\ref{953}) with (\ref{w99}) to obtain that
  \begin{align*} 
S(\lambda) &\lesssim \sum_{\alpha_{n+1},\cdots,\alpha_{d_0}}  \frac{\left(\left|\log     \lambda(\alpha)\right|+1\right)^{p-1} }{  \lambda(\alpha)^{1/\delta_{\rm{for}}}}
 \lesssim  \frac{(|\log \lambda|+1)^{p-1}}{ \lambda^{1/\delta_{\rm{for}}}  }
\end{align*}
where the second   follows from  $\left|\log     \lambda(\alpha)\right|\le |\log\lambda |+\sum_{i=n+1}^{d_0} \alpha_i $ and $|r_i|>0$ in (\ref{w99}). If $p=0$, then in (\ref{953}), it holds that
 $S(\lambda) \le \sum_{\alpha_{n+1},\cdots,\alpha_{d_0}}  
  | \lambda(\alpha)|^{ -\frac{1}{\tau}}  
   \lesssim  \lambda^{-\frac{1}{\tau}}\lesssim \lambda^{-1/\delta_{\rm{for}}-\epsilon}$ for $\lambda>1$ giving a better bound than $ \lambda^{-1/\delta_{\rm{for}}} (|\log \lambda|+1)^{-1} $.\end{proof}
Next,  we show (2) and the reverse inequality of  (1).

 \begin{proof}[Proof of (2)  for the case $0<\lambda<1$]
One can assume that  $1/\tau$ in (\ref{953}) satisfies $0< 1/\tau \ll 1$. This combined with $\mathfrak{q}_i\cdot {\bf 1}>0$ for $1\le i\le p$, enables us to   estimate
 \begin{align*} 
 S(\lambda) \lesssim\sum_{\alpha_1,\cdots,\alpha_{p},\alpha_{n+1},\cdots,\alpha_{d_0}}    \frac{2^{-(\alpha_1 \mathfrak{q}_1+\cdots+\alpha_p\mathfrak{q}_p)\cdot {\bf 1}}2^{\frac{\delta_{\rm{for}}}{\tau}(\alpha_1 \mathfrak{q}_1+\cdots+\alpha_p\mathfrak{q}_p)\cdot {\bf 1}}}{  \lambda(\alpha)^{1/\tau} } \lesssim \frac{1}{\lambda^{1/\tau}} 
    \end{align*}
    because  
    $1\ll \tau$ makes  $\delta_{\rm{for}}<\tau$.  This gives the desired bound of (2).
    \end{proof}
 
 \begin{proof}[Proof  of the lower bound   in (1)]
We shall find a   lower bound  of 
 \begin{align*}   \sum_{ j\in \mathbb{F}\cap\mathbb{Z}^d}\mathcal{I}^{\rm{sub}}_j(\lambda)= \sum_{ j\in  \mathbb{F}^{\vee}\cap\mathbb{Z}^d  } 2^{-j\cdot {\bf 1}}\int \psi(\lambda P({\bf 2}^{-j}x))\psi_{D_B}({\bf 2}^{-j}x)\chi(x)dx
 \end{align*}
under the condition $   \delta_{\rm{for}}>0$ and  $p\ge 1$  in (\ref{jq106}).
By (\ref{jq501}),  it holds that
\begin{align*}
  |P({\bf 2}^{-j}x)|  \le C 2^{-j\cdot \mathfrak{m}}\ \text{for}\  j\in \mathbb{F}^{\vee}\cap \mathbb{Z}^d\ \text{and}\ \mathfrak{m}\in \mathbb{F}\cap\Lambda(P)  
  \end{align*}
which implies  $ \psi\left( \lambda  P({\bf 2}^{-j}x) \right)=1$ whenever $|C\lambda 2^{-j\cdot \mathfrak{m}}|<1/10$.   Hence, one has
\begin{align} \label{jq242}
  \sum_{ j\in \mathbb{F}^{\vee} \cap\mathbb{Z}^d }\mathcal{I}^{\rm{sub}}_j(\lambda) \gtrsim  \sum_{(\alpha_i)\in  (M_0 \mathbb{Z}_+)^{p}\cap A } 2^{-(\alpha_1\mathfrak{q}_1\cdot{\bf 1} +\cdots+\alpha_p \mathfrak{q}_p\cdot {\bf 1}) }.
  \end{align}
where  $A: =\{(\alpha_i):  C\lambda 2^{-(\alpha_1\mathfrak{q}_1\cdot{\bf 1} +\cdots+\alpha_p \mathfrak{q}_p\cdot {\bf 1}) \delta_{\rm{for}} }\le 1/10\}$ because
one can  restrict  $j=\alpha_1\mathfrak{q}_1+\cdots+\alpha_p\mathfrak{q}_p$  by taking $\alpha_{p+1}=\cdots=\alpha_{d_0}=0$ so that
$$2^{-j\cdot \mathfrak{m}} = 2^{-(\alpha_1\mathfrak{q}_1\cdot{\bf 1} +\cdots+\alpha_p \mathfrak{q}_p\cdot {\bf 1})\delta_{\rm{for}}  }\ \text{with $\mathfrak{q}_i\cdot {\bf 1}>0$ in  (\ref{jq106}) and (\ref{sv101})}. $$ 
Reset
$
   \gamma_i=  \alpha_i\mathfrak{q}_i\cdot {\bf 1}>0\ \text{ for $i\in [p]$,} 
$  and rewrite  RHS of (\ref{jq242})  in terms of $\gamma_i$'s:
\begin{align*} 
 \sum_{(\gamma_1,\cdots,\gamma_p)\in (N\mathbb{Z}_+)^p\cap A'} 2^{-(\gamma_1 +\cdots+\gamma_p)}\ \text{ with $A'=\{  (\gamma_i):   2^{-(\gamma_1+\cdots+\gamma_p)} \le  (10 C\lambda )^{-1/\delta_{\rm{for}}}   \}$}
 \end{align*}
for some $N\in\mathbb{N}$,
giving the lower bound $\lambda^{ \frac{-1}{\delta_{\rm{for}}}}  (|\log  \lambda|+1)^{p-1}$.  
 \end{proof}
Therefore, we finish the proof of Proposition \ref{propfor}.  
    \end{proof}
 We next  show  (\ref{es0}) saying  the case $\tau\ge \delta_{\rm{for}}$ (which includes   $\tau=\delta_{\rm{for}}=0$).
\begin{proof}[Proof of (\ref{es0})]
Suppose that $\tau:=\tau_0(P,D)\ge \delta_{\rm{for}}\ge 0$ with $\lambda\ge 1$. Then  
\begin{align}\label{rmk81}
     \sum_{j\in \mathbb{F}^{\vee}\cap\mathbb{Z}^d}   \mathcal{I}^{\rm{sub}}_{j}(\lambda)  &  \le C 
     \lambda^{-1/\tau}   (|\log \lambda|+1)^{p(\tau)}   \ \text{if}\  p(\tau)  = \begin{cases}0 \quad \text{if  $ \tau >\delta_{\rm{for}}$}\\
 p   \quad  \text{if   $ \tau =\delta_{\rm{for}}$.}\end{cases}
   \end{align}
 This with (2) of Proposition \ref{th103} yields (\ref{es0}). \end{proof}

\begin{proof}[Proof of (\ref{rmk81})]
{\bf Case 1}  $\delta_{\rm{for}} >0$. We  have  (\ref{jmq17}),(\ref{953}) even if $ \tau\ge \delta_{\rm{for}}=d(\mathbb{P}_{\rm{for}})>0$. 
 Regard  $\alpha_i\mathfrak{q}_i\cdot {\bf 1}$   and   $  \lambda(\alpha) $  in $(\ref{953})$ as $\alpha_i$ and $\lambda$ in  (\ref{wik}). Then, we can apply  (\ref{wik}) for $  \sum_{\alpha_1,\cdots,\alpha_{p} } $ of  (\ref{953}) to obtain 
 \begin{align*} 
S(\lambda) &\lesssim \sum_{\alpha_{n+1},\cdots,\alpha_{d_0}}  \frac{\left(\left|\log     \lambda(\alpha)\right|+1\right)^{p(\tau)} }{  \lambda(\alpha)^{1/\tau}} \lesssim  \frac{(|\log \lambda|+1)^{p(\tau)}}{ \lambda^{1/\tau}  }.
\end{align*}
where the second inequality follows from  (\ref{w99}).   \\
{\bf Case 2} $\delta_{\rm{for}}=0$. Note $d(\mathbb{P}_{\rm{for}})\le \delta_{\rm{for}}= 0$.  We still have  (\ref{jmq17})  for $\tau\ge \delta_{\rm{for}}=0$.  Then if $\tau>\delta_{\rm{for}}=0$, we have the same estimate as above. But if $\tau=\delta_{\rm{for}}=0$, then we can take $\tau\ll 1$ in  (\ref{ss66})  and  (\ref{953}), showing
 $S(\lambda) \lesssim  \frac{1}{\lambda^{1/\tau}} $.
\end{proof}

\subsection{Backward Face Sum}
We treat  the backward  faces  similarly.

\begin{proposition}\label{th103}[Backward Face]
Let  
 $\mathbb{F}=\bigcap_{i=1}^{d_0} \pi_{\mathfrak{q}_i,r_i} \in \mathcal{F}^{k_0}(\mathbb{P}_{\rm{bac}})$  in (\ref{s40}). Suppose  that
  all $\pi_{\mathfrak{q}_i,r_i}^+$ are across-diagonal   
 and $
  \sharp \{\pi_{\mathfrak{q}_i,r_i}:
  d(\pi^+_{\mathfrak{q}_i,r_i})=\delta_{\rm{bac}} \}_{i=1}^{d_0}=p 
$.  
   \begin{itemize}
  \item[(1)]
 Let $0< \lambda\le 1$. If  $\tau_0(P,D)=\tau\in   [0,\delta_{\rm{bac}})$ in  (\ref{ss66}), then there is $C>0$:
    \begin{align*} 
 \sum_{j\in \mathbb{F}^{\vee}\cap\mathbb{Z}^d}  \mathcal{I}^{\rm{sub}}_{j}(\lambda)    &\le C    \lambda^{-1/\delta_{\rm{bac}}}(|\log \lambda|+1)^{p-1}.
 \end{align*}
If $p\ge 1$, then there is $0<b\le 1$ such that for all $ \lambda\in (0,b]$,
\begin{align*}
\sum_{j\in \mathbb{F}^{\vee}\cap\mathbb{Z}^d}  \mathcal{I}^{\rm{sub}}_{j}(\lambda)    \ge
C^{-1}\lambda^{-1/\delta_{\rm{bac}}}(|\log \lambda|+1)^{p-1}.
 \end{align*}
  \item[(2)]
 Let $ \lambda> 1$.  If $\tau_0(P,D)=\tau\in   [0,\delta_{\rm{bac}})$ in  (\ref{ss66}), then there is $C>0$:
 \begin{align*} 
 \sum_{j\in \mathbb{F}^{\vee}\cap\mathbb{Z}^d}  \mathcal{I}^{\rm{sub}}_{j}(\lambda)\le C    \lambda^{-\frac{1}{\tau}}  \  \text{which is $O(\lambda^{-(1/\delta_{\rm{for}}+\epsilon)})$ if $\tau\in [0,\delta_{\rm{for}})$}.
 \end{align*}
 \end{itemize}
\end{proposition}
\begin{proof}[Proof of Proposition \ref{th103}]
Observe that $0< \delta_{\rm{bac}}= d(\mathbb{P}_{\rm{bac}})\le d(\pi_{\mathfrak{q}_i,r_i}^+)$ as in (\ref{m18})  
  \footnote{If $\delta_{\rm{bac}}=\infty$, then $\delta_{\rm{bac}}\le d(\pi_{\mathfrak{q}_i,r_i}^+)=\infty$ for all $i$ of  $\mathbb{F}=\bigcap_{i=1}^{d_0}\pi_{\mathfrak{q}_i,r_i} $.  This with Definition \ref{df5.2} implies $\mathfrak{q}_i\cdot {\bf 1}=0$ for all $i \le d_0$. So, $n=0$ in   (\ref{jq107})-(\ref{j99}) so that $(\ref{ssv78}) =\sum_{(\alpha_{n+1},\cdots,\alpha_{d_0})  }   \lesssim \frac{1}{\lambda^{1/\tau}}$.}  for $i\in [d_0]$. Then,  rearrange   $\pi_{\mathfrak{q}_i,r_i}$  forming a backward  face $ \mathbb{F}=\bigcap_{i=1}^{d_0} \pi_{\mathfrak{q}_i,r_i}$ as
\begin{align}\label{jq107}
& \mathfrak{q}_i\cdot {\bf 1}<0\ \text{and}\   d (\pi_{\mathfrak{q}_i,r_i}^+)=\delta_{\rm{bac}} \ \text{for}\ i=1,\cdots,p,\nonumber\\
&\mathfrak{q}_i\cdot {\bf 1}<0\ \text{and}\  d(\pi_{\mathfrak{q}_i,r_i}^+)>\delta_{\rm{bac}}\ \text{for}\ i=p+1,\cdots,n,\\
&\mathfrak{q}_i\cdot {\bf 1}=0\ \text{for}\ i=n+1,\cdots,{d_0}.\nonumber
    \end{align}
With  $\tau$ in (\ref{ss66}) and $\mathfrak{m}\in \mathbb{F}\cap \Lambda(P)$,
estimate   $\sum_{ j=\alpha_1 \mathfrak{q}_1+\cdots+\alpha_{d_0}\mathfrak{q}_{d_0}\in\mathbb{F}^{\vee}\cap\mathbb{Z}^d}\mathcal{I}^{\rm{sub}}_j(\lambda)$   by
 \begin{align}\label{ssv78}
  \sum_{ \alpha_1,\cdots,\alpha_{d_0}\in  \mathbb{Z}_+  }  2^{-(\alpha_1 \mathfrak{q}_1+\cdots+\alpha_n\mathfrak{q}_n)\cdot {\bf 1}}\min\left\{\frac{1}{\lambda^{1/\tau}2^{-(\alpha_1 \mathfrak{q}_1+\cdots+\alpha_{d_0}\mathfrak{q}_{d_0})\cdot \mathfrak{m}/\tau}},1\right\}.
    \end{align}
From  (\ref{pm12}) and (2) of Lemma \ref{lems01}, we have for each $\mathfrak{q}_i$ in $ (\ref{jq107})$,,
 \begin{align}\label{j99}
 \begin{cases} 
 \mathfrak{q}_i\cdot \mathfrak{m} =\mathfrak{q}_i\cdot {\bf 1} d (\pi_{\mathfrak{q}_i,r_i}^+) \ \text{if}\     i=1,\cdots,n, \ \text{for the case $\mathfrak{q}_i\cdot {\bf 1}< 0$} \\
   \mathfrak{q}_i\cdot \mathfrak{m}=r_i <0  \ \text{if}\  i=n+1,\cdots,{d_0},\ \text{for the case $\mathfrak{q}_i\cdot {\bf 1}= 0$}. 
   \end{cases}
 \end{align} 
In  (\ref{ssv78}),(\ref{j99}), we rewrite   $ -\alpha_i\mathfrak{q}_i\cdot {\bf 1}\ge 0\ \text{as}\ \alpha_i \in \frac{1}{N}\mathbb{Z}_+\ \text{for}\ i=1,\cdots,n $  and  set
\begin{align}\label{w100}
 \lambda(\alpha)&:=\lambda 2^{-(\alpha_{p+1}\mathfrak{q}_{p+1}+\cdots+\alpha_{d_0}\mathfrak{q}_{d_0})\cdot \mathfrak{m}  }  \\
 &= \lambda 2^{\alpha_{p+1}d (\pi_{\mathfrak{q}_{p+1},r_{p+1}}^+) +\cdots+\alpha_{n}d (\pi_{\mathfrak{q}_n,r_n}^+) }  2^{ \alpha_{n+1}|r_{n+1}|+\cdots+\alpha_{d_0}|r_{d_0}|}.\nonumber
 \end{align}
Then  we rewrite   (\ref{ssv78})  as
 \begin{align}  \label{s19}
S(\lambda):=\sum_{\alpha_{p+1},\cdots,\alpha_{d_0}} 2^{(\alpha_{p+1}+\cdots+\alpha_{n})}      \sum_{\alpha_1,\cdots,\alpha_{p}} 2^{(\alpha_1+\cdots+\alpha_p)} 
\min\left\{ \frac{1} {\left( \lambda(\alpha) 2^{(\alpha_1+\cdots+\alpha_p)\delta_{\rm{bac}} } \right)^{\frac{1}{\tau}}  }, 1
 \right\}.
\end{align}
\begin{proof}[Proof of (2) in Proposition \ref{th103}]
Let $\lambda\ge 1$. Then   the  condition   $0< \tau < \delta_{\rm{bac}} <d (\pi_{\mathfrak{q}_i,r_i}^+) $ for $i=  p+1,\cdots,n$ in (\ref{w100})  implies
$
S(\lambda) \lesssim   \lambda^{-1/\tau}
$ in (\ref{s19})
which yields (2). If $\tau=0$, then take $0< \tau\ll 1$  in $S(\lambda)$.
\end{proof}
\begin{proof}[Proof for $\lesssim$  (1) in Proposition \ref{th103}]
Let $0<\lambda<1$.  If $p\ge 1$,   by
  regarding   $ \lambda(\alpha) $ in $S(\lambda)$ as $\lambda$,   apply  (\ref{wiki2})  with $\tau<\delta_{\rm{bac}}$ for the inner sum over $(\alpha_1,\cdots,\alpha_p)$ in (\ref{s19}),   \begin{align*} 
S(\lambda)&\lesssim \sum_{\alpha_{p+1},\cdots,\alpha_{d_0}} 2^{(\alpha_{p+1}+\cdots+\alpha_{n})}   \lambda(\alpha)^{-1/\delta_{\rm{bac}}}\left(\left|\log  \lambda(\alpha)   \right|+1\right)^{p-1}.
\end{align*}
For $\lambda(\alpha)$ in  (\ref{w100}), put  $\delta_i=d (\pi_{\mathfrak{q}_{i},r_{i}}^+) >\delta_{\rm{bac}}$ for $i=p+1,\cdots,n$. Then in the above,
\begin{align*}
RHS&\lesssim\sum_{\alpha_{p+1},\cdots,\alpha_{d_0}} \frac{\lambda^{-1/\delta_{\rm{bac}}}  \left( |\log \lambda|^{p-1}+|\sum_{i=p+1}^{d_0} \alpha_i|^{p-1}  \right) } { 2^{\alpha_{p+1}\left(\frac{\delta_{p+1}}{\delta_{\rm{bac}}}-1\right) +\cdots+\alpha_{n}\left( \frac{\delta_{n} }{\delta_{\rm{bac}}}-1 \right) }  2^{ \alpha_{n+1}\frac{|r_{n+1}|}{\delta_{\rm{bac}}}+\cdots+ \alpha_{d_0}\frac{ |r_{d_0}|}{\delta_{\rm{bac}}}  \ }   }\\
&\lesssim \frac{(|\log \lambda|+1)^{p-1} }{\lambda^{1/\delta_{\rm{bac}}} }.
 \end{align*}
 If $p=0$,    take $\tau$ as $\delta_{\rm{bac}}<\tau<\delta_i$ for $i=p+1,\cdots,n$ of  $S(\lambda)$ in (\ref{s19}) to have
$$  S(\lambda) \lesssim \sum_{ \alpha_{p+1},\cdots,\alpha_{d_0}\in  \mathbb{Z}_+}  \frac{ 2^{ (\alpha_{p+1}+\cdots+\alpha_n)}  } {\lambda^{\frac{1}{\tau}}  2^{(\alpha_{n+1}|r_{n+1}|+\cdots+\alpha_{d_0}|r_{d_0}|)\frac{1}{\tau}}     2^{ (\alpha_{p+1}\frac{d_{p+1}}{s}+\cdots+ \alpha_{n}\frac{d_n}{\tau}) }     } \lesssim  \lambda^{-1/\tau}
 $$
with $1/\tau=  1/\delta_{\rm{bac}}-\epsilon$. We proved  $\lesssim$ of
(1) in  Proposition \ref{th103}.
\end{proof}

  \begin{proof}[Proof for the reverse inequality in  (1) of Proposition \ref{th103}] 
 Let $\lambda<b:= \frac{1}{100C}$. As   (\ref{jq242}), with  $\mathfrak{q}_i\cdot {\bf 1}<0$ for $i=1,\cdots,p$ in (\ref{jq107}), 
 \begin{align} \label{jq243}
  \sum_{ j\in \mathbb{F}\cap\mathbb{Z}^d}\mathcal{I}^{\rm{sub}}_j(\lambda)&\gtrsim \sum_{(\alpha_i)\in (M_0 \mathbb{Z}_+)^p\cap A } 2^{-(\alpha_1\mathfrak{q}_1+\cdots+\alpha_p\mathfrak{q}_p)\cdot {\bf 1}}  \approx  \lambda^{ \frac{-1}{  \delta_{\rm{bac}}  }}  (|\log  \lambda|+1)^{p-1} 
    \end{align}
due to   $A =\{(\alpha_i) : \lambda   2^{-(\alpha_1\mathfrak{q}_1\cdot{\bf 1} +\cdots+\alpha_p \mathfrak{q}_p\cdot {\bf 1})\delta_{\rm{bac}} } <1/(10C)\}$.
\end{proof}
Therefore, we have finished the proof of Proposition \ref{th103}.
\end{proof}
  \begin{proposition}\label{propforo}$[$Oscillatory integrals$]$
In  Propositions \ref{propfor} and \ref{th103}, if the  hypothesis (\ref{ss66})  is replaced with   (\ref{sv77}), then
$\sum_{j\in \mathbb{F}^{\vee}\cap\mathbb{Z}^d} | \mathcal{I}^{\rm{osc}}_{j}(\lambda) | $ has the same upper bounds  where $ (\tau,\delta_{\rm{for}})=(0,0)$   in    (\ref{rmk81})  switched with $(\tau,\delta_{\rm{for}})=(1,0)\ \text{and}\ (1,1)$.
\end{proposition}

\section{Proof of Main Theorems \ref{mainth13}-\ref{mainth15}}\label{Sec9}
\subsection{Part (A) of Main Theorems \ref{mainth13} and \ref{mainth15}}\label{sec71}   
  \begin{proof}[Proof of (A) of Main Theorems \ref{mainth13}]
The hypothesis of balanced${\bf N}(P,D_B)$ 
and the
  normal-crossing $  (P,D_B)$  of  type  $[0,\tau]$ for $ \tau<\delta_{\rm{for}}$ yield Propositions \ref{propfor} and \ref{th103}. 
Under these propositions,   it suffices to
 claim that 
\begin{align}  \label{81}
& \sum_{ \mathbb{F}\in\mathcal{F}^{k_0}}  \sum_{j\in\mathbb{F}^{\vee}\cap \mathbb{Z}^d}\mathcal{I}^{\rm{sub}}_j(\lambda) 
  \approx \begin{cases} \lambda^{-1/\delta_{\rm{for}}  }  (|\log \lambda|+1)^{d-1-k_{\rm{for}}} \ \text{if}\  \lambda\ge 1,\\
 \lambda^{-1/\delta_{\rm{bac}}  } (|\log\lambda|+1)^{d-1-k_{\rm{bac}}} \ \text{if}\  0< \lambda<1  \end{cases} 
\end{align}
where  $\mathcal{F}^{k_0}:=\mathcal{F}^{k_0}(\mathbb{P}_{\rm{for}}) \cup\mathcal{F}^{k_0} (\mathbb{P}_{\rm{bac}}) $ in Theorem \ref{lemir1}.\\
 {\bf Case 1}.  Let $\lambda\ge  1$ and $\text{dim}(\mathbb{F}^{\rm{for}}_{\rm{main}})=k_{\rm{for}}$.  Suppose that $\delta_{\rm{for}}>0$ and $\tau\in [0,\delta_{\rm{for}})$. Then   by using    (1)  of Proposition \ref{propfor}  (the forward case)    and   (2) in Proposition \ref{th103} (the backward case), we obtain that  \begin{align} \label{881}
  \sum_{j\in \mathbb{F}^{\vee}\cap\mathbb{Z}^d}  \mathcal{I}^{\rm{sub}}_{j}(\lambda) \begin{cases}
 \approx
 \lambda^{-1/\delta_{\rm{for}}}(|\log \lambda|+1)^{p-1}\ \ \text{if $\mathbb{F}\in \mathcal{F}^{k_0}(\mathbb{P}_{\rm{for}})$}
 \\
  \lesssim  
 \lambda^{-(1/\delta_{\rm{for}}+\epsilon)}\  \  \text{if  $\mathbb{F}\in \mathcal{F}^{k_0}(\mathbb{P}_{\rm{bac}})$}.
 \end{cases}
    \end{align} 
Here the case $p=0$ has a better major term. 
Sum the  RHS of (\ref{881})    over  all (finitely many)   faces $\mathbb{F}$ in (\ref{81}). Then the largest bound is
    $ \lambda^{-1/\delta_{\rm{for}}}(|\log \lambda|+1)^{d-k_{\rm{for}}-1} $, since $d-\text{dim}(\mathbb{F}^{\rm{for}}_{\rm{main}})=d-k_{\rm{for}}$ is the largest   $p$ in Proposition \ref{propfor}. Therefore, one has the desired bound   $ \lambda^{-1/\delta_{\rm{for}}}(|\log \lambda|+1)^{d-k_{\rm{for}}-1} $ for   $\lambda\ge 1$ in (\ref{81}). \\
{\bf Case 2}. 
 Let $0<\lambda< 1$ and $\text{dim}(\mathbb{F}^{\rm{bac}}_{\rm{main}})=k_{\rm{bac}}$.  
By taking    (1) of Proposition \ref{th103}  (the backward  case) and  (2) of Proposition \ref{propfor}  (the forward case) as
\begin{align} \label{882}
  \sum_{j\in \mathbb{F}^{\vee}\cap\mathbb{Z}^d}  \mathcal{I}^{\rm{sub}}_{j}(\lambda) \begin{cases}
 \approx
 \lambda^{-1/\delta_{\rm{bac}}}(|\log \lambda|+1)^{p-1}\ \ \text{if $\mathbb{F}\in \mathcal{F}^{k_0}(\mathbb{P}_{\rm{bac}})$}
 \\
  \lesssim
 \lambda^{-\epsilon}\  \text{for $0<\epsilon\ll 1$}   \  \text{if  $\mathbb{F}\in \mathcal{F}^{k_0}(\mathbb{P}_{\rm{for}})$}.
 \end{cases}
    \end{align}
The largest bound on RHS of (\ref{882}),  among all   faces $\mathbb{F}$  in (\ref{81}), is
   $ \lambda^{-1/\delta_{\rm{bac}}}(|\log \lambda|+1)^{d-k_{\rm{bac}}-1} $  because $d-\text{dim}(\mathbb{F}^{\rm{bac}}_{\rm{main}})=d-k_{\rm{bac}}$ is the largest possible $p$ in Proposition \ref{th103}.  The reverse inequality $\gtrsim$ holds for only $0<\lambda\le b$ as in (1) of Proposition \ref{th103}.  If $b< \lambda \le 1$, then  the lower bound $ \lambda^{-1/\delta_{\rm{bac}}}(|\log \lambda|+1)^{d-k_{\rm{bac}}-1} \approx 1$  in  (\ref{81}) follows from   $0\in P(D_B)$  in  the hypothesis of Main Theorem \ref{mainth13}. 
      \end{proof}

 \begin{proof}[Proof of $\lesssim$ in (A) of Main Theorem  \ref{mainth15}]
 Under the   normal-crossing hypothesis   of type  $[1,\tau]$, we can claim $\lesssim$ of (\ref{81})  for $|\mathcal{I}^{\rm{osc}}_{j}(\lambda)|$ by
applying   Proposition \ref{propforo}  
for  the estimates of $\sum_{j\in \mathbb{F}^{\vee}\cap\mathbb{Z}^d}  |\mathcal{I}^{\rm{osc}}_{j}(\lambda) |$ analogous to (\ref{881})  and  (\ref{882}).  
\end{proof}

\begin{proof}[Proof of  $\gtrsim$ in (A) of  Main Theorem \ref{mainth15}]
Let $p=d -k_{\rm{for}} $.We show the case $\lambda\ge1$:
\begin{align*} 
  \limsup_{|\lambda|\rightarrow\infty} \left|\frac{  \mathcal{I}^{\rm{osc}}(P,D_{B,R},\lambda) }{\lambda^{-1/\delta_{\rm{for}}   }(|\log\lambda|+1)^{p-1}}\right|\ge c\ \text{for}\ \mathcal{I}^{\rm{osc}}(P,D_{B,R},\lambda)=\int e^{i\lambda   P(x)}\psi_{D_{B,R}}(x)dx.
\end{align*} 
It suffices to find $c(P)>0$ such that  for any large $M$   there is $|\lambda|\ge M$  satisfying 
\begin{eqnarray} \label{0055}
 \lim_{R\rightarrow \infty} \left|\mathcal{I}^{\rm{osc}}(P,D_{B,R},\lambda)\right| >c(P)\lambda^{-1/\delta_{\rm{for}}}|\log \lambda|^{p-1}.
 \end{eqnarray}
Assume the contrary. Then  for an arbitrary small $\epsilon>0$, there is $M_\epsilon> 1$  such that
\begin{align}\label{n86}
 \left|\mathcal{I}^{\rm{osc}}(P,D_{B,R},\lambda)\right|<2\epsilon\lambda^{-1/\delta_{\rm{for}}}|\log \lambda|^{p-1}\ \text{ for all $|\lambda|>M_\epsilon$ and all $R>R_1$} 
 \end{align}
 for some   $R_1>0$.
 Let $m\in \mathbb{Z}_+$. Apply the Fourier inversion  and  the Fubini theorem,
\begin{align}\label{n87}
\int \psi\left(\frac{P(x)}{2^{-m}}\right)\psi_{D_{B,R}}(x)dx &=  \int  \left[\int 2^{-m}\widehat{\psi}(2^{-m}\lambda)e^{i\lambda P(x)} d\lambda\right] \psi_{D_{B,R}}(x)dx\nonumber\\
&=  \int \left[\int e^{i\lambda   P(x)}\psi_{D_{B,R}}(x)dx\right]2^{-m}\widehat{\psi}(2^{-m}\lambda)d\lambda \\
&\le  \int  \left| \mathcal{I}^{\rm{osc}}(P,D_{B,R},\lambda)\right|  2^{-m}|\widehat{\psi}(2^{-m}\lambda)|d\lambda.\nonumber
\end{align}
Insert the  lower bound   in  (A) of Main Theorem \ref{mainth13} into the LHS of (\ref{n87}). Then there is $c,R_1>0$ such that for $R>R_1,$
\begin{align}\label{nk88}
c(2^{-m})^{1/\delta_{\rm{for}}}(1+ |m|)^{p-1}&\le \int  \left|\mathcal{I}^{\rm{osc}}(P,D_{B,R},\lambda)\right|  2^{-m}|\widehat{\psi}(2^{-m}\lambda)|d\lambda.
\end{align}
By using $M_\epsilon$ in (\ref{n86}), split the RHS of (\ref{nk88})  into the three intervals
\begin{align} \label{app}
  \int_{|\lambda|>M_\epsilon}+\int_{ |\lambda|\le 1 }  + \int_{1\le |\lambda|\le M_\epsilon }   \left| \mathcal{I}^{\rm{osc}}(P,D_{B,R},\lambda)\right|  2^{-m}|\widehat{\psi}(2^{-m}\lambda)|d\lambda.
\end{align}
Utilize (\ref{n86}) to majorize the  first integral over $|\lambda|>M_\epsilon$ in (\ref{app}) by
\begin{align}\label{080}
  \epsilon (2^{-m})^{1/\delta_{\rm{for}}}\int_{|\lambda|>M_\epsilon}  \frac{|\log \lambda|^{p-1}}{( 2^{-m}\lambda)^{1/\delta_{\rm{for}}} } 2^{-m}|\widehat{\psi}(2^{-m}\lambda)|d\lambda\le C\epsilon  (2^{-m})^{1/\delta_{\rm{for}}}|m|^{p-1}
\end{align}
as $|\log \lambda| ^{p-1}\le ( \log (2^{-m}\lambda) | +\log 2^m)^{p-1}\lesssim  |\log (2^{-m}\lambda) |^{p-1}+     |m|^{p-1}$.
Majorize  the remaining   integrals in (\ref{app})   by the upper bounds in (A) of Main Theorem \ref{mainth15}: 
\begin{align*}
&  \int_{ |\lambda|\le 1 }\frac{(1+|\log \lambda|)^{d-\delta_{\rm{bac}}-1} }{  \lambda^{1/\delta_{\rm{bac}}}   }  2^{-m}d\lambda+ \int_{1\le |\lambda|\le M_\epsilon }   \frac{(1+|\log \lambda|)^{p-1} }{  \lambda^{1/\delta_{\rm{for}}} }  2^{-m} d\lambda 
\end{align*}
which is smaller than $C (M_\epsilon+1)2^{-m}$ where $1<\delta_{\rm{for}}\le \delta_{\rm{bac}}$.
With this and (\ref{080}), 
$$c(2^{-m})^{1/\delta_{\rm{for}}}(1+ |m|)^{p-1}\le   \text{RHS of  (\ref{nk88})}  \le C[\epsilon (2^{-m})^{1/\delta_{\rm{for}}}|m|^{p-1}+  (M_\epsilon+1)2^{-m}] $$
which  is smaller than $2C\epsilon (2^{-m})^{1/\delta_{\rm{for}}}|m|^{p-1}$  if we take $2^{-m}$ satisfying $2^{-m(1-\frac{1}{\delta_{\rm{for}}  })}M_\epsilon \ll  \epsilon$ due to $\delta_{\rm{for}}>1$.  This with  $\epsilon \ll 1$ makes a contradiction to $c\le 2C\epsilon$.
Hence  (\ref{0055}) is  true. One can similarly obtain $\gtrsim$ in (A) of the second case $\lambda\le 1$.
\end{proof}

\subsection{The Divergence  Part in  Main Theorems \ref{mainth13} and \ref{mainth15}}

\begin{proof}[Proof of (B) of Main Theorem   \ref{mainth13}]
Assume that $ {\bf N}(P,D_B)$ is unbalanced. Then  we show   $\int \psi(\lambda P(x)) \psi_{D_{B}}(x) dx=\infty$.   Take $\mathfrak{q}\in \rm{cone}^{\vee}(B\cup \Lambda(P)\cup \{-{\bf 1}\})$ where 
   $  \pi^+_{\mathfrak{q},r}$ with $r\ge 0$ is an    off-diagonal  supporting upper half space of  $ {\bf N}(P,D_B)$:  
 \begin{align} \label{jq239}
 \mathfrak{q} \in  \rm{cone}^{\vee}(B) \ \text{such that } \mathfrak{q}\cdot \mathfrak{n}\ge 0\ \forall \mathfrak{n}\in \Lambda(P)\   \text{and}\  \mathfrak{q}\cdot {\bf 1} \le 0\ \text{where $|\mathfrak{q}|\approx 1$}. \end{align}
 If $D_B$ does not contain  a neighborhood of ${\bf 0}$, then take $\mathfrak{p}={\bf 0}$.
If $D_B$ contains a neighborhood of ${\bf 0}$, then take $\mathfrak{p}\in (\rm{cone}^{\vee}(B))^{\circ}$ satisfying that   $ {\bf 2}^{-\mathfrak{p}}\sim (y_1,\cdots,y_d)\in D_B$ with   $p_\nu\gg 1$ for all $\nu$ in  $  (p_\nu)=\mathfrak{p}$  such that  for all $\mathfrak{n}\in \Lambda(P)$,
\begin{align}\label{911i}
 2^{-\mathfrak{p}\cdot \mathfrak{n}}\le \frac{1}{2^{10}C(P)\lambda} \ \text{with $C(P):=2  |\Lambda(P)|\max\{|c_{\mathfrak{n}}|\}$} 
\end{align} 
 for $P(x)=\sum_{\mathfrak{n}\in \Lambda(P)} c_{\mathfrak{n}}x^{\mathfrak{n}}$.   For $\mathfrak{q} $ in (\ref{jq239}) and $\mathfrak{p}$ in (\ref{911i}), define
\begin{align} \label{912p}
\mathbb{Z}^d(\mathfrak{q},\mathfrak{p},R)&: =\left\{j= \alpha \mathfrak{q}+ \mathfrak{p}\in \rm{cone}^{\vee}(B)\cap \mathbb{Z}^d :  \log_2R\le \alpha\le 2 \log_2R\right\}
 \end{align}
 where $\alpha\in M_0\mathbb{Z}_+ $ with $M_0\in \mathbb{N}$ in (\ref{asan1}) and   $R\gg 1$.  From $j=\alpha\mathfrak{q} +\mathfrak{p} \in \rm{cone}^{\vee}(B)$,
 \begin{align}\label{jq93}
   \mathcal{I}^{\rm{sub}}(P,D_B,\lambda) &\ge \sum_{j\in \mathbb{Z}^d(\mathfrak{q},\mathfrak{p}, R)} 2^{-j\cdot {\bf 1}} \int \psi(\lambda P  ({\bf 2}^{-j}x))\psi_{D_{B,R }}  ({\bf 2}^{-j}x)\chi(x)dx.
\end{align}
By $\Lambda(P)\subset \pi_{\mathfrak{q},r}^+$,  we have
 $\mathfrak{q}\cdot  \mathfrak{n} \ge r$ for all $\mathfrak{n}\in\Lambda(P)$. This with $2^{\alpha}\sim R$ for $R\gg 1$ in (\ref{912p}) implies
 in (\ref{jq93}),  
 \begin{align}\label{01kg}
 | \lambda P ({\bf 2}^{-j}x)|   &\le\lambda  C(P)  \sup_{\mathfrak{n}\in \Lambda(P) }2^{-\mathfrak{p}\cdot \mathfrak{n}} 2^{-\alpha\mathfrak{q}\cdot \mathfrak{n} } \le \lambda  C(P)  \sup_{\mathfrak{n}\in \Lambda(P) }2^{-\mathfrak{p}\cdot \mathfrak{n}} 2^{-\alpha r}\nonumber\\
&\le 2^{-10} \ \ \begin{cases}
\text{if $r>0$, }\\
\text{if $r=0$ and  $\lambda\in (0,c]$ with $c:= 1/(2^{11}C(P))$}, \\
\text{if $r=0$ and ${\bf 0}\in D_B$, due to (\ref{911i}).}
\end{cases} . 
\end{align}
As   $\mathfrak{p}$ is fixed and $  \mathfrak{q}\cdot {\bf 1} \le 0$ in (\ref{jq239}), we obtain that for the above three cases of $r$,
\begin{align*}
RHS\ \text{of (\ref{jq93})} &  \ge      \sum_{ \alpha\in M_0\mathbb{Z}_+; \ \log_2R\le \alpha\le 2 \log_2R}   2^{-\mathfrak{p}\cdot {\bf 1}} 2^{-\alpha   \mathfrak{q}\cdot {\bf 1}}
  \gtrsim      \log R 
\end{align*}
    as $R\rightarrow \infty$.  This yields $\mathcal{I}^{\rm{sub}}(P,D_B,\lambda)=\infty$ in (B) of Main Theorem \ref{mainth13}.
\end{proof}

\begin{proof}[Proof of (B) of  Main Theorem \ref{mainth15}]
Let ${\bf N}(P,D_B)$ be unbalanced.  
Then, choose   $\mathfrak{q}=(q_1,\cdots,q_d) $ in (\ref{jq239}).   Set $$\Psi (x)=\sum_{ \alpha\in \mathbb{Z}_+ }\chi
\left( \frac{x_1}{2^{-\alpha q_1}} \right) \cdots\chi\left( \frac{x_d}{2^{-\alpha q_d}} \right).$$
 Then $x\in \text{supp}(\Psi)$ if and only if $x\sim  {\bf 2}^{-\alpha \mathfrak{q}}$ for some $\alpha$.  We see  that  $\Psi\in \mathcal{A}(D_B)$. To show (B) of the main theorem \ref{mainth15}, we   claim that
\begin{align*} 
 \lim_{R\rightarrow\infty} \left|\int e^{i\lambda P(x) } \sum_{ \alpha\in \mathbb{Z}_+ }\chi
\left( \frac{x_1}{2^{-\alpha q_1}} \right) \cdots\chi\left( \frac{x_d}{2^{-\alpha q_d}} \right)\psi\left( \frac{x_1}{R}  \right) \cdots \psi\left( \frac{x_d}{R}  \right)dx\right|
 =\infty.
\end{align*}
By the change of variables $x_\nu\rightarrow 2^{-\alpha q_\nu}x_\nu$ and $\chi(x)=\chi\left( x_1\right) \cdots\chi\left( x_d\right)$, 
  split the above integral   into the two terms
\begin{align*} 
  \sum_{ 0\le \alpha\le \epsilon\log_2 R} + \sum_{  \epsilon\log_2 R \le \alpha\le  C\log_2 R}\  2^{-\alpha  \langle\mathfrak{q},{\bf 1}\rangle}\int e^{i\lambda P({\bf 2}^{-\alpha \mathfrak{q}}x) } \chi
\left( x\right)  dx=A(R)+B(R).
\end{align*}
 {\bf Case 1}. Let $r>0$ in (\ref{jq239}). By (\ref{01kg}) with $\mathfrak{p}={\bf 0}$,  $| \lambda P ({\bf 2}^{-\alpha \mathfrak{q}}x)|  \ll 1$.
   This with $ \langle\mathfrak{q},{\bf 1}\rangle\le 0$ in (\ref{jq239}) yields that   
\begin{align*} 
 |B(R)|\ge2^{  (\log_2R) |\langle \mathfrak{q},{\bf 1}\rangle|} (C/2)\log_2R ,  
\end{align*}
which is much bigger than   $|A(R)|$  because $|A(R)|\le 2^{\epsilon (\log_2R) |\langle \mathfrak{q},{\bf 1}\rangle|}  \epsilon \log_2 R$. Therefore,  it holds that $\lim_{R\rightarrow\infty} |A(R)+B(R)|=\infty$. 
  \\
{\bf Case 2}. Let $ r=0$ in (\ref{jq239}).  If $\mathfrak{m}\in \mathbb{F}=\pi_{\mathfrak{q},r}\cap {\bf N}(P,D_B)$,  then $\mathfrak{q}\cdot \mathfrak{m}=0$. In (\ref{jq524}), 
\begin{align*} 
1= 2^{-\alpha \mathfrak{q}\cdot \mathfrak{m}}\gtrsim 2^{c \alpha} \sum_{\mathfrak{n}\in \Lambda(P)\setminus \mathbb{F}} |c_{\mathfrak{n}}2^{-\alpha \mathfrak{q}\cdot \mathfrak{n}}| \ \  \text{for $c>0$ and $\alpha \ge 1$}.  \end{align*}
This implies that  for $\alpha\gg 1$,
\begin{align*}
P({\bf 2}^{-\alpha \mathfrak{q}}x)
&=2^{-\alpha\mathfrak{q}\cdot \mathfrak{m}}P_{\mathbb{F}}(x)+P_{\Lambda(P)\setminus \mathbb{F}}({\bf 2}^{-\alpha\mathfrak{q}}x) = P_{\mathbb{F}}(x)+O(2^{-c\alpha}).
\end{align*}
 By this with the mean value property and   $\epsilon\log_2R\le \alpha \le C\log_2 R$  in $B(R)$,
$$ e^{i\lambda  P  ({\bf 2}^{-\alpha \mathfrak{q}}x)}=e^{i\lambda  P_{\mathbb{F}}  (x)}+O(2^{-c\alpha}\lambda)  \ \text{with $ O(2^{-c \alpha}\lambda)=O(R^{-\epsilon/2})$} $$
 for  sufficiently large $R\gg 1$. Thus for this $R$,  it holds that  \begin{align}\label{cacu}
 &B(R)     =   \sum_{  \epsilon\log_2 R \le \alpha\le  C\log_2 R} 2^{-\alpha  \langle\mathfrak{q},{\bf 1}\rangle}\left[ \int  e^{i\lambda P_{\mathbb{F}}(x)}\chi(x)dx +O(R^{-\epsilon/2})\right].  
 \end{align}
A function $Q$ defined by $\lambda\rightarrow Q(\lambda):= \int  e^{i\lambda P_{\mathbb{F}}(x)}\chi(x)dx  $ is an analytic function in $\mathbb{R}$, not identically zero.  The identity theorem implies that $\{\lambda\in \mathbb{R}:Q(\lambda)=0\}$ is a measure-zero set. For $\lambda$  with $Q(\lambda)> 0$ and  $R\gg 1$,  we have that $|\text{RHS of (\ref{cacu})}|$ is
\begin{align*}
\left|\sum_{  \epsilon\log_2 R \le \alpha\le  C\log_2 R}   2^{-\alpha  \langle\mathfrak{q},{\bf 1}\rangle} [Q(\lambda) +O(R^{-\epsilon/2})]\right|&\gtrsim \left| \sum_{  \epsilon\log_2 R \le \alpha\le  C\log_2 R}   2^{-\alpha  \langle\mathfrak{q},{\bf 1}\rangle} Q(\lambda)  \right|\\
&\ge |Q(\lambda)2^{  (\log_2R) |\langle \mathfrak{q},{\bf 1}\rangle|}(C/2)\log_2R |,
 \end{align*}
which is  $\ge \epsilon  2^{\epsilon (\log_2R) |\langle \mathfrak{q},{\bf 1}\rangle|}\log_2 R\ge |A(R)|$. So,   $\lim_{R\rightarrow\infty} |A(R)+B(R)|=\infty$.
\end{proof}

\subsection{Proof of Corollary \ref{mainth14}}

  \begin{proof}[Proof of (A) in Corollary \ref{mainth14}]
  Suppose that ${\bf N}(P,D_B)$ is balanced. Then
\begin{align*} 
\int_{D_B} |P(x)|^{-\rho}  dx&\approx\sum_{k\in \mathbb{Z}} 2^{-\rho k} \left( |\{x\in D_B:|P(x)|\le 2^k\} |-  |\{x\in D_B:|P(x)|\le 2^{k-1}\} |\right) \\
   =&(1-2^{-\rho})  \sum_{k\in \mathbb{Z}} 2^{-\rho k}  |\{x\in D_B:|P(x)|\le 2^k\} | \\
  \approx   \sum_{2^k< 1}&   2^{-\rho k}  |\{x\in D_B:  |P(x)|\le 2^k\} |+\sum_{2^k\ge 1} 2^{-\rho k}  |\{x\in D_B:|P(x)|\le 2^k\} |. 
 \end{align*}
 From    (A) of the main theorem \ref{mainth13}, there are  $C_1,C_2>0$ independent of   $k$ such that 
\begin{align*}
&C_1 2^{k/\delta_{\rm{for}}}(|k|+1)^a\le |\{x\in D_B: |P(x)|\le 2^k\}|\le C_2 2^{k/\delta_{\rm{for}}}(|k|+1)^a \ \text{if $2^k< 1$},\\
&C_1 2^{k/\delta_{\rm{bac}}}(|k|+1)^b\le |\{x\in D_B: |P(x)|\le 2^k\}|\le C_2 2^{k/\delta_{\rm{bac}}}(|k|+1)^b\ \text{if $2^k\ge 1$}
\end{align*}
for  $a=d-k_{\rm{for}}-1$ and $b=d-k_{\rm{bac}}-1$.  This  yields that
   \begin{align*}
\int_{D_B} |P(x)|^{-\rho}  dx  &\approx\sum_{2^k<1} 2^{(1/\delta_{\rm{for}}-\rho) k}(|k|+1)^a + \sum_{2^k\ge 1} 2^{(1/\delta_{\rm{bac}}-\rho ) k} (|k|+1)^b
\end{align*}
which converges if and only  if $1/\delta_{\rm{bac}}<\rho< 1/\delta_{\rm{for}}$. This proves (A).
\end{proof}
\begin{proof}[Proof of (B) in Corollary \ref{mainth14}]
Suppose that ${\bf N}(P,D_B)$ is unbalanced.
Then, the part (B) of  the main theorem \ref{mainth13} implies $ \left|\{x\in D_B:|P(x)|\le 2^k\}\right|=\infty$  for some fixed $k$. This shows    $\int_{D_B} |P(x)|^{-\rho}dx\ge 2^{-k\rho}\left|\{x\in D_B:|P(x)|\le 2^k\}\right| =\infty$.
\end{proof}

\subsection{General Class of Phase Functions and Domains}\label{Sec43}
 We shall extend Main Theorems \ref{mainth13} and \ref{mainth15} to a larger class of  smooth functions $Q$  and  regions $ D$. 

  \begin{definition}\label{rmk21}
  Let $P$ be a  polynomial and let   $\mathbb{P}={\bf N}(P,D_B)$.
Set $$\mathbb{F}_{\rm{special}}^{\vee}:=\begin{cases} [\mathbb{F}^{\rm{main}}_{\rm{for}}]^{\vee}\cup [\mathbb{F}^{\rm{main}}_{\rm{bac}}]^{\vee}\ \text{ if  $\mathbb{P}$ is balanced}\\
\mathbb{P}^{\vee}_{\rm{off}}:=\rm{cone}^{\vee}(B \cup \Lambda(P)\cup \{-{\bf 1}\})\  \text{    if $\mathbb{P}$ is unbalanced}
\end{cases}
$$  as in Figure \ref{osc10}.  Consider a region $D\subset\mathbb{R}^d$ and  a smooth function $Q$   on $D\cap (\mathbb{R}\setminus\{0\})^d$.   
Then   $(Q,D)$ is  equivalent to $(P,D_B)$, provided   (1) and (2)  below hold.
\begin{itemize}
\item[(1)] $D\cong D_B$ if    $2^{-\mathbb{F}^{\vee}_{\rm{special}}\cap B(0,r)^c}\subset D\cap (\mathbb{R}\setminus\{0\})^d\subset 2^{-(\rm{cone}^{\vee}(B)+O(1))}$ for some $r$.
\item[(2)] $Q\cong_{[\sigma,\tau]} P $  if  $    \sum_{ \sigma\le |\alpha|\le \tau}|x^{\alpha}\partial^{\alpha}Q(x)|\bigg|_{D\cap (\mathbb{R}\setminus\{0\})^d}\approx  \sum_{\mathfrak{m}\in \Lambda(P)} |x^{\mathfrak{m}}|. $
\end{itemize}  
Denote  (1) and (2) at once by  $(Q,D) \cong_{[\sigma,\tau]} (P,D_B)$. For this case, define
\begin{align*} 
{\bf N}(Q,D)  :={\bf N}(P,D_B). 
\end{align*}
\end{definition}
  \begin{exam}
Thanks to $O(1)$ in (1) of Definition \ref{rmk21}, one  can treat  the perturbed domain   $D= \{x\in\mathbb{R}^d:   |x^{\mathfrak{b}}|\le 5\ \text{for all $\mathfrak{b}\in B$}\}$ of $D_B$ satisfying $D \cong D_B$.  For  instance,  take  $ D=\{x:|x^{{\bf e}_1}|,|x^{-{\bf e}_1}|\le 5\}.$ Then $D  \cong D_{\{{\bf e}_1,-{\bf e}_1\}}$     though  $\rm{cone}({\bf e}_1,-{\bf e}_1)$  is  not strongly convex. 
\end{exam}

\begin{exam}[Fractional Laurent  Polynomial]\label{defi42}
Let $ \Lambda(P)\subset (\frac{1}{K_1}\mathbb{Z}) \times \cdots\times (\frac{1}{K_d}\mathbb{Z})$ be a finite set with $K_\nu\in\mathbb{N}$ for $\nu\in [d]$. Then we call  $P(x)=\sum_{\mathfrak{m}=(m_\nu)\in \Lambda(P)} c_{\mathfrak{m}}x^{\mathfrak{m}}$ a Laurent  polynomial.  For $m_\nu=p/q$ with $q\in \mathbb{N}$ and $p\in\mathbb{Z}$, let
  $$x_\nu^{m_\nu}=x_\nu^{p/q}:=\begin{cases} (\sqrt[q]{|x_\nu|})^p\ \text{if $x_\nu\ge 0$} \\
 \text{one of}\  \pm  (\sqrt[q]{|x_\nu|})^p\ \text{if $x_\nu< 0$.}
   \end{cases} $$ 
   Given   $P,D_B$,  one can set  ${\bf N}(P,D_B): ={\bf Ch}\left(  \Lambda(P) + \rm{cone}(B)  \right)$ as in Definition \ref{ded28}
  \end{exam}

\begin{exam}
One can exclude  the middle region $M_h=\{x\in\mathbb{R}^d:1/h\le |x_\nu|\le h\ \forall \nu\in [d]\}$      from $D_B$,  keeping
 $ D_B\setminus M(h) \cong D_B$ due to $B(0,r)^c$ in (1).  
\end{exam}

 \begin{corollary}\label{coro42} 
Let  $(Q,D) \cong_{[0,\tau]} (P,D_B)$ for $\tau<\delta_{\rm{for}}$    and  $0\in Q(D)$.
\begin{itemize}
\item[(A)]  If ${\bf N}(Q,D)$ is   balanced, then it holds that 
\begin{align*} 
|  \{x\in D:|\lambda Q(x)|\le 1\}| \approx \begin{cases} \lambda^{-1/\delta_{\rm{for}}}  (|\log \lambda|+1)^{d-1-k_{\rm{for}}} \ \text{if}\  \lambda \in [1,\infty),\\
 \lambda^{-1/\delta_{\rm{bac}}} (|\log\lambda|+1)^{d-1-k_{\rm{bac}}} \ \text{if}\   \lambda\in (0,1).  \end{cases}
\end{align*} 
 \item[(B)] If ${\bf N}(Q,D)$ is  unbalanced, then, there exists $c>0$ such that
  $$|  \{x\in D:|\lambda Q(x)|\le 1\}|=\infty\ \text{ for  all $ \lambda\in (0,c)$.}$$ 
  \end{itemize}
\end{corollary}
\begin{remark}
In Corrollary \ref{coro42},  we do not assume  that  $\rm{cone}(B)$ is  strongly convex. One can include  $\tau=  \delta_{\rm{for}}=0$   in (A).
\end{remark}

\begin{proof}[Proof of Corollary  \ref{coro42}] 
We restrict the region $D$ to $D\cap (\mathbb{R}\setminus\{0\})^d$.
Replace  $D_B$ and  $P(x)$ with
 $  D \subset 2^{-(\rm{cone}^{\vee}(B)+B(0,r)) } $  and  a smooth function $Q(x)$  in     Theorem \ref{lemir1}. Then we obtain the same decay rate for each piece  of integral   in Propositions \ref{propfor} and \ref{th103}. This enables us to have
  the upper bound of (\ref{81}), which   gives the desired upper bound for $\mathcal{I}^{\rm{sub}}(Q,D,\lambda)$. The lower bounds (\ref{jq242}) and (\ref{jq243}) are  obtained from   ${\bf 2}^{- [\mathbb{F}^{\rm{main}}_{\rm{for}}]^{\vee}\cup [\mathbb{F}^{\rm{main}}_{\rm{bac}}]^{\vee} \cap B(0,r)^c} \subset D$ in (1) of Definition \ref{rmk21}. The divergence   follows from (\ref{jq93}) and  ${\bf 2}^{-\mathbb{P}^{\vee}_{\rm{off}}\cap B(0,r)^c}\subset D_B$. 
\end{proof}

 \begin{exam}\label{ex49}
A  decay  of  a sub-level-set-measure can be slower than that of  a  corresponding   oscillatory integral in a local region. But this can occur in an extreme manner  in the global region $\mathbb{R}^d$.
For example, let  $F(x)=x_3^2-(x_1^2+x_2^2)$.
Then the local estimates are  $\mathcal{I}^{\rm{sub}}(F(x),[-1,1]^3,\lambda)=O(\lambda^{-1})$  and $\mathcal{I}^{\rm{osc}}(F(x),[-1,1]^3,\lambda)=O(\lambda^{-3/2})$.
But in the global domain, one can compute $$
\begin{cases}
\mathcal{I}^{\rm{sub}}(F(x),\mathbb{R}^3,\lambda)=\infty\\
\mathcal{I}^{\rm{osc}}(F(x),\mathbb{R}^3,\lambda)=O(\lambda^{-3/2}).
\end{cases}
$$
The above estimate of the oscillatory integral follows from the iterated integration, or from  Main Theorem  \ref{mainth15} with
 $\tau_1(P,\mathbb{R}^3)=1$, $\delta_{\rm{for}}=2/3$. For the above sublevel-set estimate, use
$ \Phi(x)=(x_1,x_2,x_3+\sqrt{x_1^2+x_2^2})$ to have
 $F\circ\Phi(x)= x_3(x_3+2\sqrt{x_1^2+x_2^2})$ on  $D=\{ x :|x_\nu^{-1}x_3|\le  10^{-1}\ \text{for}\ \nu=1,2\}$. Then
$(F\circ\Phi,D)\cong_{[0,0]} ( x_3x_1+x_3x_2,D_{\{(-1,0,1),(0,-1,1)\}})$ in Definition \ref{rmk21} and ${\bf N}(x_3x_1+x_3x_2,D_{\{(-1,0,1),(0,-1,1)\}})$  unbalanced. Thus, Corollary \ref{coro42} yields 
\begin{align*} 
\mathcal{I}^{\rm{sub}}(F(x),\mathbb{R}^3,\lambda)&\ge  |\{x\in D: |\lambda F\circ\Phi(x)|\le 1\}|  =\infty. 
\end{align*}
  \end{exam}
\begin{exam}
Let $\mathbb{P}={\bf N}(P,D_B)$. If $P$  is a polynomial,   then $\delta_{\rm{for}}=d(\mathbb{P}_{\rm{for}})\ge 0$. If  $P$ is  a Laurent polynomial,    then $d(\mathbb{P}_{\rm{for}})$  in Figure \ref{forba89}  can be negative and
 $ \delta_{\rm{for}}=\max\{0,d(\mathbb{P}_{\rm{for}})\}=0$.  For example, if $P(x)= \frac{1}{x_1^{1}x_2^{3}}+ \frac{1}{x_1^{3}x_2^{1}} $, then $d(\mathbb{P}_{\rm{for}})=-2$ and $\delta_{\rm{for}}=0$.  Since ${\bf N}(P,(\mathbb{R}\setminus\{0\})^2)={\bf conv}(\{(-1,-3),(-3,-1)\})$ is unbalanced, $\mathcal{I}^{\rm{sub}}(P,(\mathbb{R}\setminus\{0\})^2,\lambda)=\infty$.  
\end{exam}

 \section{Partition of Domain}
\subsection{Statement of Global Theorems after Partition of Domain}
Let $P(x)=((x_1^2+x_2^2)-1)^2$ and $D=\mathbb{R}^2$. Then as $\tau_0(P,D)=2>0=\delta_{\rm{for}}(P,D)$, one cannot apply Main Theorems \ref{mainth13} for   
 $|\{x\in D:\lambda P(x)|\le 1\}|$. However, we can find a partition $D=\bigcup_{i=0}^M D_i$  so as to compute $|\{x\in D_i:\lambda P(x)|\le 1\}|$ for each $i$. In this section, we  restate Main Theorems \ref{mainth13} and \ref{mainth15} by   partitioning the domains.
 \begin{main} \label{mainth4}
Let  $P$ be a polynomial $P({\bf 0})=0$  in a domain $D\subset \mathbb{R}^d$. Suppose that
  a partition $\{D_i\}_{i=0}^M$ of  $D $ with  coordinate maps  $ \phi_i:\phi^{-1}_i(D_i) \rightarrow D_i$ decomposes 
 \begin{align}\label{4.17}
 \int_{D} \psi(\lambda P(x))dx& \approx\sum_{i=0}^M\int \psi(\lambda P (x)) \psi_{D_i}(x)dx\ \text{for $\psi_{D_i}\in C^{\infty}(D_i)$}\\
\text{satisfying}\ (P\circ\phi_i,\phi^{-1}_i(D_i) )&\cong_{[0,\tau_i]} (P_i,D_{B_i})\ \text{are  normal-crossing   of type $[0,\tau_i]$. }\nonumber
\end{align}
Then with the distance and multiplicity derived from $\mathbb{P}_i:= {\bf N}(P_i,D_{B_i})$ in  (\ref{jq514}): $$[\delta_{\rm{for}}', \delta_{\rm{bac}}']{\bf 1}=\bigcap_{i=0}^M\mathbb{P}_i\cap \rm{cone}({\bf 1})\ \text{and}\ \begin{cases} k_{\rm{for}}'=\min\{k_{\rm{for}}(\mathbb{P}_i): \delta_{\rm{for}}(\mathbb{P}_i)=\delta_{\rm{for}}'\}_{i=0}^M\\
k_{\rm{bac}}'=\min\{k_{\rm{bac}}(\mathbb{P}_i): \delta_{\rm{bac}}(\mathbb{P}_i)=\delta_{\rm{bac}}'\}_{i=0}^M,\end{cases}
$$
 where $[\delta_{\rm{for}}(\mathbb{P}_i),\delta_{\rm{bac}}(\mathbb{P}_i)]=\mathbb{P}_i\cap \rm{cone}({\bf 1})$, one can have the   estimates:
\begin{itemize}
\item[(A)] If  $\mathbb{P}_i$ are  balanced  and  $\tau_i<\delta_{\rm{for}}'$ for all $ i,$ then it holds that
 \begin{align}  \label{k67}
 \int \psi(\lambda P(x)) \psi_{D}(x) dx \approx \begin{cases}
 \lambda^{- 1/\delta_{\rm{for}}' }    (|\log\lambda|+1)^{d-1-k_{\rm{for}}'   }  \ \text{if $\lambda\ge 1$}\\
   \lambda^{- 1/\delta_{\rm{bac}}'  }    (|\log\lambda|+1)^{d-1-k_{\rm{bac}}'   } \ \text{if $\lambda< 1$}.
   \end{cases}
 \end{align}
If  the   type $[0,\tau_i]$  is replaced with $[1,\tau_i]$, then   the  oscillatory integral estimates    hold.
  Here
 $\delta_{\rm{for}}', \delta_{\rm{bac}}'$ and $k_{\rm{for}}', k_{\rm{bac}}'$ are independent of  choices $\{D_i\}_{i=0}^M$.   
\item[(B)]
If at least one of  $\mathbb{P}_i$ is unbalanced, then $LHS$ of (\ref{k67}) diverges.
\end{itemize}
\end{main}

 \begin{proof}[Proof  Main Theorem \ref{mainth4}]
 Applying (A) of Main Theorem \ref{mainth13}   and Corollary  \ref{coro42},
\begin{align*}
\int \psi(\lambda P(x)) \psi_{D_i}(x)dx&=\int \psi(\lambda P \circ\phi_i(x)) \psi_{D_i}(\phi_i(x))dx\approx \int \psi(\lambda P_i(x)) \psi_{D_{B_i}}(x)dx\\
&\approx \begin{cases}
 \lambda^{- 1/d_{\rm{for}}(\mathbb{P}_i)  }    (|\log\lambda|+1)^{d-1-k_{\rm{for}}(\mathbb{P}_i)   }  \ \text{if $\lambda\ge 1$}\\
   \lambda^{- 1/d_{\rm{bac}}(\mathbb{P}_i)   }    (|\log\lambda|+1)^{d-1-k_{\rm{bac}}(\mathbb{P}_i)    } \ \text{if $\lambda< 1$.}
   \end{cases}
\end{align*}
So, the decay rates of LHS of (\ref{4.17}), according to $\lambda\ge 1$   or $\lambda<1$ are reciprocals of $$ \max\{\delta_{\rm{for}}(\mathbb{P}_i)\}_{i=0}^M\ \text{or}\  \min\{\delta_{\rm{bac}}(\mathbb{P}_i)\}_{i=0}^M$$ which coincide with the above $1/\delta_{\rm{for}}'$ or $1/\delta_{\rm{bac}}'$ respectively.  Thus, we obtain (\ref{k67}).  Assume that there is another partition  having $\delta_{\rm{for}}'', \delta_{\rm{bac}}''$ and  $k_{\rm{for}}'', k_{\rm{bac}}''$. By applying  (\ref{k67}),  it holds that with   the constants involved in $\approx$ below  independent of $\lambda$, 
   \begin{align*}  
 \lambda^{- 1/\delta'_{\rm{for}}}    (|\log\lambda|+1)^{d-1-k_{for}'}  &\approx \int\psi(\lambda P(x)) \psi_{D_B}(x)dx  
  \approx 
 \lambda^{- 1/\delta_{\rm{for}}''}    (|\log\lambda|+1)^{d-1-k_{for}''},
 \end{align*} 
showing $\delta_{\rm{for}}''= \delta_{\rm{for}}', k_{\rm{for}}''=k_{\rm{for}}'$. Similarly, $\delta_{\rm{bac}}''= \delta_{\rm{bac}}', k_{\rm{bac}}''=k_{\rm{bac}}'$.
  Finally,  the oscillatory integral estimates follow from
  Main Theorem \ref{mainth15}.
  \end{proof}
  \begin{remark}
 The above theorem is  the first  step toward a global resolution of singularity.  But, we do not establish the   resolution of singularity in this paper. 
  \end{remark}
 \subsection{Three Types of Singular Set}
 Set  $V(P):=\{x\in \mathbb{R}^d: P(x)=0\ \text{or}\ \nabla P(x)={\bf 0}\}$ of singular points.  To treat non-local $V(P)$ with a partition $\{D_i\}$, we  apply Main Theorem 3 for the following model  cases:
\begin{itemize}
\item[(i)] $V(P)$ is a compact irreducible curve (circle) in Example \ref{ex46},
\item[(ii)] $V(P)$ is a non-compact irreducible curve (parabola) in Example \ref{expara},
\item[(iii)] $V(P)$ is the union of the cases (i) and (ii) in Example \ref{exboth}.
\end{itemize}

\begin{exam} [Circle]\label{ex46}
 Let $P_{\rm{circ}}(x)=(x_1^2+x_2^2-1)^2$ and $D_B=\mathbb{R}^2$.  Then 
\begin{align} 
 \int  \psi(\lambda P_{\rm{circ}}(x)) \psi_{D_B}(x) dx &\approx   \lambda^{-1/2} \ \text{if}\ \lambda\in (0,\infty),\label{4.199}\\
 \int   e^{i\lambda P_{\rm{circ}}(x)}  \psi_{D_B}(x)dx &=O(  \lambda^{-1/2}) \ \text{if}\ \lambda\in (0,\infty).\label{4.200}
\end{align} 
  Due to the compactness of    $V(P_{\rm{circ}})=S^1$,
  we can  find  $\{\mathfrak{c}_i\}_{i=1}^M\subset V(P_{\rm{circ}})$ such that
\begin{align}\label{p021}
 V(P_{\rm{circ}})+ [-h/2,h/2]^d  \subset\bigcup _{i=1}^M  \mathfrak{c}_i+[-h,h]^d\ \text{with $|h|\ll 1 $}.
\end{align}
 Set $D_i=\mathfrak{c}_i+ [-h,h]^2$ for $i\in [M]$ and $D_0=D_B\setminus \bigcup_{i\in [M]} D_i$.  Then
 $$\int \psi(\lambda P_{\rm{circ}}(x))\psi_{D_B}(x)dx=\sum_{i=0}^M \int \psi(\lambda P_{\rm{circ}}(x))\psi_{D_i}(x)dx.$$
\begin{itemize}
\item[(1)]   Take  $\phi_0=Id$ on $D_0$ to  define $P_0=P_{\rm{circ}}\circ\phi_0 $ and $\phi_0^{-1}(D_0)=D_0$. Then   $(P_0,\phi_0^{-1}(D_0))\cong (P_{\rm{circ}},D_B)$    is normal-crossing   of type $[0,\tau_0]$ for $\tau_0=0$.
\item[(2)] Fix $i\in [M]$ and let $\mathfrak{c}_i=(c_1,c_2)\in V(P)=S^1$ in (\ref{p021}). Then
\begin{align*}
  \int \psi(\lambda P_{\rm{circ}}(x))\psi_{D_i}(x)dx&=\int \psi(\lambda P_{\rm{circ}}(x))\psi\left(\frac{x-\mathfrak{c}_i}{h}\right)dx\\
& =  \int \psi (\lambda P_{\rm{circ}} (x+\mathfrak{c}_i))\psi\left(\frac{x}{h}\right)dx.
\end{align*}
Let $|c_1|\le|c_2|$ for $(c_1,c_2)\in S^1$. As $|x|\le h\ll 1 $, express $P (x+\mathfrak{c}_i) =(x_1^2+x_2^2  +2c_1 x_1+2c_2 x_2 )^2  =[(x_2-a(x_1))(x_2-b(x_1))]^2$ with
\begin{align}
  & \begin{cases}   a(x_1)  =-c_2+\sqrt{c_2^2-(2c_1x_1+x_1^2)} =-( c_1/c_2)x_1 +O(x_1^2) \\
   b(x_1)=-c_2-\sqrt{c_2^2-(2c_1x_1+x_1^2)}=-2c_2  +O(x_1 )\approx -2c_2 .
   \end{cases}\label{s33}
  \end{align}
Using a  coordinate map $\phi_i(x_1,x_2)=x+\mathfrak{c}_i+(0,a(x_1))  $,   define $ P_i(x)$:
  \begin{align}\label{gn1}
  P_{\rm{circ}}\circ\phi_i(x_1,x_2) =  [x_2(x_2+a(x_1)-b(x_1))]^2=(4c_2^2+O(x_1^2))x_2^2\approx x_2^2
   \end{align}
 on   $\phi_i^{-1}(D_i)=\{|x_1|\le h, |x_2+a(x_1)|\le h\}  \cong [-1,1]^2$.  This leads
  \begin{align*}
& \int \psi (\lambda P_{\rm{circ}}(x+\mathfrak{c}_i))\psi\left(\frac{x}{h}\right)dx
 =  \int \psi (\lambda P_i(x))\psi_{\phi_i^{-1}(D_i)}(x)dx 
 \end{align*}
for  $(P_i,\phi_i^{-1}(D_i) ) \cong (  x_2^2,[-1,1]^2)$     of type $[0,\tau_i]$ with $\tau_i=0$.
\item[(3)] Set   
 $\mathbb{P}_i:= \begin{cases}  {\bf N} ((1-(x_1^2+x_2^2))^2,\mathbb{R}^2)\ \text{if $i=0$}\\
  {\bf N} (x_2^2,[-1,1]^d ) \ \text{if $i\in [M]$.}\end{cases}$ Then
  $\bigcap_{i=0}^M\mathbb{P}_i\cap \rm{cone}({\bf 1})= [\delta_{\rm{for}}', \delta_{\rm{bac}}']{\bf 1}=2{\bf 1}$ due to $\begin{cases}
\text{ $[\delta_{\rm{for}}(\mathbb{P}_0), \delta_{\rm{bac}}(\mathbb{P}_0)]=[0,2]$}\\
\text{   $[\delta_{\rm{for}}(\mathbb{P}_i), \delta_{\rm{bac}}(\mathbb{P}_i)]=[2,\infty]$}
   \end{cases}$
   where $k_{\rm{for}}'=1$ and $k_{\rm{bac}}'=1$ because $k_{\rm{for}}(\mathbb{P}_i)=1$ for  $i\in  [M]$ and $k_{\rm{bac}}(\mathbb{P}_0)=1$. 
\end{itemize}
Hence   Main Theorem  \ref{mainth4}   gives (\ref{4.199}). 
The type condition    $[1,1]$     yields (\ref{4.200}). 
\end{exam}

\begin{exam} [Parabola]\label{expara}
 Let $P_{\rm{para}}(x)=(x_2-x_1^2)^2$ in  $D_B=\mathbb{R}^2$.  If $ \lambda>0$,
\begin{align*} 
&\int \psi(\lambda P_{\rm{para}}(x))\psi_{D_B}(x) dx  =\infty \ \text{and}  \\
&   \left|\int   e^{i\lambda P_{\rm{para}}(x)}\Psi_{D_{B}}(x) dx\right| =\infty \   \text{for some $ \Psi_{D_B}\in \mathcal{A}(D_B)$}. 
\end{align*} 
\begin{proof}
Split $\mathbb{R}^2=D_1\cup D_2$ where $$D_1=\{x: |x_2-x_1^2|\ge \epsilon |x_1^2|\}\ \text{and}\ D_2=\{x: |x_2-x_1^2|< \epsilon |x_1^2|\}.$$
Set  $\phi_i(x)=(x_1,x_2+x_1^2)$ and  $P_i(x)=P_{\rm{para}}\circ\phi_i(x)=x_2^2$  for $i=1,2$.
 Then
\begin{align*}
 \phi_1^{-1}(D_1)&:=\{x:  |x_1^{2}x_2^{-1}|\le \epsilon^{-1}\}\cong D_{\{(2,-1)\}}\ \text{and}\ \mathbb{P}_1={\bf N}(P_1,D_{\{(2,-1)\}})\\
 \phi_2^{-1}(D_2)&:=\{x:  |x_1^{-2}x_2^{1}|\le \epsilon \}\cong D_{\{(-2,1)\}}\ \text{and}\ \mathbb{P}_2={\bf N}(P_2,D_{\{(-2,1)\}})\ \text{is unbalanced}.
 \end{align*}
 Thus apply Main Theorem \ref{mainth4} to obtain  the above estimates.
\end{proof}
\end{exam}

\begin{exam}[Circle $\cup$ Parabola]\label{exboth}
 Let $P(x)=(x_1^2+x_2^2-1)^2(x_2-x_1^2)^2$ and $D_B=\mathbb{R}^2$.  Then Main Theorem \ref{mainth4} shows
\begin{align} 
 \int  \psi(\lambda P(x)) \psi_{D_B}(x) dx &\approx\begin{cases}
 \lambda^{- 1/2 }    (|\log\lambda|+1)   \ \text{if $\lambda\ge 1$}\\
   \lambda^{- 1/4}   (|\log\lambda|+1)     \ \text{if $\lambda< 1$}
   \end{cases} \label{709}\\
 \left|\int   e^{i\lambda P(x)} \psi_{D_B}(x)dx \right|&\lesssim\begin{cases}
 \lambda^{- 1/2 }    (|\log\lambda|+1)   \ \text{if $\lambda\ge 1$}\\
   \lambda^{- 1/4}     (|\log\lambda|+1)   \ \text{if $\lambda< 1$.}
   \end{cases}\label{710}
\end{align}
\end{exam} 
\begin{proof}[Proof of (\ref{709}) and (\ref{710})]
Let $\psi+\psi^c\equiv 1$ on  $\mathbb{R}^2$. Define the singular regions of the circle and the parabola as
$$D_{\rm{circ}}=\left\{ |x_1^2+x_2^2-1|\le \epsilon \right \}\ \text{and}\ D_{\rm{para}}=\left\{  \frac{|x_2-x_1^2|}{x_1^2}\le \epsilon \right \}$$
and the  non-singular regions as
$$D^{\rm{away}}_{\rm{circ}}=\left\{ |x_1^2+x_2^2-1|\ge \epsilon/2 \right \}\ \text{and}\ D^{\rm{away}}_{\rm{para}}=\left\{  \frac{|x_2-x_1^2|}{x_1^2}\ge \epsilon/2 \right \}.$$
Decompose $Id_{\mathbb{R}^2}=\sum_{i=1}^4 \psi_i$ with $\psi_i\in \mathcal{A}(D_i)$  defined by
\begin{align*}
 \psi_1(x)&:=\psi^c\left(\frac{ x_1^2+x_2^2-1}{\epsilon} \right)\psi^c\left(\frac{ x_2-x_1^2}{ \epsilon     x_1^2} \right)\ \text{supported}\ D_1:=D^{\rm{away}}_{\rm{circ}} \cap D^{\rm{away}}_{\rm{para}} \\
 \psi_2(x)&:=\psi \left(\frac{ x_1^2+x_2^2-1}{\epsilon} \right)\psi^c\left(\frac{ x_2-x_1^2}{ \epsilon     x_1^2} \right)\ \text{supported}\ D_2:=D_{\rm{circ}}  \cap D^{\rm{away}}_{\rm{para}} \\
 \psi_3(x)&:=\psi^c\left(\frac{ x_1^2+x_2^2-1}{\epsilon} \right)\psi\left(\frac{ x_2-x_1^2}{ \epsilon     x_1^2} \right)\ \text{supported}\ D_3:= D^{\rm{away}}_{\rm{circ}} \cap D_{\rm{para}} \\
 \psi_4(x)&:=\psi\left(\frac{ x_1^2+x_2^2-1}{\epsilon} \right)\psi\left(\frac{ x_2-x_1^2}{ \epsilon     x_1^2} \right)\ \text{supported}\ D_4:=D_{\rm{circ}}\cap D_{\rm{para}}
\end{align*} 
{\bf Case $(P,D_1)$}. This is the non-singular region. Set $P_1=P$ and $\phi_1(x)=Id.$
Then $(P_1\circ\phi_1,\phi_1^{-1}(D_1))\cong ((1+x_1^4+x_2^4)(x_1^4+x_2^2),\mathbb{R}^2)$ is normal-crossing of type $[0,\tau_1]$ with $\tau_1=0$. Define $\mathbb{P}_1={\bf N}( (1+x_1^4+x_2^4)(x_1^4+x_2^2),\mathbb{R}^2)$ where 
\begin{align*}
\text{$\delta_{\rm{for}}(\mathbb{P}_1)=4/3$ and $\delta_{\rm{bac}}(\mathbb{P}_1)=4$ with $k_{\rm{for}}(\mathbb{P}_1)=1$ and $k_{\rm{bac}}(\mathbb{P}_1)=0$. }
\end{align*} 
{\bf Case  $(P,D_2)$}.    Cover  $D_2=D_{\rm{circ}}  \cap D^{\rm{away}}_{\rm{para}}=\bigcup_{\ell=1}^M D_{2,\ell}\ \text{where} \ D_{2,\ell}:= D_2\cap (\mathfrak{c}_\ell+[-h,h]^2)$ with $h\ll \epsilon\ll 1$  and 
\begin{align}
 \mathfrak{c}_\ell=(c_1(\ell),c_2(\ell))\in S^1\setminus D_{\rm{para}}\ \text{where}\  h^2\ll |c_2(\ell)-c_1(\ell)^2|\approx 1.
\end{align}
Then,
decompose $$\int_{\mathbb{R}^2}  \psi(\lambda P (x))\psi_2(x ) dx=\sum_{\ell=1}^M \int     \psi (\lambda P (x) ) \psi^c\left(\frac{x_2- x_1^2}{\epsilon x_1 }  \right)   \psi\left( \frac{x-\mathfrak{c}_\ell}{h}\right)dx. $$
  As  (\ref{s33}),  the pullback of $D_{2,\ell}$ is  the    coordinate map $
 \phi_\ell:  \phi_\ell^{-1}(D_{2,\ell}) \rightarrow D_{2,\ell}$:
\begin{align}\label{0596}
\phi_\ell(x)=(x_1+c_1(\ell), x_2+c_2(\ell)+a(x_1))
\end{align}
changing the above integrals as
\begin{align*}
  \int     \psi (\lambda P\circ\phi_\ell(x) )  \psi^c\left(\frac{x_2+a(x_1)+c_2(\ell)- (x_1+c_1(\ell))^2}{\epsilon(x_1+c_1(\ell))}  \right)  \psi\left(  \frac{ x_1,x_2+a(x_1) }{h}\right)dx. 
\end{align*}
 Define
 $P_{2,\ell}(x):= P\circ\phi_\ell(x) =P_{\rm{cir}}\circ\phi_\ell(x) P_{\rm{para}}\circ\phi_\ell(x)\approx x_2^2$     due to   $P_{\rm{cir}}\circ\phi_\ell(x) \approx x_2^2$ in  (\ref{gn1})  and
$
P_{\rm{para}}\circ\phi_\ell(x) =
 \left[x_2+a(x_1)+c_2(\ell)- (x_1+c_1(\ell))^2\right]^2 \approx |c_2(\ell)-c_1(\ell)^2|^2\approx 1. $
Thus the above integrals become 
\begin{align*}
  \int     \psi (\lambda  P_{2,\ell} (x))    \psi\left(  \frac{ x_1,x_2+a(x_1) }{h}\right)dx=\int     \psi (\lambda P_{2,\ell} (x) )\psi_{\phi_\ell^{-1}(D_{2,\ell})}(x)dx. 
\end{align*}
The support of the  integrals is $\phi_\ell^{-1}(D_{2,\ell}) \cong \{x: |x|\le h \}$. Thus
  $(P_{2,\ell} (x), \phi_\ell^{-1}(D_{2,\ell} )  )\cong ( x_2^2,[-1,1]^2)$  is normal-crossing   of type $[0,\tau]$ with $\tau=0$. 
From $\mathbb{P}_{2,\ell} ={\bf N}(x_2^2 ,[-1,1]^2)$,
\begin{align*}
\text{$\delta_{\rm{for}}(\mathbb{P}_{2,\ell})=2$ and $\delta_{\rm{bac}}(\mathbb{P}_{2,\ell})=\infty$ with $k_{\rm{for}}(\mathbb{P}_{2,\ell})=1$ and $k_{\rm{bac}}(\mathbb{P}_{2,\ell})=1$. }
\end{align*} 
{\bf Case $(P,D_3)$}. On $D_3=D^{\rm{away}}_{\rm{circ}} \cap D_{\rm{para}} $, as  $P_{\rm{para}}(x)=(x_2-x_1^2)^2$  is  degenerate and   $P_{\rm{circ}}(x)=(x_1^2+x_2^2-1)^2$ is normal-crossing,  treat the parabola with $\phi_3(x_1,x_2)=(x_1, x_2+x_1^2)$ and take $P_3(x)=P\circ\phi_3(x) =x_2^2[x_1^2+(x_2+x_1^2)^2-1]^2$. Then
\begin{align*}
 \int  \psi(\lambda P (x))\psi_3(x ) dx &=\int  \psi(\lambda P_3 (x))\psi \circ\phi_3(x) dx \\
&=\int  \psi(\lambda P_3 (x))\psi \left(\frac{ x_2 }{ \epsilon  x_1^2} \right)\psi^c\left(\frac{ x_1^2+(x_2+x_1^2)^2-1}{\epsilon} \right)  dx \\
&= \int  \psi(\lambda P_3 (x))\psi_{\phi_3^{-1}(D_3)}(x)dx.
\end{align*}
From $\phi_3^{-1}(D_3)=\{ |x_2|\le  \epsilon x_1^2\ \text{and}\  |x_1^2+(x_2+x_1^2)^2-1|\ge \epsilon/2\}$, it follows that
 $ (P_3(x), \phi_3^{-1}(D_3)) \cong (x_2^2(x_1^4+x_1^8+1), D_{\{(-2,1)\}})$ with $\mathbb{P}_3={\bf N}(x_2^2(x_1^4+x_1^8+1),D_{\{(-2,1)\}}) $:
 \begin{align*}
\text{$\delta_{\rm{for}}(\mathbb{P}_3)=2$ and $\delta_{\rm{bac}}(\mathbb{P}_3)=4$ with $k_{\rm{for}}(\mathbb{P}_3)=k_{\rm{bac}}(\mathbb{P}_3)=1$. }
\end{align*} 
{\bf Case $(P,D_4)$}. There exists  $\mathfrak{c}=(c_1,c_2)\in S^1\cap \{c_2=c_1^2\}$ such that $D_4=D_{\rm{circ}} \cap D_{\rm{para}}\subset  \mathfrak{c}+[-h,h]^2 $. Thus  one can replace $\psi_4 $ supported on $D_4$ with  $ \psi \left(\frac{x_2- x_1^2}{\epsilon x_1 }  \right)\psi\left(\frac{x-\mathfrak{c}}{h}\right) $.  In view of $ a(x_1)  = - \frac{ c_1x_1 }{ c_2 }+O(|x_1|^2)  $ in (\ref{s33}),(\ref{0596}),  change coordinates via  $\phi_4^1(x)=(x_1+c_1,x_2+c_2+a(x_1))$ as
\begin{align*}
 \int_{\mathbb{R}^2}  \psi(\lambda P (x))\psi_4(x ) dx &\approx \int     \psi (\lambda P\circ\phi^1_4  (x))   \psi\left( \frac{x_2-k x_1 -x_1^2}{\epsilon(x_1+c_1)} \right)   \psi\left(  \frac{(x_1,x_2+a(x_1))}{h}\right)dx\\
 &  \approx \int     \psi (\lambda P\circ\phi^1_4  (x))     \psi\left(  \frac{(x_1,x_2+a(x_1))}{h}\right)dx
 \end{align*}
for       $k =2c_1+1/c_2$  and
$x$ supported in   $|x|\lesssim h \ll 1$. From $ P_{\rm{cir}}\circ\phi^1_4 (x)  \approx x_2^2$ and $P_{\rm{para}}\circ\phi^1_4(x)=(x_2+c_2+a(x_1)-(x_1+c_1)^2)^2\approx (x_2-(kx_1 +x_1^2))^2$,  
split the integral:
\begin{align*}
 \int     \psi (\lambda P^1_4 (x)) \left(\psi^c\left(\frac{x_2-kx_1}{\epsilon x_1}  \right) + \psi\left(\frac{x_2-kx_1}{\epsilon x_1}  \right) \right)      \psi\left(  \frac{(x_1,x_2+a(x_1))}{h}\right)dx.
\end{align*}
where $P^1_4=P\circ\phi^1_4$. The first part supported on 
  $D_4^1:=\{ |x_2-kx_1|\gtrsim |x_1|\ \text{and}\  |x|\ll1 \}$  corresponds to $(P^1_4, D_4^1)\cong ( |x_2|^2(|x_1|+|x_2|)^2, [-1,1]^2)$.  We next apply  another coordinate change via $\phi_4^2(x)=(x_1,x_2+kx_1)$ and $P^2_4=P^1_4\circ\phi^2_4$ for the second part:
\begin{align*} 
\int     \psi (\lambda P^2_4 (x))   \psi\left(\frac{x_2}{\epsilon x_1}  \right)     \psi\left(  \frac{(x_1,x_2+kx_1+a(x_1))}{h}\right)dx 
\end{align*}
supported on
  $D_4^2:=\{ |x_2| \ll |x_1|\ \text{and}\  |x|\ll1 \}$ so  that
 $(P^2_4, D_4^2)\cong ( |x_1|^2|x_2|^2, [-1,1]^2)$. 
Moreover, $P^\nu_4$ is of type $[0,\tau]$ with $\tau=0$ on $D^\nu_4$.   Take $\mathbb{P}^\nu_4:={\bf N}(P^\nu_4,[-1,1]^2)$. Then 
\begin{align*}
\text{$\delta_{\rm{for}}(\mathbb{P}_4^\nu)=2$ and $\delta_{\rm{bac}}(\mathbb{P}^\nu_4)=\infty$ with $k_{\rm{for}}(\mathbb{P}^\nu_4)=0$ and $k_{\rm{bac}}(\mathbb{P}^\nu_4)=1$ for $\nu=1,2$.}
\end{align*} 
{\bf Conclusion}. By applying the Main Theorem 3 with   $\cap_{i=1}^4(\mathbb{P}_i\cap \rm{cone}({\bf 1}))$ given by
$$[2,4] {\bf 1}=[4/3,4] \cap [2,\infty] \cap [2,4]\cap [2,\infty] {\bf 1}\ \text{and}\
0=\begin{cases} k_{\rm{for}}'=\min\{k_{\rm{for}}(\mathbb{P}_i): \delta_{\rm{for}}(\mathbb{P}_i)=2\}\\
k_{\rm{bac}}'=\min\{k_{\rm{bac}}(\mathbb{P}_i): \delta_{\rm{bac}}(\mathbb{P}_i)=4\}\end{cases}
$$
  to obtain (\ref{709}). Similarly, we have (\ref{710}).
   \end{proof}

     \subsection{Face-Nondegeneracy in Global Domains} \label{Sec99}
\begin{definition}
 Call
   $(P,D_B)$  {\bf face-nondegenerate} of type $[\sigma,\tau]$ if $\tau$ is minimal: \begin{align} \label{0246}
 \sum_{\sigma\le |\alpha|\le \tau}\big|\partial_{x}^{\alpha}P_{\mathbb{F}}\big|_{  (\mathbb{R}\setminus\{0\})^d}\ \text{are non-vanishing for all   faces  $\mathbb{F}$ of  $\mathbb{\bf N}(P,D_B)$.} 
 \end{align}
\end{definition}
In \cite{Gr1},  Greenblatt weakened the assumption (\ref{vc0}) of Varchenko  \cite{V} by restricting  the orders $\tau$  of  zeros of $P_{\mathbb{F}}$   less than   $\delta:= d({\bf N}(P, [-1,1]^d))$, which is equivalent to      the face-nondegeneracy in (\ref{0246})  of type    $\tau<\delta$ and $\sigma=1$.
 In a small neighborhood   $D  \cong D_{B}$ with $B=\{{\bf e}_\nu\}_{\nu=1}^d$, one can see that $(P,D)$  is normal-crossing of type $[\sigma,\tau]$  in  (\ref{sv3})  
  if and only if $(P,D_{B})$   is face-nondegenerate of type $[\sigma,\tau]$ in (\ref{0246}), which had already appeared  in  Theorem 1.5 of \cite{V1}.   
 This equivalence does not  always hold in a global domain $D\cong D_B$.  But, it does hold,  once  $D$ is away from a middle region $M_h:= \left\{x\in\mathbb{R}^d: h^{-1}\le |x_\nu|\le h\ \text{for all}\ \nu\in [d]\right\} $ for some $h\ge 1$.  
\begin{theorem}\label{prop44}
Let $D\cong D_B$. Then 
  $(P,D_B)$ is {\bf face-nondegenerate} of type $[\sigma,\tau]$ if and only if    $ (P,D\setminus M_h)$ for some  $h\ge 1$ is {\bf normal crossing} of  type $[\sigma,\tau]$.
  \end{theorem}
  We shall prove Theorem \ref{prop44} in Section \ref{Sec14}. 
Consequently,  one can   replace the {\bf normal-crossing} hypothesis of Main Theorem \ref{mainth4} by  {\bf face-nondegeneracy}:
$$\text{$(P\circ\phi_i,\phi^{-1}_i(D_i) )\cong (P_i,D_{B_i})$ are  {\bf face-nondegenerate}   of type $[0,\tau_i]$}$$
  after choosing the decomposition $D=\left(\bigcup_{i=1}^{M+1} D_i \right)\cup D_{\rm{nonsing}}$:
\begin{itemize}
\item $D_{i}=B_{\epsilon}({\bf c}_i)$ for ${\bf c}_i\in  P^{-1}(0)\cap [-h,h]^d$ so that
 $\bigcup_{i=1}^{M} D_{i}\supset  P^{-1}(0)\cap [-h,h]^d $
where
  $\phi^{-1}_i(D_i)=B_\epsilon({\bf 0})$   with  $\phi_i(x)=x+{\bf c}_i$ (or a further  coordinate change).
\item $D_{M+1}=\bigcup_{\nu=1}^d\{x\in D: |x_\nu|\ge h\}$ and $D_{\rm{nonsing}}=D\cap [-h,h]^d\cap (P^{-1}(0))^c.$ 
\end{itemize}
Here we need to choose $\epsilon, 1/h\ll 1$. See
  Examples \ref{ex46} through \ref{exboth}.

\section{Oriented Simplicial Dual Faces}\label{Sec13}
We shall prove Theorem \ref{lemir1}. We start with    Observations \ref{o6.1} and \ref{o6.2}.
\begin{ob}[Representation of Face and Dual face]\label{o6.1}
 Given    $\mathbb{P}$, there is $\Pi(\mathbb{P})$ such that $\mathbb{P}=\bigcap_{\pi\in \Pi(\mathbb{P})} \pi^+$  as in (\ref{jq820}).
For  $\mathbb{F}\in\mathcal{F}^k(\mathbb{P)}$,   set $\Pi(\mathbb{F}^+):=\{\pi_{\mathfrak{q},r} \in  \Pi(\mathbb{P}): \mathbb{F}\subset\pi_{\mathfrak{q},r}\}$. Then it holds that
  \begin{align}
\mathbb{F}=\bigcap_{\pi_{\mathfrak{q},r}\in \Pi(\mathbb{F}^+)}\pi_{\mathfrak{q},r}\cap \mathbb{P} \ \text{and}\  \mathbb{F}^{\vee}=\rm{cone}\big(\{\mathfrak{q}: \pi_{\mathfrak{q},r}\in \Pi(\mathbb{F}^+) \}\big)\label{39c}
  \end{align} 
  where  the dual face $\mathbb{F}^{\vee}$     is   a $(d-k)$ dimensional cone. 
   See   Propositions 4.1 and 4.2 of  \cite{Kim}. For further studies, we refer  \cite{CLS,Ful,O} for readers.
   \end{ob} 
   \begin{ob}[Dual face and Dual cone]\label{o6.2}
The dual face $\mathbb{F}^{\vee}$ in (\ref{39c}) is  the dual cone $( \mathbb{F}^+)^{\vee}$   of the polyhedron $\mathbb{F}^+=\bigcap_{\pi_{\mathfrak{q},r}\in \Pi(\mathbb{F}^+)}\pi_{\mathfrak{q},r}^+  (\text{conical extension of $\mathbb{F}$}).$  See the first picture of Figure \ref{octa}.  
  \end{ob}
   \begin{figure}
 \centerline{\includegraphics[width=13cm,height=8cm]{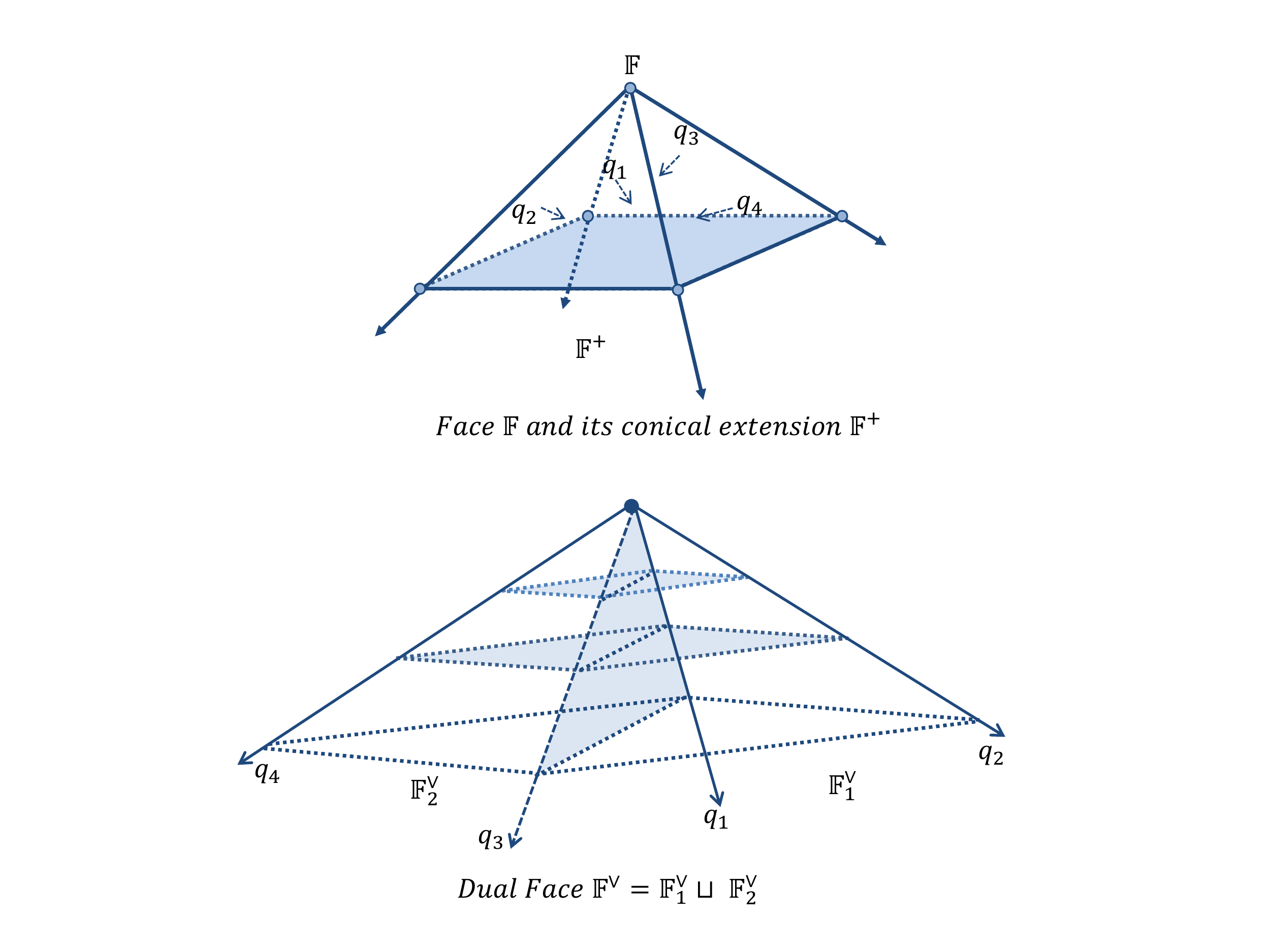}}
 \caption{
 The  polyhedron $\mathbb{P}$ has the  vertex $\mathbb{F}$ represented as 
 $\pi_{\mathfrak{q}_1}\cap\pi_{\mathfrak{q}_2}\cap \pi_{\mathfrak{q}_3}\cap \pi_{\mathfrak{q}_4}$. 
 Its dual face $\mathbb{F}^{\vee}=\rm{cone}(\mathfrak{q}_1,\mathfrak{q}_2,\mathfrak{q}_3,
 \mathfrak{q}_4) $ splits into two simplicial cones $\mathbb{F}_1^{\vee}= \rm{cone}\left(\mathfrak{q}_1,\mathfrak{q}_2,\mathfrak{q}_3\right)$ and $\mathbb{F}_2^{\vee} =\rm{cone}\left(\mathfrak{q}_1,\mathfrak{q}_3,\mathfrak{q}_4\right) $.  We can regard $\mathbb{F}_1=\pi_{\mathfrak{q}_1}\cap\pi_{\mathfrak{q}_2}\cap \pi_{\mathfrak{q}_3} $ and $\mathbb{F}_2=\pi_{\mathfrak{q}_1}\cap\pi_{\mathfrak{q}_3}\cap  \pi_{\mathfrak{q}_4}$ as   two faces, different from $\mathbb{F}$, because $\mathbb{F}_1,\mathbb{F}_2$ have the dual faces $\mathbb{F}_1^{\vee},\mathbb{F}_2^{\vee} $ different from   $\mathbb{F}^{\vee}$   while $\mathbb{F}_1$ and $\mathbb{F}_2$ themselves are same to $\mathbb{F}$ as  sets.  }\label{octa}
 \end{figure}   
  We  prove the two propositions making all dual faces $\mathbb{F}^{\vee}$ simplicial and oriented. \\

\noindent
 {\bf Simplicial duals}. See an idea of proof   in Figure \ref{octa} and the triangulization lemma.
 \begin{lemma}[Simplicial Cones]\label{lem5.6}
Every  $n$-dimensional   cone  $\mathbb{K}=\rm{cone}(\{\mathfrak{q}_i\}_{i=1}^m)$ in $V$ can be expressed as the union $\bigcup_{\ell=1}^L\mathbb{K}_\ell $ of  ess disjoint $n$-dimensional simplicial cones  $\mathbb{K}_\ell=\rm{cone}(\mathfrak{q}^{\ell}_1,\cdots,\mathfrak{q}^\ell_{n})$ where $\{\mathfrak{q}^{\ell}_1,\cdots,\mathfrak{q}^\ell_{n}\}\subset \{\mathfrak{q}_i\}_{i=1}^m$. Moreover, every  face of simplicial cone  is simplicial.
 \end{lemma}
  We omit its proof.
By   Lemma \ref{lem5.6}, we have a  simplical decomposition of $\mathbb{F}^{\vee}$.
 \begin{proposition}[Dual face Splits to Simplicial Cones]\label{facerep}  
  Recall $\mathbb{F}^{\vee}$    in (\ref{39c}) where
  $\mathbb{F}=\bigcap_{\pi_{\mathfrak{q},r}\in \Pi(\mathbb{F}^+)}\pi_{\mathfrak{q},r}   \in\mathcal{F}^{k}(\mathbb{P})$.  Then,  the dual face $\mathbb{F}^{\vee}$ splits
  \begin{align}
& 
 \mathbb{F}^{\vee} =\bigcup_{\ell=1}^L \mathbb{F}_\ell^{\vee}\ \text{with all  $\mathbb{F}_\ell^{\vee}$  ess  disjoint simplicial dual cones with}  \label{jq513} \\
&  \mathbb{F}_\ell:=\bigcap_{\pi_{\mathfrak{q},r}\in \Pi_\ell}\pi_{\mathfrak{q},r}   \cap\mathbb{P}\ \text{are same to $\mathbb{F}$ as a set where  $  \bigcup_{\ell=1}^L\Pi_\ell= \Pi(\mathbb{F}^+)$.}  \label{6.14}
 \end{align} 
 \end{proposition}
 \begin{proof}
Let $\Pi(\mathbb{F}^+)=\{\pi_{\mathfrak{q}_{i},r_i}\}_{i=1}^m$ in   (\ref{39c}).  Apply  the simplicial decomposition of the cone $\mathbb{K}$ of Lemma \ref{lem5.6} for  the  $(d-k)$-dimensional cone $ \mathbb{F}^{\vee}=\rm{cone}(\{\mathfrak{q}_i\}_{i=1}^m)$ as
\begin{align}\label{hh1}
\mathbb{F}^{\vee} =\bigcup_{\ell=1}^L \rm{cone}(\{\mathfrak{q}_{i(\ell)} \}_{i=1}^{d-k})\ \text{for $ \{\mathfrak{q}_{i(\ell)} \}_{i=1}^{d-k}\subset \{\mathfrak{q}_i\}_{i=1}^m$}
\end{align} 
 where $ \rm{cone}(\{\mathfrak{q}_{i(\ell)} \}_{i=1}^{d-k})$  are ess  disjoint simplicial cones. Then  for each $\ell$,   choose $$\Pi_\ell:=\{\pi_{\mathfrak{q}_{i(\ell)},r_{i(\ell)}} \}_{i=1}^{d-k}\subset \Pi(\mathbb{F}^+)\ \text{and}\ 
\begin{cases}
\mathbb{F}_\ell :=\bigcap_{i=1}^{d-k}\pi_{\mathfrak{q}_{i(\ell)},r_{i(\ell)}} \cap \mathbb{P} \\
 \mathbb{F}_\ell^+:=\bigcap_{i=1}^{d-k}\pi^+_{\mathfrak{q}_{i(\ell)},r_{i(\ell)}}. \end{cases}
 $$
   Then    (\ref{hh1}) with
 $   (\mathbb{F}_\ell)^{\vee}  =  \rm{cone}(\{\mathfrak{q}_{i(\ell)} \}_{i=1}^{d-k}) $ 
shows 
   (\ref{jq513}).
Next,    $\text{rank}(\{\mathfrak{q}_{i(\ell)}\}_{i=1}^{d-k})=d-k$ implies that    $ \mathbb{F}_\ell$ is a $k$-dimensional boundary object of $\mathbb{P}$ containing $  \bigcap_{i=1}^{m} \pi_{\mathfrak{q}_{i},r_{i }}$ which is $\mathbb{F} $.  Thus  $\mathbb{F}_\ell$ and $\mathbb{F}$ coincide, showing (\ref{6.14}).  See  Figure \ref{octa}.
 \end{proof}
\noindent
{\bf  Orientation}.    Call   $\mathbb{F}\in \mathcal{F}(\mathbb{P})$ and $\mathbb{F}^\vee$  forward or backward oriented     if $\mathbb{F}^{\vee}\subset\mathbb{P}^{\vee}_{\rm{for}}$ or $ \subset\mathbb{P}^{\vee}_{\rm{bac}}$.  
 In Figure \ref{forba89}, we demonstrate  $\mathbb{F}^\vee  =\mathbb{F}_{\rm{for}}^{\vee}\cup\mathbb{F}_{\rm{bac}}^{\vee}$ showing how to reset $ \mathcal{F}(\mathbb{P})$ as $\mathcal{F}(\mathbb{P}_{\rm{for}})\cup \mathcal{F}(\mathbb{P}_{\rm{bac}})$.  Indeed, one can switch the single vertex $\mathbb{F}=\pi_{\mathfrak{q}(1)}\cap \pi_{\mathfrak{q}(2)} $   with
the   two  newly formed vertices $\mathbb{F}_{\rm{for}}:=\pi_{\mathfrak{q}(1)}\cap\pi_{(-1,1)}$ of $\mathbb{P}_{\rm{for}}$ and $\mathbb{F}_{\rm{bac}}:=\pi_{(-1,1)}\cap\pi_{\mathfrak{q}(2)}$ of $\mathbb{P}_{\rm{bac}}$, which are same to the original   $\mathbb{F}$ as a set, but have  distinct duals $\mathbb{F}_{\rm{for}}^{\vee} \subset \mathbb{P}_{\rm{for}}^{\vee}$ and $\mathbb{F}_{\rm{bac}}^{\vee} \subset\mathbb{P}_{\rm{bac}}^{\vee}$. 
We state this intuition   as the proposition below.
 \begin{proposition}\label{prop110}[Dichotomy to  Oriented Faces]
Let $\mathbb{F}\in \mathcal{F}^{k}(\mathbb{P})$.   Suppose that 
\begin{align}
(\mathbb{F}^{\vee})^{\circ}\cap(\mathbb{P}_{\rm{for}}^{\vee})^{\circ}\ne \emptyset\ \text{and}\ (\mathbb{F}^{\vee})^{\circ}\cap(\mathbb{P}_{\rm{bac}}^{\vee})^{\circ}\ne \emptyset. \label{jq823}
\end{align}
Then we can split  the dual face $\mathbb{F}^\vee  =\mathbb{F}_{\rm{for}}^{\vee}\cup\mathbb{F}_{\rm{bac}}^{\vee}$ as
\begin{align}
 \begin{cases} \mathbb{F}_{\rm{for}}^{\vee}  : =\mathbb{F}^{\vee}\cap \mathbb{P}_{\rm{for}}^{\vee}=\rm{cone}(\{\mathfrak{p}_{i}\}_{i=1}^m) \\ 
   \mathbb{F}_{\rm{bac}}^{\vee}  : =\mathbb{F}^{\vee}\cap \mathbb{P}_{\rm{bac}}^{\vee}  =\rm{cone}(\{\mathfrak{q}_{j}\}_{j=1}^n) 
   \end{cases} 
 \label{jq821}
\end{align}
for $\text{dim}(\mathbb{F}_{\rm{for}}^{\vee}  )=\text{dim}(\mathbb{F}_{\rm{bac}}^{\vee}  )  =d-k$, which are dual faces of    
\begin{align}\label{jq822} 
    \begin{cases}
 \mathbb{F}_{\rm{for}}:=\bigcap    \pi_{\mathfrak{p}_i,r_i}\cap  \mathbb{P}\in\mathcal{F}^{k}(\mathbb{P}_{\rm{for}}) \ \text{with}\ \Pi(\mathbb{F}_{\rm{for}}^+)=\{ \pi_{\mathfrak{p}_i,r_i}  \}_{i=1}^m  \\
\mathbb{F}_{\rm{bac}}:= \bigcap  \pi_{\mathfrak{q}_j,s_j}\cap \mathbb{P}\in\mathcal{F}^{k}(\mathbb{P}_{\rm{bac}})\ \text{with}\  \Pi(\mathbb{F}_{\rm{bac}}^+)=\{\pi_{\mathfrak{q}_j,s_j} \}_{j=1}^n  
   \end{cases}
 \end{align} 
which are identical to $\mathbb{F}$ as  a set. Here $ \pi_{\mathfrak{p}_i,r_i}\in \overline{\Pi}(\mathbb{P}_{\rm{for}})$ and $\pi_{\mathfrak{q}_j,s_j}\in\overline{\Pi}(\mathbb{P}_{\rm{bac}})$. If $\mathbb{F}^{\vee}\subset \mathbb{P}_{\rm{for}}^{\vee}\ \text{or}\ \mathbb{P}_{\rm{bac}}^{\vee}$,  rewrite $\mathbb{F}$ as $\mathbb{F}_{\rm{for}}\in\mathcal{F}^{k}(\mathbb{P}_{\rm{for}}) $ or $\mathbb{F}_{\rm{bac}}\in\mathcal{F}^{k}(\mathbb{P}_{\rm{bac}}) $ respectively.

   \end{proposition}
\begin{proof}[Proof of Proposition \ref{prop110}]
 Since $\mathbb{F}\in \mathcal{F}^{k}$, 
 the cone $\mathbb{F}^{\vee}$ of  dim $d-k$    is imbedded in the $d-k$ dimensional subspace $U\subset V$.
From (\ref{jq823}),   both $(\mathbb{F}^{\vee})^{\circ}\cap(\mathbb{P}_{\rm{for}}^{\vee})^{\circ}$  and $ (\mathbb{F}^{\vee})^{\circ}\cap(\mathbb{P}_{\rm{bac}}^{\vee})^{\circ}$ 
contain a non-empty open sets  in $U$. 
Thus,    both $ \mathbb{F}^{\vee} \cap \mathbb{P}_{\rm{for}}^{\vee} $  and $  \mathbb{F}^{\vee}  \cap \mathbb{P}_{\rm{bac}}^{\vee} $  in   $U$    are  $d-k$ dimensional polyhedral cones.  
 Therefore  we have  (\ref{jq821}) such that $
\text{rank}(\{\mathfrak{p}_i\}_{i=1}^m)=d-k\ \text{and}\ \text{rank}(\{\mathfrak{q}_i\}_{i=1}^n)=d-k. 
$ 
Since   $\{\mathfrak{p}_i \}_{i=1}^m\subset\mathbb{F}^{\vee} \cap \mathbb{P}^{\vee}_{\rm{for}}\subset \mathbb{F}^{\vee}$ in (\ref{jq821}),   there are    $\pi_{\mathfrak{p}_1,r_1},\cdots,\pi_{\mathfrak{p}_m,r_m}\in    \overline{\Pi}(\mathbb{P})$ containing $\mathbb{F}$. Hence $\mathbb{F}_{\rm{for}}$, defined in (\ref{jq822}), is  at most $k$-dimensional object, containing    $\mathbb{F}$. 
 This  implies that   $\mathbb{F}_{\rm{for}}$  in (\ref{jq822}) coincides with   $\mathbb{F}$ as a set. Similarly,   $\mathbb{F}_{\rm{bac}}=\mathbb{F}$.  
 \end{proof}
  \noindent
 {\bf Decomposition by Oriented Simplicial Dual Faces}
Let    $\text{dim}(\mathbb{P}^{\vee})=d-k_0$.  Insert  $ \mathbb{F} \in\mathcal{F}(\mathbb{P}_{\rm{for}})\cup\mathcal{F}(\mathbb{P}_{\rm{bac}})$  in Proposition \ref{prop110} into    (\ref{07hh}) as
  $$\mathbb{P}^{\vee}    =\bigcup_{\mathbb{F} \in \mathcal{F}(\mathbb{P}_{\rm{for}})\cup \mathcal{F}(\mathbb{P}_{\rm{bac}})} \mathbb{F}^{\vee}   =\bigcup_{\mathbb{F} \in \mathcal{F}^{k_0}(\mathbb{P}_{\rm{for}})\cup \mathcal{F}^{k_0}(\mathbb{P}_{\rm{bac}})} \mathbb{F}^{\vee}.
$$   
Define $\mathcal{F}^{k}_{\rm{os}}(\mathbb{F}) =\{\mathbb{F}_\ell\}$  as the set of $k$-dimensional  newly-formed faces $\mathbb{F}_\ell$ in  (\ref{6.14})   having ess   disjoint simplicial  duals $\mathbb{F}_\ell^{\vee}$  in (\ref{jq513})  forming $  \bigcup_{ \mathbb{F}_\ell\in \mathcal{F}^{k}_{\rm{os}}(\mathbb{F}) } \mathbb{F}_\ell^{\vee}= \mathbb{F}^{\vee}$. By inserting $\mathcal{F}^{k}_{\rm{os}} :=\bigcup_{\mathbb{F}\in \mathcal{F}^{k}} \mathcal{F}^{k}_{\rm{os}}(\mathbb{F})$ and  $\mathcal{F}_{\rm{os}}:=\bigcup_{k\ge k_0}\mathcal{F}^{k}_{\rm{os}}$ into the above decomposition,
 \begin{align} 
 \mathbb{P}^{\vee} &=  \bigcup_{\mathbb{F} \in \mathcal{F}_{\rm{os}}(\mathbb{P}_{\rm{for}})\cup \mathcal{F}_{\rm{os}}(\mathbb{P}_{\rm{bac}})} \mathbb{F}^{\vee}   =\bigcup_{\mathbb{F} \in \mathcal{F}^{k_0}_{\rm{os}}(\mathbb{P}_{\rm{for}})\cup \mathcal{F}^{k_0}_{\rm{os}}(\mathbb{P}_{\rm{bac}})} \mathbb{F}^{\vee}.\label{jq415}
 \end{align}  
    We  can apply (\ref{jq415}) to (\ref{255}). Fix
 $\mathbb{P}={\bf N}(P,D_B)$ to prove Theorem \ref{lemir1}.
   \begin{proof} [Proof of Theorem \ref{lemir1}]
One can apply (\ref{jq415}) for   (\ref{255}) to have (\ref{s40})  where   $\{\mathfrak{q}_i\}_{i=1}^{d_0}\subset\mathbb{Q}^d$ is the set of  linearly independent vectors.  We next claim (\ref{asan1}).  First take $M_0$ as the product of the denominators of all entries in  $ \{\mathfrak{q}_i\}_{i=1}^{d_0}$. This implies the first inclusion of (\ref{asan1}).
For the $d\times d_0$   matrix $ A:=(\mathfrak{q}_1|\cdots|\mathfrak{q}_{d_0})$,    it holds 
 $ j\in \mathbb{F}^{\vee}\cap \mathbb{Z}^d\ \text{if and only if}\ j=\sum_{i=1}^{d_0}\alpha_i \mathfrak{q}_i=A(\alpha_i)\in \mathbb{Z}^d\ \text{for $\alpha_i\ge 0$}$. Since the first $d_0\times d_0$ sub-matrix $A_0$ of $A$ is non-singular, 
\begin{align}  \label{6.22}
 \mathbb{F}^{\vee}\cap &\mathbb{Z}^d\subset  \bigg\{ \sum_{i=1}^{d_0}\alpha_i \mathfrak{q}_i: (\alpha_i )_{i=1}^{d_0}\in A_0^{-1}(\mathbb{Z}^{d_0})\cap\mathbb{R}_+^{d_0}\bigg\}.    
  \end{align}
Take $M_1$   as   the product the denominators of all entries in $A_0^{-1}$    in (\ref{6.22}). Then $  A_0^{-1}(\mathbb{Z}^{d_0})\subset  \left(\frac{1}{M_1}\mathbb{Z}_+\right)^{d_0}$ in (\ref{6.22}) implies the second inclusion of (\ref{asan1}). \end{proof}  
 
 \section{Equivalence with Face-Nondegeneracy}\label{Sec14}

We prove Theorem \ref{prop44}. We need a notion of a neighborhood  of a dual face.

\subsection{Neighborhoods of Dual Faces}
Suppose  that  $D\cong D_B$ and $D$ is  away from the middle region   $M_h=\{x: 1/h\le |x_\nu|\le h\ \text{for all $\nu$}\}$ for $h\gg 1$.   Thus $D\subset\{x: |x_\nu|<1/h\ \text{or}\ |x_\nu|>h\ \text{for some $\nu$}\}$.  Take $h=2^{  d r^{d+100}}$ with $r\ge 1$. Then $$D\subset {\bf 2}^{-N_r^c}\ \text{where}\ N_r=B({\bf 0},d r^{d+100})$$    
Hence,   $
 D\subset  2^{- \rm{cone}^{\vee}(B)\cap  N_r^c+O(1)}$. Thus we shall work $\rm{cone}^{\vee}(B)\cap  N_r^c$ rather than
$\rm{cone}^{\vee}(B)$. One can split
by the faces in (\ref{jq415}), 
\begin{align}\label{a004}
 \rm{cone}(B)^{\vee}\cap N_r^c = \bigcup_{\mathbb{F}\in \mathcal{F}(\mathbb{P}) }   \mathbb{F}^{\vee}\cap N_r^c\ \text{where  $\mathcal{F}(\mathbb{P})=\mathcal{F}_{\rm{os}}(\mathbb{P}_{\rm{for}})\cup  \mathcal{F}_{\rm{os}}(\mathbb{P}_{\rm{bac}}).$} 
 \end{align}
  Next, consider   a neighborhood of a dual face $\mathbb{F}^{\vee}\cap N_r^c$. 
 \begin{definition}[Neighborhood of $\mathbb{F}^{\vee}$ in $\mathbb{G}^{\vee}$]\label{de49}
  Let  $\mathbb{F} \in  \mathcal{F}^k  $ with $\mathbb{G}\in \mathcal{F}^{k_0}  $ such that $\mathbb{G}\preceq\mathbb{F}$  in    (\ref{a004}).  Since $\mathbb{F}^{\vee},\mathbb{G}^{\vee}$ simplicial, one can take  such that 
\begin{align}
\text{$\rm{cone}( \{\mathfrak{q}_i\}_{i=1}^{d-k})=\mathbb{F}^{\vee}$  and $\rm{cone}( \{\mathfrak{q}_i\}_{i=1}^{d})=\mathbb{G}^{\vee}$ where $\mathfrak{q}_i\in \Pi(\mathbb{P})$}.
\end{align}
Denote $\rm{basis}(\mathbb{F}^{\vee})=\{\mathfrak{q}_i \}_{i=1}^{d-k}$ and
$\rm{basis}(\mathbb{G}^{\vee})=\{\mathfrak{q}_i \}_{i=1}^{d}$ to
 define  a neighborhood  $ \mathcal{N}^k_r(\mathbb{F}^{\vee}|\mathbb{G}^{\vee})$  of $\mathbb{F}^{\vee}\cap N_r^c$ in $\mathbb{G}^{\vee}\cap N_r^c$:
 \begin{align*}
 \bigg\{  \sum_{  \mathfrak{q}_i\in  \rm{basis}(\mathbb{F}^{\vee}) } \alpha_{i}\mathfrak{q}_i +\sum_{  \mathfrak{q}_j\in  \rm{basis}(\mathbb{G}^{\vee})\setminus \rm{basis}(\mathbb{F}^{\vee})} \alpha_{j}\mathfrak{q}_j \in N_r^c
 :  \alpha_i\ge r^{k+1}\ \text{and}\  0\le \alpha_j<r^k \bigg\}.
\end{align*}
The vector   $ j\in \mathcal{N}^k_r(\mathbb{F}^{\vee}|\mathbb{G}^{\vee})$   is the sum of the first term (main term) in $ \mathbb{F}^{\vee}$ of size $ \ge  d r^{d+99 }$   and the second term (error term) in $  \mathbb{G}^{\vee}$  of size $\le r^kd$. Therefore,
 $\mathcal{N}^k_r(\mathbb{F}^{\vee}|\mathbb{G}^{\vee})$ is  a perturbation of   the cone   $\mathbb{F}^{\vee} $ in $\mathbb{G}^{\vee}$ located   away from both   $\partial(\mathbb{F}^{\vee})$ and the origin. For simplicity,   write it as  $\mathcal{N}_r(\mathbb{F}^{\vee}|\mathbb{G}^{\vee})$. Notice that if $k=k_0$,  owing to $   \mathbb{F} =\mathbb{G}$, we can write $\mathcal{N}_r^k(\mathbb{F}^{\vee}|\mathbb{G}^{\vee})$ as
\begin{align} \label{26}
 \mathcal{N}_r^{k_0}(\mathbb{G}^{\vee}|\mathbb{G}^{\vee})&:=\bigg\{   \sum_{  \mathfrak{q}_i\in  \rm{basis}(\mathbb{G}^{\vee})  } \alpha_{i}\mathfrak{q}_i  \in N_r^c
 :  \alpha_i\ge r^{k_0+1} \bigg\}.
\end{align}
\end{definition}

\begin{lemma}\label{lemkim1}  [Strict Dual Face Decomposition]
Let $\mathbb{P}$ be a polyhedron in $\mathbb{R}^d$ and   let $ \rm{dim}(\mathbb{P}^{\vee})=d-k_0$ and $\mathcal{F}=\mathcal{F}^{\rm{os}}$ in (\ref{jq415})  and (\ref{a004}). Then,  \begin{eqnarray}\label{j101}
 \rm{cone}(B)^{\vee}\cap N_r^c  = \bigcup_{\mathbb{G}\in \mathcal{F}^{k_0}(\mathbb{P})}\bigcup_{\{\mathbb{F}\in  \mathcal{F}:  \mathbb{G}\preceq \mathbb{F}\}} \mathcal{N}_r(\mathbb{F}^{\vee}|\mathbb{G}^{\vee}).
\end{eqnarray} 
\end{lemma}

 \begin{proof}[Proof of Lemma \ref{lemkim1}]
We  prove $k_0=0$.  Note $\supset$ is true from $\mathcal{N}_r(\mathbb{F}^{\vee}|\mathbb{G}^{\vee})  \subset
 \rm{cone}(B)^{\vee}$. 
To claim $\subset$, 
let $\mathfrak{p}\in   \rm{cone}(B)^{\vee}\cap N_r^c $. By (\ref{a004}), find $\mathbb{G}\in \mathcal{F}^0(\mathbb{P})$: 
\begin{align*} 
\mathfrak{p}=\alpha_1\mathfrak{q}_1+\cdots+\alpha_d\mathfrak{q}_d\in \bigcup_{\mathbb{F}\in \{\mathbb{F}\in  \mathcal{F}:  \mathbb{G}\preceq \mathbb{F}\}}   \mathbb{F}^{\vee}\cap N_r^c\ \text{with}\   \rm{basis}(\mathbb{G}^{\vee})=\{\mathfrak{q}_j\}_{j=1}^d.
\end{align*}
It suffices to show that $\mathfrak{p}\in  \mathcal{N}_r^k(\mathbb{F}^{\vee}|\mathbb{G}^{\vee})$ for some $\mathbb{F}^{\vee}\preceq\mathbb{G}^{\vee}$ with $0\le k\le d-1$. We can assume that $0\le \alpha_1\le \cdots\le \alpha_d$ above.
Set $d+1$ number of disjoint  intervals
\begin{align*}
\text{ $I_k:=[r^k,r^{k+1})$ where $k=1,\cdots,d-1$, and $I_d:=[r^d,\infty)$ and  $I_0:=[0,r)$. }
\end{align*}
\begin{itemize}
\item
Observe $\alpha_d\in I_d$ because $|\mathfrak{p}|\ge dr^{d+100}$ for $\mathfrak{p}\in N_r^c=\{j|j|\ge dr^{d+100}\}$.
\item
Next $\alpha_{d-1}\in I_{d-1}\cup I_{d}$. If not, $\alpha_{d-1}\in I_0\cup \cdots\cup I_{d-2}$, namely, $\alpha_{d-1}<r^{d-1}$ and $\alpha_d\ge r^d$, leading $\mathfrak{p}\in \mathcal{N}^{d-1}_r(\mathbb{F}^{\vee}|\mathbb{G}^{\vee})$ for $ \mathbb{F}^{\vee} =\rm{cone}(\{\mathfrak{q}_d\})$. 
\item
Next $\alpha_{d-2}\in I_{d-2}\cup I_{d-1}\cup I_d$. If not, $\alpha_{d-2}\in I_0\cup \cdots\cup I_{d-3}$,  namely, $\alpha_{d-2}<r^{d-2}$ with $ \alpha_{d-1}\ge r^{d-1}$, leading $\mathfrak{p}\in \mathcal{N}_r(\mathbb{F}^{\vee}|\mathbb{G}^{\vee})$ with $ \mathbb{F}^{\vee} =\rm{cone}(\{\mathfrak{q}_d,\mathfrak{q}_{d-1}\})$.
\item
   Repeat until  $\alpha_1\in I_1\cup\cdots\cup I_d$. 
 So $\alpha_d\ge \cdots\ge\alpha_1\ge r^1$ and $\mathfrak{p}\in  \mathcal{N}^0_r(\mathbb{G}^{\vee}|\mathbb{G}^{\vee})$.
\end{itemize}
Therefore, we are done with $\supset$ in (\ref{j101}). 
\end{proof}
  \begin{proposition}\label{lemir11}[Strict Dual Face Decomposition] 
Let $P$ be a polynomial and $B\subset \mathbb{Q}^d$ with $\text{dim}(\rm{cone}^{\vee}(B))= d-k_0$. Then, one can decompose   $ \int\psi(\lambda P(x))\psi_{D_B}(x)dx$:
 \begin{align*}  
 \sum_{\mathbb{G}\in \mathcal{F}^{k_0}(\mathbb{P})  }
 \sum_{\mathbb{F}\in \{\mathbb{F}\in  \mathcal{F}:  \mathbb{G}\preceq \mathbb{F}\}}
   \sum_{j\in\mathcal{N}_r( \mathbb{F}^{\vee}|\mathbb{G}^{\vee})\cap\mathbb{Z}^d} \int  \psi(\lambda P(x)) \psi_{D_{B}}(x)\chi\left(\frac{x}{{\bf 2}^{-j}}\right) dx
    \end{align*} 
where     $ \mathbb{F}^{\vee}$ are oriented  simplicial cones of the form $\rm{cone}(\mathfrak{q}_1,\cdots,\mathfrak{q}_{d-k})$ with $k\ge k_0$ and $\mathfrak{q}_i\in \mathbb{Q}^d\cap \{1/2\le |\mathfrak{q}|\le  1\}$  having the  rational coordinates as in (\ref{asan1}). Moreover,  if  $\mathbb{F}$ is of dimension $k$, then  there exists $r_P>0$ such that for $r>r_P$,
 \begin{align}\label{624}
   2^{Cr^{k+1}} \left(  \sum_{\mathfrak{n}\in\Lambda(P)\setminus \mathbb{F} } |c_{\mathfrak{n}} |2^{-j\cdot \mathfrak{n}}\right) \le  |c_{\mathfrak{m}} |2^{-j\cdot \mathfrak{m}}   \ \text{  $\forall \mathfrak{m} \in \mathbb{F} $ and $j \in
\mathcal{N}^k_r(\mathbb{F}^{\vee}|\mathbb{G}^{\vee})$ }
\end{align}
that is a general version of (\ref{jq524}).    If  $\sum_{\sigma\le |\alpha|\le \tau} |\partial_x^{\alpha}[P_{\mathbb{F}}]_{(\mathbb{R}\setminus\{0\})^d}|$ is non-vanishing,   then   there is $C>0$ and  $ r_P$ such that for   $x\sim 2^{-j}$ and $r>r_P$ as  in (\ref{624}),
  \begin{align}\label{a635}
\frac{2^{-j\cdot \mathfrak{m}} }{2^{6Cr^k}}&\le\sum_{\sigma\le |\alpha|\le \tau}|x^{\alpha} \partial_x^{\alpha}P(x)|\approx \sum_{\sigma\le |\alpha|\le \tau}|x^{\alpha} \partial_x^{\alpha}P_{\mathbb{F}}(x)|\le \frac{2^{-j\cdot \mathfrak{m}} }{ 2^{-6Cr^k}}.
\end{align}
 
 \end{proposition}
 The decomposition of the integral follows from the application of (\ref{j101}).

 \subsection{Proof of (\ref{624}) and (\ref{a635})}   To show (\ref{624}), we need the following lemmas and the definition of some constants involving the coefficients of $P$.

\begin{lemma}[\L ojasiewicz]\label{loja}
Let $U\subset\mathbb{R}^d$ be an open set containing a compact set $K$. Suppose $g$ and $G$ are real analytic functions (polynomials) in  $U$  such that
$$\{u\in U: g(u)=0\}\subset\{u\in U: G(u)=0\}$$
Then there is  constants $\mu,C>0$ such that 
$$|g(u)|\ge C|G(u)|^\mu\ \text{for all}\ u\in K.$$
\end{lemma}
\begin{proof}
See its proof in \cite{G2} and  \cite{L}.
\end{proof}

 \begin{lemma}\label{lo1}
Suppose that $Q:\mathbb{R}^d\rightarrow \mathbb{R}$ is a  polynomial, non-vanishing on $ (\mathbb{R}\setminus \{0\})^d$. Then there are constants  $\gamma,B>0$ independent of $x$ such that, 
\begin{align}\label{0kk}
  |Q(x)|\ge  B\min\{|x|^{-\gamma}, |x_1\cdots x_d|^{\gamma}\}. \end{align} 
  \end{lemma} 
  
\begin{proof}[Proof of Lemma \ref{lo1}]
Set
$K_i=\{u:  |u_1|,\cdots,|u_{d}|\le 1\ \text{and}\ |u_i|\le 1/100 \}$ and  
 \begin{align}\label{9jj}
g_i(u):=u_i^n Q\left(\frac{u_1}{u_i},\cdots,\frac{u_{i-1}}{u_i}, \frac{1}{u_i},\frac{u_{i+1}}{u_i},\cdots,\frac{u_d}{u_i}\right)\ \text{for $n=\text{deg}(Q)$}.
\end{align}
 Then $g$ is a polynomial, because $u_i^n$ cancels the all $u_i^{-s}$ with $0\le s\le n$ arising from $Q(\cdot)$.  Since $Q$ is non-vanishing in $(\mathbb{R}\setminus \{0\})^d$, we can observe that for $U:=\mathbb{R}^d$,
$$\{u\in U:g_i(u)=0\}\subset\{u\in U: u_1\cdots u_d=0\}.$$
 Then, one can apply  Lemma \ref{loja} for $G(u)=u_1\cdots u_d$ to  have $C,\mu_i\ge 1$:
$$|g_i(u_1,\cdots,u_d)|\ge C |u_1\cdots u_d|^{\mu_i}    \ \text{for all $u\in K_i$}.$$
Set $W_i =\{x:|x_i|\ge |x_1|,\cdots,|x_{d}|\ \text{and}\ |x_i|\ge 100\} $
and   $\Phi_i:W_i \rightarrow \Phi(W_i)=K_i$ by   
$$ \Phi(x):=\left(\frac{x_1}{x_i},\cdots,\frac{x_{i-1}}{x_i}, \frac{1}{x_i},\frac{x_{i+1}}{x_i},\cdots,\frac{x_d}{x_i}\right)=(u_1,\cdots,u_d). $$ 
  This coordinate change in (\ref{9jj}) implies that
\begin{align*} 
| x_i^{-n} Q(x_1,\cdots,x_d)|=|g_i(u)|\ge C |u_1\cdots u_d|^{\mu_i}  \ \text{for $x\in W_i$}.
\end{align*}
This with $100\le |x_i|\le |x| \le d|x_i|$ in $W_i$ yields that for   $ \nu=\max\{\mu_i (d+1)-n\},$ 
$$|Q(x)|\ge (x_i)^n (x_i)^{-\mu_i (d+1)}  |x_1\cdots x_d|^{\mu_i}\ge C |x |^{-\nu }  |x_1\cdots x_d|^{\nu} \ \text{if $x\in W_i $}$$ 
which  holds true for all $i\in [d]$. Hence
\begin{align}\label{01j}
|Q(x)|\ge C\min\{ |x|^{-2\nu}, |x_1\cdots x_d|^{2\nu}\}\ \text{for $x\in \bigcup_{i=1}^d W_i \supset\{x:|x|\ge 100d \}$}.
\end{align}
On the other hand, by applying Lemma \ref{loja} again for $\{x\in \mathbb{R}^d: Q(x)=0\}\subset \{x\in \mathbb{R}^d: x_1\cdots x_d=0\}$ due to non-vanishing condition of $Q$ on $(\mathbb{R}\setminus\{0\})^d$,
 $$|Q(x)|\ge C|x_1\cdots x_d|^{\mu}\ \text{on the compact set} \  K:=\{x: |x|\le 100d\}.$$
By this with (\ref{01j}),  we obtain   (\ref{0kk}) for $\gamma=\max\{2\nu, \mu\}$. \end{proof}

\begin{lemma}
Let $Q$ be a non-vanishing polynomial on $(\mathbb{R}\setminus\{0\})^d$  and let  $U_{2^{100dr^k}}:=\{y:2^{-100dr^k}\le |y_{\nu}|\le 2^{100dr^k}\ \text{for}\ \nu=1,\cdots,d\}$ where $r\ge 1$. Then  there is $b>0$ depending on $Q$ such that
\begin{align} \label{12.13}
2^{-b(Q)r^k}\le |Q(y)|\le 2^{b(Q)r^k}\ \text{for all $y\in U_{2^{100dr}}$}. 
\end{align}
\end{lemma}
\begin{proof}
It follws from Lemma \ref{lo1}.
\end{proof}

\begin{definition}[Constants Associated with $P$]\label{def217} Let $\mathbb{P}={\bf N}(P,D_B)$. Given  $P(x)=\sum_{\mathfrak{m}\in \Lambda(P)}c_{\mathfrak{m}}x^{\mathfrak{m}}$,  we define  the maximal ratio  of coefficients of $P$ as
  \begin{align*}
  C_P& :=\frac{\sum_{\mathfrak{n}\in \Lambda(P)} |c_{\mathfrak{n}}|}{\min\{1,|c_{\mathfrak{n}}|:\mathfrak{n}\in \Lambda(P) \}}.
  \end{align*}
From (\ref{12.13}), we take $b=\max\{b(P_{\mathbb{F}}):\ (P_{\mathbb{F}})_{(\mathbb{R}\setminus\{0\})^d}\ \text{non-vanishing}\ \mathbb{F}\in \mathcal{F}(\mathbb{P})\}$.
Set the two constants as
  \begin{align*}
 H&:=  (b+10d+\text{deg}(P) )^{10}   \sum_{ \mathfrak{m},\mathfrak{n}\in\Lambda(P) }| \mathfrak{n}-\mathfrak{m}|, \\   
L&  :=\min\left\{1,(\mathfrak{n}-\mathfrak{m})\cdot \mathfrak{q}: \mathfrak{n}\in \Lambda(P)\setminus \pi_{\mathfrak{q}}\ \text{and}\ \mathfrak{m}\in \Lambda(P)\cap \pi_{\mathfrak{q}} \right\}_{ \mathfrak{q} \in\mathbb{P}^{\vee} }.\nonumber
 \end{align*}
where $ 1/2\le|\mathfrak{q}|\le 1$. Then  $L>0$ because $(\mathfrak{n}-\mathfrak{m})\cdot \mathfrak{q}>0$  for $\mathfrak{n}\in \pi_{\mathfrak{q}}^+\setminus \pi_{\mathfrak{q}} \ \text{and}\ \mathfrak{m}\in  \pi_{\mathfrak{q}}$.
\end{definition}
 
 \begin{proof}[Proof of (\ref{624})] \label{lem13} 
Let $ r_P:=\max\{ C_{P}/L,  H/L  \}$ and    $j \in
\mathcal{N}^k_r(\mathbb{F}^{\vee}|\mathbb{G})$  for  $0\le k\le d-1$. Then,  we claim (\ref{624}) for $C=\frac{8L }{10}$. 
 By Definition \ref{de49}, we can write $j$ as
\begin{align}\label{ko33}
 \sum_{\mathfrak{q}_i\in \rm{basis}(\mathbb{F}^{\vee})}\alpha_{i}\mathfrak{q_i} +\sum_{\mathfrak{q}_{\ell}\in \rm{basis}(\mathbb{G}^{\vee})\setminus \rm{basis}(\mathbb{F}^{\vee})}\alpha_{\ell}\mathfrak{q}_{\ell}:  0\le \alpha_{\ell} <r^{k}\ \text{and}\  r^{k+1}\le \alpha_{i}
\end{align}
where the second sum is zero if $k=0$. Let $\mathfrak{m},\tilde{\mathfrak{m}}\in \mathbb{F} $.
   As
  $ (\mathfrak{m}-\tilde{\mathfrak{m}})\cdot \mathfrak{q}_i =0$  for  $\mathfrak{q}_i\in \rm{basis}(\mathbb{F}^{\vee})$, we have   \begin{align}\label{65t}
|j\cdot (\mathfrak{m}-\tilde{\mathfrak{m}})|\le \left|\sum_{\mathfrak{q}_\ell\in   \rm{basis}(\mathbb{G}^{\vee})\setminus \rm{basis}(\mathbb{F}^{\vee})}\alpha_{\ell}\mathfrak{q}_\ell\cdot  (\mathfrak{m}-\tilde{\mathfrak{m}})\right|\le r^{k}H. \end{align}  
Let $\mathfrak{m}\in  \mathbb{F}=\bigcap_{\mathfrak{q}_i\in \rm{basis}(\mathbb{F}^{\vee})} \pi_{\mathfrak{q}_i,r_i}\cap \mathbb{P}$ and $\mathfrak{n}\in
\Lambda(P)\setminus \mathbb{F}$. Then  $\mathfrak{m}\in \pi_{\mathfrak{q}_i}$ and $\mathfrak{n}\in \pi_{\mathfrak{q}_i}^+$ for all $\mathfrak{q}_i\in \rm{basis}(\mathbb{F}^{\vee})$, whereas   $\mathfrak{n}\in (\pi_{\mathfrak{q}_{s}}^+)^{\circ}$ for some $\mathfrak{q}_s\in \rm{basis}(\mathbb{F}^{\vee})$. Thus
\begin{align*} 
\text{ $\mathfrak{q}_{s}  \cdot(\mathfrak{n}-\mathfrak{m})>0$ \ \text{and}\ \  $ \mathfrak{q}_i \cdot(\mathfrak{n}-\mathfrak{m})\ge 0$ for all $\mathfrak{q}_i\in \rm{basis}(\mathbb{F}^{\vee})$. }
\end{align*}
From  this with the constants in Definition \ref{def217}, $\alpha_{i}\ge r^{k+1}$ and $\alpha_{\ell}<r^k$ in (\ref{ko33}),
$$\sum_{\mathfrak{q}_i\in \rm{basis}(\mathbb{F}^{\vee})} \alpha_{i}\mathfrak{q}_i \cdot(\mathfrak{n}-\mathfrak{m}) \ge  r^{k+1} L\ \text{and}\   \big|\sum_{\mathfrak{q}_\ell\in  \rm{basis}(\mathbb{G}^{\vee})\setminus \rm{basis}(\mathbb{F}^{\vee})}\alpha_\ell\mathfrak{q}_\ell\cdot  (\mathfrak{n}-\mathfrak{m})\big| < r^k H. $$
This with (\ref{65t}) implies that $
j\cdot (\mathfrak{n}-\mathfrak{m})   \ge \frac{9r^{k+1}L }{10} 
$
and  $2^{r^{k+1}L /10}> r^{k+1}L  \ge C_{P}$ due to $r>r_P$. Therefore,
$$
2^{-j\cdot \mathfrak{m}}\ge  2^{r^{k+1}\frac{9L}{10}}  2^{-j\cdot \mathfrak{n}}\ge 2^{r^{k+1}\frac{8L}{10}}  2^{-j\cdot \mathfrak{n}} C_{P}.$$  
This together with $C_{P}$ defined in Definition \ref{def217} yields that
\begin{align*} 
2^{-j\cdot \mathfrak{m}}\ge    \frac{  2^{(r^{k+1})\frac{8L}{10}}  \sum_{\mathfrak{n}\in
\Lambda(P)\setminus \mathbb{F}} |c_{\mathfrak{n}}|2^{-j\cdot \mathfrak{n}}  }{\min\{|c_{\mathfrak{m}}|:\mathfrak{m}\in \Lambda(P)\cap \mathbb{F}\}} 
\end{align*}
which yields (\ref{624}). 
\end{proof}  
 
 \begin{proof}[Proof of (\ref{a635})]
 We show (\ref{a635}) for $\sigma=\tau=0$.
Let
$j\in \mathcal{N}^k_r(\mathbb{F}^{\vee}|\mathbb{G})$  in Definition \ref{de49}.
Then
  $j=\mathfrak{q}+\mathfrak{u}$ for $\mathfrak{q}\in \mathbb{F}^{\vee}$ with $|\mathfrak{q}|\ge dr^{d+99}$ and $\mathfrak{u}\in \mathbb{G}^{\vee}$ with $|\mathfrak{u}|\le r^kd$.      Let $x\sim 2^{-{\bf 0}}$. Then
\begin{align}\label{488n}
P({\bf 2}^{-j}x)= P_{\mathbb{F}}({\bf 2}^{-j}x)+\sum_{\mathfrak{n}\in \Lambda(P)\setminus\mathbb{F}} 
2^{-j\cdot \mathfrak{n}} c_{\mathfrak{n}} x^{\mathfrak{n}}\end{align}
with $ P_{\mathbb{F}}({\bf 2}^{-j}x)  =   2^{- \mathfrak{q}\cdot \mathfrak{m}} P_{\mathbb{F}}\left(2^{-\mathfrak{u}}x \right).$  Note that $2^{-(d+1)r^k}\le |2^{-\mathfrak{u}_\nu}x_\nu| \le     2^{|\mathfrak{u}|+1} \le   2^{(d+1)r^k} $    where $ |\mathfrak{u}|\le r^kd $  in the above. Thus, we use (\ref{12.13}) to    obtain that 
 $$    2^{-br^k} \le   |P_{\mathbb{F}}\left(2^{-\mathfrak{u}}x \right)|\le 2^{br^k}. $$ 
This multiplied by $2^{-\mathfrak{q}\cdot \mathfrak{m}}= 2^{-j\cdot \mathfrak{m}} 2^{\mathfrak{u}\cdot \mathfrak{m}}\in    2^{-j\cdot \mathfrak{m}} [ 2^{-\deg(P) dr^k}, 2^{\deg(P) dr^k}  ]$  
leads that there exists  $C= \text{deg}(P) d(b+1) >0$ such that  for $r>r_P$,
\begin{align*} 
 2^{-j\cdot \mathfrak{m}}  2^{-3 Cr^k}\le  2^{-\mathfrak{q}\cdot \mathfrak{m}}  |P_{\mathbb{F}}\left({\bf 2}^{-\mathfrak{u}}x \right)|    \le     2^{3Cr^k}  2^{-j\cdot \mathfrak{m}}.
  \end{align*}
  Thus,   $|P_{\mathbb{F}}({\bf 2}^{-j}x)  |= 2^{-\mathfrak{q}\cdot \mathfrak{m}}  |P_{\mathbb{F}}\left({\bf 2}^{-\mathfrak{u}}x \right)|   $   in (\ref{488n}) with  (\ref{624})  implies (\ref{a635}) for $\sigma=\tau=0$.  To consider the general case,  replace $|P_{\mathbb{F}}(2^{-j}x)|$ and $P(2^{-j}x)|$  by $$\sum_{\sigma\le |\alpha|\le \tau}|(2^{-j\cdot \alpha}x^{\alpha} )\partial_x^{\alpha}P_{\mathbb{F}}(2^{-j}x)|\ \text{and}\ \sum_{\sigma\le |\alpha|\le \tau}|2^{-j\cdot \alpha}x^{\alpha} \partial_x^{\alpha}P(2^{-j}x)|$$
  where
  $$\sum_{\sigma\le |\alpha|\le \tau}|(2^{-j\cdot \alpha}x^{\alpha} )\partial_x^{\alpha}P_{\mathbb{F}}(2^{-j}x)|= 2^{- \mathfrak{q}\cdot \mathfrak{m}} \sum_{\sigma\le |\alpha|\le \tau}|(2^{-\mathfrak{u}\cdot \alpha}x^{\alpha} )\partial_x^{\alpha}P_{\mathbb{F}}(2^{-\mathfrak{u}}x)|.$$
Here $2^{-\mathfrak{q}\cdot \mathfrak{m}}=     2^{-j\cdot \mathfrak{m}} [ 2^{-\deg(P) dr^k}, 2^{\deg(P) dr^k}  ]$  with the non-vanishing condition 
  $$ 2^{-b' r^k} \le  \sum_{\sigma\le |\alpha|\le \tau}|(2^{-\mathfrak{u}\cdot \alpha}x^{\alpha} )\partial_x^{\alpha}P_{\mathbb{F}}(2^{-\mathfrak{u}}x)|\le 2^{b' r^k} $$
  leads (\ref{a635}) for $P_{\mathbb{F}}$. Finally, this   with  the difference  $ \sum_{\sigma\le |\alpha|\le \tau}|(2^{-j\cdot \alpha}x^{\alpha} )\partial_x^{\alpha}(P-P_{\mathbb{F}})(2^{-j}x)| \lesssim \sum_{\mathfrak{n}\in \Lambda(P)\setminus\mathbb{F}} 
2^{-j\cdot \mathfrak{n}} \lesssim 2^{-Cr{k+1}}$ gives  the desired estimate for $P$.
\end{proof}

 \subsection{Proof of Theorem \ref{prop44} (Normal-Crossing $\Leftrightarrow$ Face-nondegeneracy)}  \label{Sec10}

 \begin{proof}[Normal Crossing $\Rightarrow$ Face-Nondegeneracy]  
Suppose that there exists $h\ge 1$ satisfying  (\ref{sv3}) on $D=D\cap M_h^c$.  Let $\mathbb{F}$ be a face with  $\mathbb{F}^{\vee}=\rm{cone}(\{\mathfrak{q}_i\}_{i=1}^{s})$ and let    $y\in (\mathbb{R}\setminus\{0\})^d$.
We claim that $ \sum_{\sigma\le |\alpha|\le \tau}| (\partial_x^{\alpha} P_{\mathbb{F}}) (y)|  \ne 0$.
Choose $ \mathfrak{q}:=(\mathfrak{q}_1+\cdots+\mathfrak{q}_s) \in
 (\mathbb{F}^{\vee})^{\circ}$. Then there is $\rho>0$ such that  for all $r>0$,
 \begin{align}\label{004}
 \text{
  $  2^{-r\mathfrak{q}\cdot\mathfrak{n}}\le 2^{-r\rho} 2^{-r\mathfrak{q}\cdot \mathfrak{m}}   $ for all  $\mathfrak{m}\in \Lambda(P)\cap \mathbb{F}\ \text{and}\ \mathfrak{n}\in \Lambda(P)\setminus \mathbb{F}$.} 
 \end{align}
 Take $r\gg 1$ and  $\mathfrak{q}=(q_1,\cdots,q_d)$ above so that
  $$ x:={\bf 2}^{-r\mathfrak{q}} y=(2^{-rq_1}y_1,\cdots,2^{-rq_d}y_d) \in  D_B\cap M_h^c\subset D, $$because $y_i\ne 0$ for all $i$ with $ \mathfrak{q}   \in
 (\mathbb{F}^{\vee})^{\circ}$.   Thus, by (\ref{004}) and (\ref{sv3}),  for $\mathfrak{m}\in \Lambda(P)\cap\mathbb{F}$,
\begin{align*}
2^{-r\mathfrak{q}\cdot \mathfrak{m}}|y^{\mathfrak{m}}|= |x^{\mathfrak{m}}|&\lesssim \sum_{\sigma \le |\alpha|\le \tau} |x^{\alpha}\partial^{\alpha}_xP(x)|   \lesssim \sum_{\sigma\le |\alpha|\le \tau}  |[x^{\alpha}\partial^{\alpha}_x P]_{\mathbb{F}}(x)|\big|  +\sum_{\mathfrak{n}\in \Lambda(P)\setminus\mathbb{F}}  | x^{\mathfrak{n}}|\\
&\le  2^{-r\mathfrak{q}\cdot \mathfrak{m}}\sum_{\sigma\le |\alpha|\le \tau} |y^{\alpha} \partial^{\alpha}_y P_{\mathbb{F}}(y_1,\cdots,y_{d})|+2^{-r\mathfrak{q}\cdot \mathfrak{m}}2^{-r \rho }  \sum_{\mathfrak{n}\in \Lambda(P)\setminus\mathbb{F}}  |y^{\mathfrak{n}}|.
\end{align*} 
Divide by $2^{-r\mathfrak{q}\cdot \mathfrak{m}}$ and take
  $r\rightarrow\infty$. Then
$  |y^{\mathfrak{m}}|\lesssim\sum_{\sigma\le |\alpha|\le \tau} |y^{\alpha} \partial^{\alpha}_y P_{\mathbb{F}}(y_1,\cdots,y_{d})|. $
From this with $y^{\mathfrak{m}}\ne 0$ for $y\in (\mathbb{R}\setminus\{0\})^d$, we obtain  $\sum_{\sigma\le |\alpha|\le \tau}| \partial_y^{\alpha} P_{\mathbb{F}} (y)|  \ne 0$.   
 \end{proof}

 \begin{proof}[Face-Nondegenercay $\Rightarrow$ Normal Crossing]
 
   Choose   $  h :=  2^{-d r^{d+100}} $ where $r=r_P$   in Lemma \ref{lem13}.
    Let $x\in D \subset 2^{-\rm{cone}^{\vee}(B)+O(1)}\cap M_h^c$. Then from (\ref{j101}), we can express $x={\bf 2}^{-j}y$ for some $y\sim 2^{-{\bf 0}}$ and $j\in  \mathcal{N}_r(\mathbb{F}^{\vee}|\mathbb{G})$ in Definition \ref{de49}. Thus, the  non-vanishing condition of $P_{\mathbb{F}}|_{(\mathbb{R}\setminus\{0\})^d}(x)$ implies (\ref{a635}) with $2^{j\cdot \mathfrak{m}}=|x^{\mathfrak{m}}|$.  \end{proof}

\end{document}